\documentclass[11pt,a4paper]{article}
\usepackage{geometry} 
\usepackage{amsmath,amsfonts,amsthm,amssymb} 
\usepackage[dvipsnames]{xcolor}  
\usepackage{enumitem} 
\usepackage{graphicx, float, subfig, overpic, comment} 
\usepackage{stmaryrd} 
\usepackage{mathrsfs,mathtools} 
\usepackage{titlesec} 
\usepackage[font=small,labelfont=bf,tableposition=top]{caption} 
\usepackage{authblk} 
\usepackage{todonotes,color} 

\usepackage{mathabx}

\usepackage{hyperref} 
\hypersetup{
	colorlinks=true,
	linkcolor=blue,
	citecolor=green,
}

\numberwithin{equation}{section}

\theoremstyle{plain}
\newtheorem{lemma}{Lemma}[section]
\newtheorem{corollary}[lemma]{Corollary}
\newtheorem{proposition}[lemma]{Proposition}
\newtheorem{remark}[lemma]{Remark}

\newtheorem{definition}[lemma]{Definition}

\newtheorem{theorem}[lemma]{Theorem}
\newtheorem{ansatz}[lemma]{Ansatz}

	\newcommand{\be}{\begin{equation}}
	\newcommand{\ee}{\end{equation}}
	\newcommand{\bes}{\begin{equation*}}
	\newcommand{\ees}{\end{equation*}}
	\newcommand{\bpm}{\begin{pmatrix}}
	\newcommand{\epm}{\end{pmatrix}}
	\newcommand{\qtext}[1]{\quad \mbox{ #1 } \quad}
	
	\newcommand{\dron}[2]  {\frac{\partial#1   }{\partial#2}}
	
	\newcommand{\modu}[1]{\left\vert#1 \right\vert}
	
	\newcommand\eps   {{\varepsilon}}
	
	\makeatletter
	\newcommand*\bigcdot{\mathpalette\bigcdot@{.6}}
	\newcommand*\bigcdot@[2]{\mathbin{\vcenter{\hbox{\scalebox{#2}{$\m@th#1\bullet$}}}}}
	\makeatother

       \newcommand{\degre}{^{\degree}}
	
\DeclareMathOperator{\Tr}{Tr}	
\DeclareMathOperator{\Det}{Det}	

\def\gns@setupmathcomp{%
\expandafter\ifx\csname tcmu\endcsname\relax
  \DeclareSymbolFont{gns@font}{TS1}{\familydefault}{m}{n}
  \ifx\mv@bold\@undefined\else
    \SetSymbolFont{gns@font}{bold}{TS1}{\familydefault}{\bfdefault}{n}
  \fi
\DeclareMathSymbol{\tcdegree}{\mathord}{gns@font}{176}       
}
 
\def\gns@usetcsymbols{%
\DeclareRobustCommand{\degre}{%
  \ifmmode\tcdegree\else\textdegree\fi}
 
}

\definecolor{myGreen}{RGB}{0, 200, 0}

		\newcommand\cA{{\mathcal A}}

		\newcommand\cB{{\mathcal B}}
		
		\newcommand\cBt{{\tilde{\cB}}}
			
		\newcommand\fC{{\mathfrak C}}

		\newcommand\Dt  {{\tilde D }}

		\newcommand\rd {{\mathrm d }}

		\newcommand\rH  {{\mathrm H }}
		\newcommand\rHt  {\tilde{\rH }}

		\newcommand\cH  {{\mathcal H }}

		\newcommand\rK  {{\mathrm K }}

		\newcommand\rM  {{\mathrm M }}

		\newcommand{\NN}{{\mathbb N}}

		\newcommand\cO{{\mathcal O} }
		
		\newcommand\Qt  {\tilde{ Q }}
		
		\newcommand{\RR}{{\mathbb{R}}}
		
		\newcommand\rR  {{\mathrm R }}

		\newcommand{\TT}{{\mathbb T}}
		
		\newcommand\wt {{\tilde{w} }}

		\newcommand\ry{{\mathrm y}}

		\newcommand{\ZZ}{{\mathbb Z}}

		\newcommand\gam   {\gamma  }
		\newcommand\Gam   {\Gamma  }

		\newcommand\Gamh{{\hat{\Gam}}}

		\newcommand\lam   {\lambda }

		\newcommand\psit{\tilde{\psi}}

\def\UDelaunay{\Upsilon^{\mathrm{Del}}}
\def\UPoincare{\Upsilon^{\mathrm{Poi}}}
\def\OmegaM{\Omega_{\mathrm{M}}}
\def\iM{i_{\mathrm{M}}}
\def\fC{\varphi}
\def\fCP{\varphi_{\mathrm{CP}}}
\def\fCPE{\widehat{\varphi}_{\mathrm{CP}}}
\def\fiM{\psi}
\def\fiMC{\psi^{i_{\mathrm{M}}}}
\newcommand{\wJ}{\widetilde{J}}
\newcommand{\wOmegaM}{\widetilde{\Omega}_{\mathrm{M}}}
\newcommand{\whOmegaM}{\widehat{\Omega}_{\mathrm{M}}}
\newcommand{\prim}{\mathrm{pri}}
\newcommand{\secn}{\mathrm{sec}}
\newcommand{\hor}{\mathrm{h}}
\newcommand{\ver}{\mathrm{v}}

\newcommand{\Jres}{J_{\mathrm{res}}}

\def\al{\alpha}

\def\eps{\varepsilon}

\def\RR{\mathbb R}

\def\ZZ{\mathbb Z}

\def\NN{\mathbb N}

\def\PP{\mathcal P}

\def\CCC{\mathcal C}

\def\DDD{\mathcal D}

\def\NNN{\mathcal N}

\def\MM{\mathcal M}

\def\FF{\mathcal F}

\def\JJ{\mathcal J}
\def\GG{\mathcal G}
\def\II{\mathcal I}

\def\RRR{\mathcal R}

\def\TT{\mathbb T}
\def\TTT{\mathcal T}

\def\OO{\mathcal O}
\def\KK{\mathcal K}

\def\cH{\mathcal H}

\def\HH{\mathcal H}
\def\tR{\mathtt R}
\def\tD{\mathtt D}
\def\~{\tilde}

\def\de{\delta}

\def\wt{\widetilde}

\def\wh{\widehat}

\def\eps{\varepsilon}
\def\pa{\partial}

\def\unst{\mathrm{u}}
\def\sta{\mathrm{s}}
\def\inner{\mathrm{in}}
\def\outer{\mathrm{out}}

\def\sl{m}
\def\gg{s}
\def\CP{\mathrm{CP}}
\def\AV{\mathrm{AV}}
\def\Saros{\mathrm{Saros}}
\def\PM{\Pi}

\title{On the Arnold diffusion mechanism in Medium Earth Orbit} 
\author[1]{Elisa Maria Alessi}
\author[2,5]{Inmaculada Baldom\'a}
\author[3]{Mar Giralt}
\author[4,5]{Marcel Guardia}
\affil[1]{IMATI-CNR, Istituto di Matematica Applicata e Tecnologie informatiche ``E. Magenes'', Consiglio Nazionale delle Ricerche, Via Alfonso Corti 12, 20133 Milano, Italy}
\affil[2]{Departament de Matem\`atiques \& IMTECH, Universitat Polit\`ecnica de Catalunya, Diagonal 647, 08028 Barcelona, Spain}
\affil[3]{IMCCE, CNRS, Observatoire de Paris, Universit\'e PSL, Sorbonne Universit\'e, 77
Avenue Denfert-Rochereau, 75014 Paris, France}
\affil[4]{Departament de Matem\`atiques i Inform\`atica, Universitat de Barcelona, Gran Via, 585, 08007 Barcelona, Spain}
\affil[5]{Centre de Recerca Matem\`atica, Campus de Bellaterra, Edifici C, 08193 Barcelona, Spain}

\date{\today}

\begin{document}

\maketitle 

 \renewcommand{\thefootnote}{\roman{footnote}}
 \footnotetext[0]{\emph{E-mail adresses:} 
 \href{mailto:em.alessi@mi.imati.cnr.it}{em.alessi@mi.imati.cnr.it} (E.M. Alessi),
\href{mailto:immaculada.baldoma@upc.edu}{immaculada.baldoma@upc.edu} (I. Baldom\'a), 
\href{mailto:mar.giralt@obspm.fr}{mar.giralt@obspm.fr} (M. Giralt),
\href{mailto:guardia@ub.edu}{guardia@ub.edu} (M. Guardia).}
 \renewcommand{\thefootnote}{\arabic{footnote}}

\begin{abstract}
Space debris mitigation guidelines represent the most effective method to preserve the circumterrestrial environment. Among them, end-of-life disposal solutions play a key role. In this regard, effective strategies should be conceived not only on the basis of novel technologies, but also following an advanced theoretical understanding.
A growing effort is devoted to exploit natural perturbations to lead the satellites towards an atmospheric reentry, reducing the disposal cost, also if departing from high-altitude regions. In the case of the Medium Earth Orbit region, home of the navigation satellites (like GPS and Galileo), the main driver is the gravitational perturbation due to the Moon, that can increase the eccentricity in the long term. In this way, the pericenter altitude can get into the atmospheric drag domain and the satellite can eventually reenter.

In this work, we show how  an Arnold diffusion mechanism  can trigger the eccentricity growth. Focusing on the case of Galileo, we consider a hierarchy of Hamiltonian models, assuming that the main perturbations on the motion of the spacecraft are the oblateness of the Earth and the gravitational attraction of the Moon. First, the Moon is assumed to lay on the ecliptic plane and periodic orbits and associated stable and unstable invariant manifolds are computed for various energy levels, in the neighborhood of a given resonance. Along each invariant manifold, the eccentricity increases naturally, achieving its maximum at the first intersection between them. This growth is, however, not sufficient to achieve reentry. By moving to a more realistic model, where the inclination of the Moon is taken into account, the problem becomes non-autonomous and the satellite is able to move along different energy levels. Under the ansatz of transversality of the stable and unstable manifolds in the autonomous case, checked numerically, Poincar\'e-Melnikov techniques are applied to show how the Arnold diffusion can be attained, by constructing a sequence of homoclinic orbits that connect invariant tori at different energy levels on the normally hyperbolic invariant manifold. 

\end{abstract}

\tableofcontents
\section{Introduction}

One of the key actions to mitigate the space debris problem and ensure a sustainable exploitation of the circumterrestrial environment is to implement effectively end-of-life procedures. While for the Low Earth Orbit and the Geosynchronous protected regions, there exist well-defined guidelines \cite{IADC}, for the Medium Earth Orbit (MEO) region still a discussion is ongoing. This is due to the fact that the MEO region is very broad (in principle, it goes from an altitude of 2000 km up to the geostationary altitude) and not yet critical in terms of population density. 

The MEO region is mostly known because it is where the satellites of the Global Navigation Satellite Systems (GNSS) orbit, namely, GPS, Galileo, GLONASS and Beidou-M \cite{Alessi2014}. They cover a range of semi-major axis between about $25500$ km (GLONASS) and $29600$ km (Galileo) and their nominal inclination is $55\degre\pm 2\degre$ (GPS and Beidou-M), $56\degre\pm 2\degre$ (Galileo), $65\degre\pm 2\degre$ (GLONASS) (see, e.g., \cite{Rossi2008, Alessi2016} and references herein). In the recent ESA's zero debris policy \cite{ESA0debris}, the GNSS orbits are defined as ``valuable orbits'' and the possibility of extending the protection to this region is introduced.
Given the high altitude, so far at the end of life the GNSS satellites were either left in the operational orbit or re-orbited by a given amount \cite{Alessi2014}. The accumulation of non-operational satellites in a limited region eventually will lead to a critical situation in terms of potential collisions and thus fragmentation. For this reason,  Jenkin and Gick \cite{Jenkin2005} proposed to dilute the collision probability by increasing the orbital eccentricity of the satellites at the end of life. Indeed, the GNSS inclination and altitude are such that in the long term the third-body perturbation, that is, the gravitational perturbation exerted by Sun and Moon, could lead to a natural eccentricity growth up to reentry, if a suitable initial orbital orientation is chosen. This mechanism was proven numerically by several authors (see, e.g., \cite{Rossi2008, Radtke2015, Alessi2016, Armellin2018, Gondelach2019, Pellegrino2021, Pellegrino2022}), but an exhaustive theoretical explanation of the underlying mechanism is still missing. In particular, the singular resonance hypothesis, that is, an integrable model, cannot explain the maximum value of eccentricity that can be observed by considering a full dynamical model, nor the Saros periodicity detected.

More precisely, following \cite{Kaula1962} the disturbing function corresponding to the gravitational perturbation due to a third body can be written as a series expansion depending on semi-major axis, eccentricity, inclination of both the satellite and the third body, and a periodic term involving the longitude of the ascending node, the argument of perigee and the mean anomaly of both the satellite and the third body. By doubly-averaging the periodic terms (over the orbital period of the satellite and the orbital period of the Moon), and considering only the first-order effect, we define the secular Hamiltonian.\footnote{In what follows, this step will be omitted because it can be found in several past works (see references above).} 

A possible way to deal with the secular Hamiltonian is to assume that only one periodic term is dominant at a time, in particular, when its argument is resonant. Most of the past works\footnote{It is beyond the scope of this work to provide a review of all the past investigations on the subject. Here, we recall only the works that are relevant to our contribution.} focus on such ``isolated resonance hypothesis'' and on the eccentricity growth associated with an inclination-dependent resonance, involving the argument of the pericenter $\omega$ and the longitude of the ascending node of the satellite~$\Omega$. For quasi-circular orbits at MEO altitudes, and assuming the Galileo inclination ($\approx 56\degre$), the dominant resonance is $2\dot\omega+\dot\Omega\approx 0$\footnote{For GLONASS the inclination is $\approx 63\degre$ and thus the dominant resonance is $2\dot\omega$.}. As a first approximation, $\omega$ and $\Omega$ change in time because of the Earth's oblateness, but the third body perturbation has also an effect.


The authors in \cite{Rosengren2015, Daquin2016} assumed the isolated resonance hypothesis and brought forward the concept that the eccentricity growth that ensures a reentry is due to a chaotic behavior, that occurs when two or more resonances overlap. The Chirikov criterion is mainly proven by detailed Fast Lyapunov Indicators (FLI) maps that show the location of the chaotic regions.  The phase space associated with each possible resonance of the secular dynamics was investigated also in~\cite{Lei2022}, but in this case, the authors remarked the role of the Laplace plane.
More recently, the authors in~\cite{Gkolias2019,Daquin2022,Legnaro2023} proposed the idea that the high eccentricity growth is due, not to the chaos generated by overlapping resonances, rather to the normally hyperbolic invariant manifold (NHIM) associated with the given resonance and through an Arnold diffusion mechanism. They showed, with refined FLI maps and an {\it ad hoc} Hamiltonian derived for circular orbits in the neighborhood of the resonance, that what triggers the phenomenon is the variation of a given integral of motion. They also emphasized the role of the Laplace plane for the motion on the invariant tori.
Some of the invariant objects ``necessary'' to achieve Arnold diffusion were already identified in~\cite{ Daquin2022,Legnaro2023}. These papers also showed, numerically, hyperbolicity and unstable motions along the resonance. 


%
The goal of the present paper is to explain the mechanism that creates the drift and how to construct drifting orbits across the neighborhood of an isolated resonance. 
In particular, we focus our study on the resonance $2\omega + \Omega$ with the Galileo satellites values. Nonetheless, our approach is fairly general and can be adapted for other resonances and values.
In other words, we follow the idea that it is the NHIM the cause that we have to investigate, and we will show how it is possible to construct homoclinic connections that ``connect'' different energy levels on the NHIM associated with the $2\omega+\Omega$ resonance. Moreover, we will show how to construct orbits that shadow (follow closely) this sequence of homoclinic orbits. These shadowing orbits achieve a drift in eccentricty large enough so that the satellite renters the Earth's atmosphere. 

We will focus only on the third-body effect exerted by the Moon, thus neglecting the Sun, and we will apply a perturbative approach with respect to the inclination of the Moon on the ecliptic plane~$\iM$ (which has a real value of $\iM = 5.15^{\circ}$). More details on this will be given in the next subsection. 
Notice that the model we consider is a rather simplified version of the initial problem. However, we expect that the outcome of this work will pave the way for the analysis of a more realistic model, that can be used to define effective end-of-life solutions.


The paper is structured as follows. In Section \ref{sec:secularHam} we define the secular Hamiltonian and provide formulas for it. Then, we state the main result of this paper. In Section \ref{sec:coordinates} we introduce a ``good'' system of coordinates which captures the timescales of the model. In Section \ref{section:geometry}, relying on what we call hierarchy of models, we explain how the Arnold diffusion mechanism takes place in the secular Hamiltonian. 
Finally, we also provide definitions for the main tools we use: a normally hyperbolic invariant manifold and the associated scattering maps. In Section \ref{sec:coplanar} we analyze the dynamics for the coplanar model (that is, taking $i_\rM=0)$. We use this analysis to describe, in Section \ref{sec:fullham}, the dynamics for the full secular Hamiltonian with $i_\rM>0$ small enough.

\subsection{An Arnold diffusion mechanism}

V. Arnold in 1964 (see \cite{arnold1964instability}) showed, in  a cleverly chosen model,  that actions can vary drastically in nearly integrable Hamiltonian systems (of at least 2 and a half degrees of freedom, that is a phase space of at least dimension 5). Then, he conjectured that such behavior, nowadays called Arnold diffusion, should be typical. In (nearly integrable) physical models, it is  expected that Arnold diffusion is a fundamental mechanism leading to transport in phase space. Such transport is achieved by drifting along resonances.

The understanding of Arnold diffusion mechanisms has had outstanding progress in the last decades, relying on a wide variety of techniques:  the original geometric approach by Arnold, which has been deeply developed in  \cite{Bolotin:1999, DelshamsLS00, DelshamsLS06b, delshams2006biggaps, DelshamsH05, Gelfreich:2008, delshams2019instability, GelfreichTuraev2017},  variational methods \cite{ChengY04, Cheng:2017}, topological tools \cite{gidea2006topological,ClarkeFG23},  the so-called separatrix map \cite{Treschev04, Treschev:2012} or a combination of different approaches \cite{Kaloshin:2016, Kaloshin:2020}.

As stated before, the goal of this paper is to explain how an Arnold diffusion mechanism can enhance drift in eccentricity along the $2\omega+\Omega$ resonance. The mechanism we propose relies on geometric tools in the spirit of the seminal work by Arnold. They are explained in detail in Section \ref{sec:mathtools}.

Let us mention here just the main ingredients. We show that, along the resonance, there exists a normally hyperbolic invariant cylinder (see Definition \ref{def:nhim} below). 
The stable and unstable manifolds of this cylinder intersect transversally. Thus, we can construct trajectories that, following closely a sequence of homoclinic orbits to the cylinder, achieve a considerable drift in eccentricity. Such drift allows the satellite, by flattening its osculating ellipse, to enter the Earth's atmosphere.

To perform such analysis, we assume that the inclination of the Moon $i_\rM$ with respect to the ecliptic plane is small enough. This allows us to use perturbative arguments to construct the ``highways'' that lead to the drift. Note that it is fundamental that $i_\rM>0$, because in the coplanar case, namely, when $i_\rM=0$, there are not enough dimensions to achieve Arnold diffusion.

Even if our tools rely on the fact that $i_\rM>0$ is small enough, we expect that the same mechanism takes place for a realistic value of $i_\rM$, because, indeed, it is expected that transport phenomena are even stronger and more robust in far from integrable Hamiltonians.
Nevertheless, in that setting, one cannot combine analytical and numerical techniques to describe the cylinder and its invariant manifolds, as we do in the present paper. Instead, one has to describe them fully numerically. In this work, numerical simulations are done for a lower dimensional model (what we call the coplanar Hamiltonian) and can be kept ``simple''. Otherwise, it would require to compute numerically high dimensional objects in the 5-dimensional phase space.
Moreover, when $i_\rM$ increases other resonances could come to play also an important role and an resonance overlapping could take place. 

Finally, note that, for applications, it is fundamental to know the speed of such transport mechanisms. For arbitrarily small $i_\rM>0$ the mechanism is rather slow (the drifting time is $T\sim 1/i_\rM$). However, for a realistic value of $i_\rM$ it is expected to be faster. Moreover, one may expect that combining Arnold diffusion mechanism with maneuvers could speed up considerably the drifting.

\section*{Acknowledgements}

E.M. Alessi acknowledges support received by the project entitled ``{co-orbital motion and
three-body regimes in the solar system}",
funded by Fondazione Cariplo through the program ``{Promozione dell'attrattivit\`a e competitivit\`a
dei ricercatori su strumenti
dell'European Research Council -- Sottomisura rafforzamento}".

I. Baldom\'a has been supported by the grant PID-2021-
122954NB-100 funded by the Spanish State Research Agency through the programs
MCIN/AEI/10.13039/501100011033 and “ERDF A way of making Europe”.

M. Giralt has been supported by the European Union’s Horizon 2020 research and innovation programme under the Marie Sk\l odowska-Curie grant agreement No 101034255.
M. Giralt has also been supported by the research project PRIN 2020XB3EFL ``Hamiltonian and dispersive PDEs".

M. Guardia has been supported by the European Research Council (ERC) under the European Union's Horizon 2020 research and innovation programme (grant agreement No. 757802). 
M. Guardia is also supported by the Catalan Institution for Research and Advanced Studies via an ICREA Academia Prize 2019. 

This work is also supported by the Spanish State Research Agency, through the Severo Ochoa and Mar\'ia de Maeztu Program for Centers and Units of Excellence in R\&D (CEX2020-001084-M).


\section{The secular Hamiltonian system and the main result}\label{sec:secularHam}


Let us consider a spacecraft that is affected by the gravitational attraction of the Earth, the perturbation due to the Earth's oblateness and the lunar gravitational perturbation. In the geocentric equatorial reference system, the motion of the spacecraft takes place on an ellipse, that is described by the orbital elements semi-major axis $a$, eccentricity $e$, inclination $i$, longitude of the ascending node $\Omega$ and argument of pericenter $\omega$. The ellipse changes in time due to the perturbations. The orbit of the Moon is defined in the geocentric ecliptic reference system by the corresponding orbital elements $(a_\rM, e_\rM, i_\rM, \Omega_\rM, \omega_\rM)$, where $a_\rM = 384400 \,\mbox{km}$, $e_\rM =0.0549006$, $i_\rM	=5.15\degre$, while the longitude of the ascending node of the Moon with respect to the ecliptic plane varies approximately linearly with time in a period $\TTT_{\Omega_\rM}$ of 1 Saros (about 6585.321347 days~\cite{aR73, eP91}) due to the solar gravitational perturbation, namely,
\begin{equation}\label{eq:OmegaM}
    \Omega_\rM(t)=	\Omega_{\rM,0} + n_{\Omega_\rM}t, \quad \quad \quad n_{\Omega_\rM}=2\pi/\TTT_{\Omega_\rM},
\end{equation}
where $\Omega_{\rM,0}$ is the longitude of ascending node of the Moon at a given epoch.

Since we consider the secular model (the system averaged over the mean anomalies of both the satellite and the moon), the semi-major axis $a$ is a constant of motion which, in the case of Galileo, corresponds to $a=29600$ km.
Let us define
 $$\alpha	=	a /a_\rM,$$
which characterizes the Earth-satellite distance with respect to the Earth-Moon distance.

Then, in Delaunay action-angle variables\footnote{Recall that in celestial mechanics, these are the classical canonical variables, defined as 	\be \label{def:LGHlgh}
		\begin{aligned}
				L	&=	\sqrt{\mu a},		&\quad			l	&=	M,		\\		
				G	&=	L\sqrt{1-e^2},		&\quad			g	&=	\omega,	\\
				H	&=	G\cos i,			&\quad			h	&=	\Omega.	
		\end{aligned}
	\ee},
 the secular dynamics is described by the Hamiltonian \cite{Daquin2016}
\begin{equation}\label{def:fullHamiltonian}
\mathtt{H}(L,G,H, g, h, \OmegaM;\iM)=    \rH_\rK(L)	
		+	\rHt_{0}(L, G, H) +	\alpha^3 \rHt_1(L,G,H, g, h, \OmegaM;\iM),
\end{equation}
where
\begin{equation*} 
\rH_\rK(L)=	-\frac{1}{2}\frac{\mu^2}{L^2}
\end{equation*}
 is the constant term associated with the Earth's monopole, being $\mu=398600.44 \,\mbox{km}^3/\mbox{s}^2$ the mass parameter of the Earth, and
 \be \label{def:initialHamiltonians_J2}
\rHt_0(L, G, H)	=\frac{1}{4}\frac{\rho_0}{L^3} \frac{G^2	-3H^2}{G^5}
\ee
is the perturbative term associated with the Earth's oblateness,  averaged over the orbital period of the spacecraft, being $\rho_0=\mu^4J_2 R^2$, with $J_2 = 1.08 \times 10^{-3}$ the coefficient of the second zonal harmonic in the geopotential and $R = 6378.14\,\mbox{km}$ the mean equatorial radius of the Earth. 
The secular perturbative term due to the Moon is instead
\begin{equation*} 
\begin{split}
\rHt_1	&(L,G,H, g, h, \OmegaM; \iM)
\\=&-\frac{\rho_1}{L^2} 
\sum_{m=0}^2 \sum_{p=0}^2 \Dt_{m,p}(L,G,H) \sum_{s=0}^2
c_{m,s}F_{2,s,1}(i_\rM) 
\\&\times \left[U_2^{m,-s}(\epsilon)
\cos\left(\psit_{m, p, s}(g,h,\OmegaM)\right) +	U_2^{m,s}(\epsilon)\cos\left(\psit_{m, p, -s}(g,h,\OmegaM)\right)
\right].
\end{split}
\end{equation*}
The function $U_2^{m, \mp s}(\epsilon)$ corresponds to the Giacaglia function with $\epsilon =	23.44\degre$ (see Table~\ref{tab:U} in Appendix~\ref{appendix:HamiltonianDelaunay}) being the obliquity of the ecliptic with respect to the equatorial plane.
Also, one has that
\be \label{def:Dtilde}
\Dt_{m,p}(L, G,H)	=	\tilde{F}_{m,p}(G,H)\tilde{X}_{p}(L,G),
\ee
where $\tilde{F}_{m,p}(G,H)$ is a function of the Kaula's inclination functions $F_{2, m, p}(i)$\footnote{Since $H=G\cos i$, one has that $\tilde{F}_{m,p}(G,H)= F_{2,m,p} (\arccos\frac{H}{G}).$}, 	see~\cite{Kaula66}, and $\tilde{X}_{p}(L,G)$ is a function of the zero-order Hansen coefficient $X_0^{2,2-2p}(e)$\footnote{Since $G=L\sqrt{1-e^2}$, one has that $\tilde{X}_p(L,G) = X_0^{2,2-2p}\left(\sqrt{1-\frac{G^2}{L^2}} \right)$.} , see~\cite{Hughes1980}, and
\be 
\label{def:PsiTilde}
\psit_{m,p, \sigma}(g,h, \OmegaM)=	2(1 -p)g	
+m h + \sigma\left(\Omega_\rM-\frac{\pi}{2}\right)	- 	\ry_{\modu{\sigma}}\pi, 
\ee  
 where $\sigma= \pm s$
 defines the relative orientation satellite-Moon.
 In addition, the other variables are constants defined as
 \begin{equation}\label{def:auxiliaryFunctionsOriginal}
 \begin{aligned}
 \rho_1 	&= &	\frac{\mu\mu_\rM}{(1	-	e_\rM^2)^{3/2}},\quad \quad \quad  \quad \quad  \quad  \quad  \quad
 &\eps_n	&=& 	\left\{
		\begin{array}{ll}
		1		&	\mbox{if $n=0$}\\
		2		&	\mbox{if $n\neq 0$}
		\end{array}\right.,\\
c_{m,s}	 & = &	(-1)^{\lfloor m/2\rfloor} \frac{\eps_m\eps_s}{2}\frac{(2-s)!}{(2 + m)!},\quad \quad \quad ~
	&\ry_{|s|}	&=&	\left\{
		\begin{array}{ll}
		0		&	\mbox{if $s$ is even}\\
		1/2		&	\mbox{if $s$ is odd}
		\end{array}\right. ,
 \end{aligned}
\end{equation}
 where $\mu_\rM=	4902.87 \,\mbox{km}^3/\mbox{s}^2$ is the mass parameter of the Moon.

\begin{remark}\label{rem:auton}
	By Kaula's inclination functions (see \cite{Kaula66}), one obtains that
	\begin{align} \label{eq:KaulaInclinationMoon}
		F_{2,s,1}(\iM) 
		=	
		\left\{\begin{array}{ll}
			-\frac12+\frac34\sin^2 \iM		&	\mbox{if $s=0$,}\\[0.5em]
			-\frac32 \sin \iM \cos \iM 				& 	\mbox{if $s=1$,}\\[0.5em]
			\frac32 \sin^2 \iM			& 	\mbox{if $s=2$.}
		\end{array}
	\right.
	\end{align}
    We remark that when $i_\rM=0$, the inclination function $F_{2,s,1}(0)=0$ for all~$s$, except for $s=0$ which is $F_{2,0,1}(0)=-1/2$.
    In other words, assuming that the Moon lies on the ecliptic plane, the Hamiltonian $\mathtt{H}$ in \eqref{def:fullHamiltonian} is autonomous, because the angle $\psit_{m,p,s}(g,h, \OmegaM)$ does not depend on $\OmegaM$ for $s=0$. 
\end{remark}
 
\begin{remark}\label{rmk:unities}
In the numerical computations that will be presented throughout the text, all the variables will be taken in non-dimensional units defined in such a way that the semi-major axis $a$ of the orbit of the Gallileo satellites (equal to 29600 km) is the unit of distance and the corresponding orbital period is 2$\pi$.     
\end{remark}

The main result of the paper is the following. It assumes several ans\"atze that are stated below and are verified numerically.

\begin{theorem}\label{thm:main}
Consider the secular Hamiltonian $\mathtt{H}$ in \eqref{def:fullHamiltonian} with the parameters just fixed, take
\begin{equation}\label{def:alpha}
\alpha= 0.077
\end{equation}
and assume that the Ans\"atze \ref{ansatz:periodic}, \ref{ansatz:phaseshiftouter} and \ref{ansatz:straightening} are satisfied. 

Then, for $i_\rM>0$ small enough, there exist a time $T>0$ and 
a trajectory $z(t)=(L(t),G(t),H(t), g(t), h(t), \OmegaM(t))$ such that
\[
\mathtt{H}(z(0))=1.7\cdot 10^{-8} \qquad \text{and}\qquad \mathtt{H}(z(T))=1.3\cdot 10^{-6}.
\]
Moreover, along the energy increase the  osculating eccentricity and inclination satisfy
\[
e(0)= 0\qquad  \text{and}\qquad e(T)=0.78,
\]
and
\[
i(t)\in [56.06\degre,58.09\degre] \qquad \text{for}\qquad t\in [0,T].
\]
\end{theorem}
Note that such evolution would not be possible if $i_\rM=0$ since in that case the Hamiltonian is autonomous and therefore the energy is a first integral (recall Remark~\ref{rem:auton}).

Concerning the evolution of the eccentricity, we note that it is already oscillating when $i_\rM=0$. However, the size of this oscillations depends on the energy value. Indeed, the orbit that we obtain in Theorem \ref{thm:main} is such that at the initial time, assuming $i=56.06\degre$ the eccentricity is oscillating (approximately) between 0 and 0.35, whereas at the final time $T$ is oscillating between 0 and 0.78. Therefore, at the initial time the satellite is far from the Earth's atmosphere whereas at the final time the satellite is reentering.

Theorem \ref{thm:main} does not provide any estimate on the time $T>0$ needed to achieve the drift in eccentricity. As mentioned before, more quantitative shadowing arguments compared to those used in the present paper (see, for instance \cite{treschev02, Piftankin06}) should lead to estimates for $T=T(\iM)$ of the form
\[
|T|\leq \frac{C}{\iM},
\]
for some constant $C>0$ independent of $\iM>0$ (see Remark \ref{rmk:time} below).

Finally, in what follows, we will omit the Keplerian term of the Hamiltonian $\rH_{\rK}$, since it does not contribute to the variation of the orbital elements.
As well, we omit the dependence of the Hamiltonian on the variable $L$ since it is a constant of motion.

\section{A good system of coordinates
}\label{sec:coordinates}

Let us start by focusing on the region of the phase space where an  orbit satisfies the ``$2g+h$ resonance'' in the unperturbed problem, defined by the Hamiltonian $ \rHt_0$ in \eqref{def:initialHamiltonians_J2} (i.e., $\alpha = 0$), which is integrable.


	%
  
The resonance is the set of points in action space where the condition $\dot{x}= 0$ holds, where $x=2g+h$ is the resonant angle. This is satisfied if the orbital inclination is equal to $i_\star\simeq 56.06\degre$ in the prograde case or to $i_\star\simeq 110.99\degre$ in the retrograde case.
Indeed, according to the equations of motion, the resonance occurs when
\[
5H^2 - G^2 - HG	=	0,
\]
that is, for all $G\neq 0$ (i.e. $e<1$) and $H=G\cos i_\star$ (see~\eqref{def:LGHlgh}) that satisfy
\begin{equation*}
5 X^2 - 1 - X = 0, \quad \quad \quad X = \cos i_\star.
\end{equation*}
Hence, we can distinguish two situations: 
\begin{itemize}
\item the prograde case for 
$i_{\star}	=	\arccos\left(\frac{1 + \sqrt{21}}{10}\right)	\simeq 56.06\degre$,
\item the retrograde case for 
$i_{\star}	=	\arccos\left(\frac{1 - \sqrt{21}}{10}\right)	\simeq 110.99\degre$.
\end{itemize}

In what follows, $i_\star$ will be mentioned as the inclination of the (exact) $2g+h$-resonance and we will focus on the prograde case.

The unperturbed Hamiltonian highlights that $x$ is constant for $i =i_\star$, while it circulates for $\modu{i-i_\star}>0$. Moreover, in a small enough neighborhood of the (exact) resonance the angular variables evolve at different rates: $g$ and $h$ are ``fast'' angles compared to $x$ which undergoes a ``slow'' rotation of $\cO(i-i_\star)$.
 
\subsection{Slow-fast coordinates \texorpdfstring{$(y,x)$}{(y,x)}}\label{sec:slowfast}

Instead of using the Delaunay action-angle variables, in order to take advantage of the timescales separation, we introduce the transformation
$
{(G, H, g, h) = \UDelaunay (y, \Gam, x, h) }
$
 given by
\begin{equation}\label{eq:xy}  
x = 2g + h, \qquad \qquad 
y = \frac{G}2, \qquad \qquad \Gam = H-\frac{G}2 
\end{equation}
and the symplectic form
$
\rd x 	\wedge	\rd y
+	\rd	h	\wedge	\rd \Gam.
$

Notice that several symplectic transformations are possible, however we prefer the one such that the resonant action $y$ does not depend on the inclination\footnote{Even if we focus our analysis on the $2g+h$ resonance, our approach is fairly general. A similar change of coordinates can be considered for a different resonance. What is important is to choose the angle $x$ as the one corresponding to the slow dynamics. The other changes of variables are chosen so that the transformation is symplectic.}. 
Hence, the action-angle variables $(y,x)$ are associated with the variation in eccentricity.

  In slow-fast variables, the Hamiltonian of the full problem can be written as
\begin{equation*}
		\rH(y,\Gam, x, h, \OmegaM;\iM) = \rH_{0}(y, \Gam)	+	\alpha^3 \rH_1 (y,\Gam, x, h, \OmegaM; \iM)	
\end{equation*}
where
\begin{equation*}
\begin{split}
\rH_0(y , \Gam)	&=	(\rHt_0 \circ \UDelaunay)(y,\Gam)
= \frac{\rho_0}{128}\frac{y^2  -6y\Gam- 3\Gam^2}{L^3y^5} 
\end{split}
\end{equation*}
and
\be 
\label{def:hamiltonianDelaunayH1}
\begin{split}
\rH_1&(y , \Gam, x, h, \OmegaM;\iM)
=(\rHt_1 \circ \UDelaunay) (y , \Gam, x, h, \OmegaM;\iM).
\end{split}
\ee
See Appendix~\ref{appendix:HamiltonianDelaunay} and, in particular, equation~\eqref{eq:expansionH1Delaunay} for a detailed expansion of the functions involved in the expression of the Hamiltonian $\rH_1$.

Notice that, in the unperturbed problem ($\alpha = 0$), $\Gamma$ and $y$  are first integrals of the problem, because the Hamiltonian does not depend on the angles $h$ and $x$, while in the full problem ($\alpha>0$), the phase space is no more integrable. Moreover, in the integrable case once fixed $a$ and $e$, the dynamics can be catalogued as a function of the inclination.


In the coordinates just defined, the $2g + h$–resonance becomes the $x$–resonance. In the unperturbed problem, the action space associated with the resonance is defined by the condition
	\bes	
\dot{x}=\dron{\rH_0(y, \Gam)}{y}	= \frac{3}{128}\frac{\rho_0}{L^3}\frac{   5\Gam^2 + 8y\Gam - y^2}{y^6} = 0.
	\ees
 In other words, the resonance can take place on the two lines 
	\begin{equation}\label{def:resonance}
		\begin{split}
		\Gam	=	y\frac{-4 + \sqrt{21}}{5}	 \quad \text{and} \quad
		\Gam	=	y\frac{-4 - \sqrt{21}}{5},	 
		\end{split}
	\end{equation}
with $(y, \Gam) \neq (0,0)$, that are associated with prograde and retrograde orbits, respectively.

\subsection{Poincar\'e coordinates \texorpdfstring{$(\eta,\xi)$}{(eta,xi)}}
\label{sec:poincare}

The next step is to study the dynamics in the neighborhood of a circular orbit, that is for small values of $e>0$.
However, the slow-fast variables derived from the Delaunay variables (as happens with the original Delaunay variables) are singular at {$e=0$}. In order to overcome this difficulty, we introduce the set of Poincar\'e coordinates
$
	{(y, \Gam, x, h) = \UPoincare (\eta, \Gam, \xi, h)  }
$
	where
	\begin{equation}\label{eq:xieta}
		\xi		=	\sqrt{2L	-	4y}\cos \left(\frac{x}2\right), \qquad \qquad \qquad
		\eta	=	\sqrt{2L	-	4y}\sin \left(\frac{x}2\right),
	\end{equation}
which are symplectic. 	Notice that $\xi$  and $\eta$ are respectively equivalent to 
		$e\cos (x/2)$ and
		$e\sin (x/2)$ for quasi-circular orbits,
that is when $e\approx 0$.

In this set of coordinates, the Hamiltonian of the full problem can be written as
\begin{equation}\label{def:Ham:H:Poinc}
\HH(\eta,\Gam, \xi, h, \OmegaM;\iM) = \HH_{0}(\eta, \Gam,\xi)	+	\alpha^3 \HH_1(\eta,\Gam, \xi, h, \OmegaM;\iM)
\end{equation}
where
\begin{equation}\label{def:hamiltonianPoincareH0}
\begin{aligned}
\HH_{0}(\eta, \Gam,\xi)
&=
(\rH_0 \circ \UPoincare)(\eta, \Gam,\xi) \\
&=
\frac{\rho_0}{2} 
\frac{(2L-\xi^2-\eta^2)^2 - 24(2L-\xi^2-\eta^2) \Gam - 48\Gam^2}{L^3(2L-\xi^2-\eta^2)^5}
\end{aligned}
\end{equation}
and
\begin{equation} \label{def:hamiltonianPoincareH1}
	\HH_1(\eta , \Gam, \xi, h, \OmegaM;\iM) 	=	(\rH_1 \circ \UPoincare) (\eta , \Gam, \xi, h, \OmegaM;\iM).
\end{equation}
See Appendix~\ref{appendix:HamiltonianPoincare} and, in particular, equation~\eqref{eq:expansionH1Poincare} for a detailed expansion of the functions involved in the expression of Hamiltonian $\HH_1$.

\section{The hierarchy of models and dynamical systems tools} 
\label{section:geometry}

We devote this section to explain the geometric framework that will lead to the drifting orbits. To this end, it is convenient to autonomize the Hamiltonian. Let us recall equation~{\eqref{eq:OmegaM}} for the time variation of $\OmegaM\in\TT$, namely, 
\begin{equation*}
    \Omega_\rM(t)=	\Omega_{\rM,0} + n_{\Omega_\rM}t.
\end{equation*}
We introduce $\Omega_\rM$ as a variable and  define its symplectic conjugate variable $J$. Then, we define the 3-degree-of-freedom Hamiltonian
\begin{equation}\label{def:HamExtPhsSp}
\KK(\eta, \Gam, J, \xi, h, \Omega_\rM;i_\rM)=\HH (\eta, \Gam, \xi, h, \Omega_\rM;\iM)+n_{\Omega_\rM}J,
\end{equation}
where $\HH$ is the Hamiltonian introduced in \eqref{def:Ham:H:Poinc}.

We analyze this Hamiltonian in two steps, relying on the smallness of $i_\rM$. In Section~\ref{section:hierarchy}, we explain the strategy of this analysis. In this section, we also introduce the $h$-averaged Hamiltonian. We \emph{will not rely} on it to construct our diffusion mechanism. However, we introduce it since it has been widely  studied numerically in literature (see, e.g., \cite{Daquin2022}).  and is a convenient simplified model to use as a first step in certain numerical studies of the full problem. In the paper \cite{AlessiBGGP23}, we perform a detailed analytic study of the $h-$averaged Hamiltonian, and we describe its equilibrium points and their stability. 

%

In Section \ref{sec:mathtools}, we introduce the two main tools that will create the ``highway'' of unstable orbits: a normally hyperbolic invariant cylinder and the associated scattering maps.

\subsection{The hierarchy of models}\label{section:hierarchy}

To construct the intermediate models, we rely on the fact that the model depends on two parameters: the inclination of the Moon with respect to the ecliptic plane $i_\rM$ and the ratio between the semi-major axis of the satellite and the one of the Moon $\al$.

In our analysis,  we consider $i_\rM$ arbitrarily small - the perturbative parameter -  and a realistic value for $\al$ (see Theorem \ref{thm:main}). Still, the smallness of $\al$ creates different timescales that may be taken into account.


\paragraph{Main reduction: the Coplanar Model}

Since we are assuming that the inclination of the Moon is small, the main reduction that one can do to have an intermediate model is to take $i_\rM=0$. We refer to this model as the Coplanar Model since it corresponds to assuming that the orbit of the Moon is coplanar to that of the Earth.
When doing this reduction, the Hamiltonian $\KK$ in \eqref{def:HamExtPhsSp} becomes $\Omega_\rM$ independent, see Remark~\ref{rem:auton}.

Starting from \eqref{def:HamExtPhsSp}, we define
\begin{equation}\label{def:KHam:CP}
\begin{split}
\KK_\CP(\eta, \Gam, J, \xi, h)=\KK(\eta, \Gam, \xi, h, \OmegaM;0)
=\HH (\eta, \Gam, \xi, h, \Omega_\rM;0)+n_{\Omega_\rM}J,
\end{split}
\end{equation}
where the subscript $\CP$ stands for coplanar. Since $\KK_\CP$ is $\Omega_\rM$-independent,  $J$ is a first integral. In fact, one can work in the reduced phase space and consider the 2-degree-of-freedom Hamiltonian
\begin{equation}\label{def:Ham:Coplanar} 
\HH_\CP(\eta, \Gam, \xi, h)=\HH(\eta, \Gam, \xi, h, \OmegaM;0).
\end{equation}
This Hamiltonian can be written as
\[
\HH_\CP(\eta, \Gam, \xi, h)= \HH_{0}(\eta, \Gam,\xi)	+	\alpha^3 \HH_{\CP,1}(\eta,\Gam, \xi, h),
\]
where $\HH_0$ is the Hamiltonian introduced in~\eqref{def:hamiltonianPoincareH0} and $\HH_{\CP,1}$ is the Hamiltonian $\HH_1$ in~\eqref{def:hamiltonianPoincareH1} with $i_\rM=0$, that is,
\begin{equation*}
\HH_{\CP,1}(\eta,\Gam, \xi, h)=
\HH_1(\eta,\Gam,\xi,h,\OmegaM;0).
\end{equation*}
See Appendix~\ref{appendix:HamiltonianPoincare} and, in particular \eqref{eq:expressionPoincareHCP1}, for the explicit expression of $\HH_{\CP,1}$.

\paragraph{A possible further reduction: the $h$-averaged problem}%

The departing point of our analysis is the coplanar Hamiltonian in~\eqref{def:Ham:Coplanar} above. However, in  literature extensive investigations have been based on the $h$-averaged coplanar model. The reduction to the $h$-averaged model is based on the fact that, if $\alpha$ is a small parameter, the autonomous Hamiltonian $\HH_\CP$ has a timescale separation between the slow and fast angles, respectively $x$ and $h$. Indeed,  
\begin{equation*}
\dot \xi,\dot \eta\sim \alpha^3 \qquad \text{whereas}\qquad \dot h\sim 1.
\end{equation*}
A classical way to exploit this feature is to simplify the Hamiltonian $\HH_\CP$  by another one in which the fast oscillations have been removed by
 averaging over the longitude of the ascending node $h$. That is, 
\[	\HH_\AV(\eta,\Gam, \xi)	=	\frac{1}{2\pi} \int_0^{2\pi}\HH_\CP(\eta, \Gam, \xi, h)\rd h.\]
	We refer to this Hamiltonian  as the  $h$-averaged Hamiltonian.
Note that since the Hamiltonian $\HH_0$ in 	\eqref{def:hamiltonianPoincareH0} is $h$-independent, $\HH_\AV$ can be written as 
\begin{equation*}
\HH_\AV(\eta,\Gam, \xi)	=	\HH_{0}(\eta,\Gam,\xi)	+	 \frac{\alpha^3}{2\pi} \int_0^{2\pi}\HH_{\CP,1}	(\eta,\Gam, \xi, h)\rd h.
\end{equation*} 
Notice that $\HH_\AV$ is a 1-degree-of-freedom Hamiltonian provided that $\Gamma$ is a first integral of the $h$-averaged system. 

\subsubsection{Theoretical results for the coplanar and \texorpdfstring{$h$}{h}-averaged system}\label{sec:theoretical_averaged_coplanar}

The mechanism we use to show the existence of drifting orbits is the Arnold diffusion mechanism (see Section~\ref{sec:mathtools} and~\ref{sec:Arnold} below). A key point in the analysis is to study the relative position between the invariant manifolds of the hyperbolic periodic orbits of the coplanar model. 

In this work, we use circular periodic orbits, that is, the periodic orbits located at $(\eta,\xi)=(0,0)$ (see~\eqref{eq:xieta}).  Indeed, let us recall that $\HH_{\CP}(\eta,\Gam,\xi,h)$ is a $2$-degrees-of-freedom autonomous Hamiltonian that can be expressed as $\HH_{\CP}=\HH_0 + \al^3\HH_{\CP, 1}$ with $\HH_0$ as given in~\eqref{def:hamiltonianPoincareH0} and  $\HH_{\CP, 1}$ in~\eqref{eq:expressionPoincareHCP1}.
From the explicit expression of the Hamiltonian obtained in Appendix~\ref{appendix:HamiltonianPoincare}, one can easily see that it does not have linear terms with  respect to $(\eta,\xi)=(0,0)$.
This implies that $(\eta,\xi)=(0,0)$ is invariant. Then, to analyze the circular periodic orbits, it is enough to look for periodic solutions of 
the $1$-degree-of-freedom Hamiltonian $\HH_{\CP}(0,\Gam,0,h)$.
In Appendix \ref{appendix:HamiltonianPoincare} we provide formulas for this Hamiltonian (see Lemma~\ref{lemma:periodicOrbit}).


In~\cite{AlessiBGGP23}, we prove the existence of periodic orbits at $\{\eta=\xi=0\}$ when the semi-major axis $a \in [a_{\min},a_{\max}]$ with 
$$
a_{\max}=30000\, \mathrm{km}, \qquad a_{\min} =6378.14\, \mathrm{km }\; \text{(to avoid collision)}.
$$
We also introduce $L_{\min,\max}=\sqrt{\mu a_{\min,\max}}$.
\begin{theorem}\label{thm:coplanar_periodic_orbits}
For any $L\in [L_{\min},L_{\max}]$, we define
$$
\mathbf{E}_1 (L) = {\HH}_\CP(0,0.49 L,0,\pi), \qquad 
\mathbf{E}_2(L) = {\HH}_\CP (0,0,0,0).
$$
Then, for any energy level $E\in [\mathbf{E}_{1}(L),\mathbf{E}_{2}(L)] $, there exists a periodic orbit of the form $\PP_E(t)=(0, {\Gam}_E(t),0,h_E(t))$ such that, $h_E(0)=0$ and, for $t\geq 0$, $\dot{h}_E(t)\neq 0$, $\Gam_E(t) \in [0, 0.49 L]$ and
$$
\HH_\CP(0, {\Gam}_E(t),0,h_E(t))= E.$$  

In addition, for the case of Galileo, namely $L=1$ in non-dimensional units (which corresponds to $a=29600$ km see Remark~\ref{rmk:unities}), one has that
$$
\mathbf{E}_1(1)=-2.515161379204321 \cdot 10^{-5}, \qquad \mathbf{E}_2(1)= 2.477266122798186\cdot 10^{-6}.
$$
\end{theorem}
This result ensures the existence of periodic orbits but does not give information about the character of them. A first theoretical approach for solving this problem is to study the character of the equilibrium point $(\eta,\xi)=(0,0)$ for the $h$-averaged system and to consider the coplanar Hamiltonian as a perturbation with respect to $\alpha$ of $\HH_\AV$. Since in this work we deal with a realistic value of $\alpha=0.077$ (see~\eqref{def:alpha}), this perturbative approach is not useful in our case. However it gives an insight about the scenario we can encounter. In~\cite{AlessiBGGP23}, the following result about $\HH_\AV$ is proven.
\begin{theorem} \label{thm:averaged_results}
There exist two functions ${\Gamma}_{1}, {\Gamma}_{2}: [L_{\min},L_{\max}] \to \left (0,\frac{L}{2}\right ) $ such that, for $L\in [L_{\min}, L_{\max}]$:
\begin{itemize}
    \item If either $\Gamma\in (0, \Gamma_1(L) )$ or $\Gamma\in \left (\Gamma_2(L), \frac{L}{2} \right )$, then $(0,0)$ is a center of $\HH_\AV$. 
    \item If $\Gamma\in (\Gamma_1(L),\Gamma_2(L))$, then $(0,0)$ is a saddle of $\HH_\AV$.
    \item If $\Gamma=\Gamma_1(L)$ or $\Gamma=\Gamma_2(L)$, the origin is a degenerated equilibrium point of $\HH_\AV$. 
\end{itemize}
In addition, $L^{-1} \Gamma_1(L) \in \left (0,\frac{\sl}{2}\right )$ and $L^{-1} \Gamma_2(L) \in \left (\frac{\sl}{2},\frac{1}{2}\right )$, with $\sl:=\frac{-4+\sqrt{21}}{5}$ (i.e. the slope of the prograde resonance line in~\eqref{def:resonance}).
 \end{theorem}

For a given $L$, the values of $\Gamma_{1,2}$ can be numerically computed (as the zeroes of some appropriate function). For instance, for the case of Galileo (i.e., $L=1$), 
\[
\Gamma_1(1)=0.029613649805289, \qquad \Gamma_2(1)=0.084971418151141,
\]
and the corresponding energy levels (for the $h$-averaged system) are
$$
\mathbf{E}_1^{\AV}(1)=2.072230388690642 \cdot 10^{-6},\qquad  \text{and} \qquad \mathbf{E}_2^{\AV}(2)=-3.473759155836634\cdot 10^{-7},$$ respectively. Therefore, for $L=1$ and $E\in (\mathbf{E}_2^{\AV}(1),\mathbf{E}_1^{\AV}(1))$, as a consequence of Theorem~\ref{thm:averaged_results}, the origin is a saddle point of the $h$-averaged system.

Applying perturbation techniques, it can be seen (see~\cite{AlessiBGGP23}) that if $\alpha^3$ is small enough, the circular periodic orbits for the coplanar Hamiltonian, $\HH_\CP$, stated in Theorem~\ref{thm:coplanar_periodic_orbits}, are of saddle type if they belong to certain energy levels close enough to those of $\HH_\AV(0,\Gamma,0)$ with $\Gamma\in (\Gamma_1(L),\Gamma_2(L))$. However, as expected, there is no quantitative information about the maximum value of $\alpha$ such that this result can be applied.  
 For this reason, we have decided to consider the coplanar (non-averaged) model to build the Arnold diffusion mechanism. This is important to have a more realistic numerical computation of the hyperbolic periodic orbits and the corresponding manifolds (see Section~\ref{sec:periodicorbits:coolanar}).

\subsection{Normally hyperbolic invariant manifolds and scattering maps}\label{sec:mathtools}

We devote this section to explain the main tools that we use to construct the ``instability paths'' along which the eccentricity of the satellite drifts. These are normally hyperbolic invariant manifolds, the homoclinic intersections of their invariant manifolds and the associated scattering maps.

Such objects can be defined for maps or flows. We will use them both for maps and flows in Sections \ref{sec:coplanar} and \ref{sec:fullham} below. Here we provide the definitions for flows. The ones for maps are analogous.

In this section we denote by $M$ a $\mathcal{C}^r$ smooth manifold, by $X \in \mathcal{C}^r(M, TM)$ a vector field on $M$ and by $\varphi^t : M \to M$ the associated smooth flow. Let $\Lambda \subset M$ be a compact $\varphi^t$-invariant submanifold, possibly with boundary. By $\varphi^t$-invariant we mean that $X$ is tangent to $\Lambda$, but that orbits can escape through the boundary (a concept sometimes referred to as \emph{local} or \emph{weak} invariance). 

\begin{definition}\label{def:nhim}
We call $\Lambda$ a \emph{normally hyperbolic invariant manifold} for $\varphi^t$ if there is $0 < \nu < \vartheta^{-1}$, a positive constant $C$ and an invariant splitting of the tangent bundle
\begin{equation*}
T_{\Lambda} M = T \Lambda \oplus E^{\sta} \oplus E^{\unst}
\end{equation*}
with ${\displaystyle T\Lambda= \bigcup_{P\in \Lambda} \big (\{P\} \times T_P \Lambda\big )}$, ${\displaystyle E^{\sta,\unst} = \bigcup_{P\in \Lambda} \big (\{P\}\times E_P^{\sta,\unst} \big )}$ such that, for all $P \in \Lambda$,
\begin{equation*} 
\def\arraystretch{1.5}
\begin{array}{ll}
\| D \varphi^t (P) v \| \leq C \vartheta^{| t |} & \mbox{ for all } t \in \mathbb{R}, v \in T_P M, \\ 
\| D \varphi^t (P) v \| \leq C \nu^t & \mbox{ for all } t \geq 0, v \in E_P^{\sta},\\
\| D \varphi^t (P) v \| \leq C \nu^{-t} & \mbox{ for all } t \leq 0, v \in E_P^{\unst}.
\end{array}
\end{equation*}

Moreover, $\Lambda$ is called an $r$-\emph{normally hyperbolic invariant manifold} if it is $\mathcal{C}^r$ smooth  and
\begin{equation}\label{eq_largespectralgap}
0 < \nu < \vartheta^{-r} < 1 \qquad \text{for}\qquad r \geq 1.
\end{equation}
\end{definition}

Fenichel Theory \cite{fenichel1971persistence,fenichel1974asymptotic,fenichel1977asymptotic} ensures that normally hyperbolic invariant manifolds are persistent under perturbations. Moreover, they possess stable and unstable invariant manifolds $W^{\sta,\unst} (\Lambda)\subset M$ defined as follows. The local stable manifold $W^{\sta}_{\mathrm{loc}}(\Lambda)$ is the set of points in a small neighbourhood of $\Lambda$ whose forward orbits never leave the neighbourhood, and tend with exponential rate 
to $\Lambda$. The local unstable manifold $W^{\unst}_{\mathrm{loc}}(\Lambda)$ is the set of points in the neighbourhood whose backward orbits stay in the neighbourhood and tend exponentially to $\Lambda$.  We then define  the stable and unstable manifold of $\Lambda$ as
\begin{equation*}
W^{\sta}(\Lambda) = \bigcup_{t \geq 0}^{\infty} \varphi^{-t} \left( W^{\sta}_{\mathrm{loc}}(\Lambda) \right), \quad W^{\unst}(\Lambda) = \bigcup_{t \geq 0}^{\infty} \varphi^{t} \left( W^{\unst}_{\mathrm{loc}}(\Lambda) \right).
\end{equation*}
The stable and unstable manifolds of $\Lambda$ are foliated by what is usually called the strong stable and strong unstable foliations, the leaves of which we denote by $W^{\sta,\unst}(P)$ for $P \in \Lambda$. For each $P\in \Lambda$, the leaf $W^{\sta}(P)$ (resp. $W^{\unst}(P)$) of the strong stable foliation is tangent at $P$ to $E^{\sta}_P$ (resp. $E^{\unst}_P$). Moreover, the foliations satisfy that $\varphi^t \left( W^{\sta} (P) \right) = W^{\sta} \left( \varphi^t (P) \right)$ and $\varphi^t \left( W^{\unst} (P) \right) = W^{\unst} \left( \varphi^t (P) \right)$ for each $P\in \Lambda$ and $t \in \mathbb{R}$. 
 
Then, one can define  the usually called \emph{wave maps} $\pi^{\sta,\unst} : W^{\sta,\unst} (\Lambda) \to \Lambda$ to be projections along leaves of the strong stable and strong unstable foliations. That is to say, if $Q \in W^{\sta} (\Lambda)$ then there is a unique $Q_+ \in \Lambda$ such that $Q \in W^{\sta}(Q_+)$, and so $\pi^{\sta}(Q)=Q_+$. Similarly, if $Q \in W^{\unst} (\Lambda)$ then there is a unique $Q_- \in \Lambda$ such that $Q \in W^{\unst}(Q_-)$, in which case $\pi^{\unst}(Q)=Q_-$. The points $Q_\pm$ satisfy that
\[
\lim_{t \to \pm \infty} \mathrm{dist}(\varphi^t(Q_{\pm}),\varphi^t(Q)) = 0.
\]
Now, suppose that $Q \in \left(W^{\sta}(\Lambda) \pitchfork W^{\unst} (\Lambda)\right) \setminus \Lambda$ is a transverse homoclinic point such that $Q \in W^{\sta}(Q_+) \cap W^{\unst}(Q_-)$. We say that the homoclinic intersection at $Q$ is \emph{strongly transverse} if
\begin{equation}
\begin{split} \label{eq_strongtransversality}
T_Q W^{\sta} (Q_+) \oplus T_Q \left( W^{\sta}(\Lambda) \cap W^{\unst}(\Lambda) \right) = T_Q W^{\sta} (\Lambda), \\
T_Q W^{\unst} (Q_-) \oplus T_Q \left( W^{\sta}(\Lambda) \cap W^{\unst}(\Lambda) \right) = T_Q W^{\unst} (\Lambda).
\end{split}
\end{equation}
In this case, in a small enough neighborhood $\Xi$ of $Q$ in $W^{\sta}(\Lambda) \cap W^{\unst}(\Lambda)$, \eqref{eq_strongtransversality} holds at each point of $\Xi$, and the restrictions to $\Xi$ of the holonomy maps, $\pi^{\sta,\unst}$, are bijections onto their images. We call $\Xi$ a \emph{homoclinic channel}. In such domain, following \cite{delshams2008geometric}, we define the scattering map as follows.


\begin{definition}\label{def:scattering}
Let $Q_-\in \pi^{\unst} \left( \Xi \right)$, let $Q = \left(\left. \pi^{\unst} \right|_{\Xi} \right)^{-1} (Q_-)$, and let $Q_+ = \pi^{\sta}(Q)$. The \emph{scattering map} $S : \pi^{\unst} (\Xi) \to \pi^{\sta} (\Xi)$ is defined by
\begin{equation*}
S = \pi^{\sta} \circ \left( \pi^{\unst} \right)^{-1} : Q_- \longmapsto Q_+.
\end{equation*}
\end{definition}


Assume that both $M$ (the manifold) and $X$ (the vector field) are $\mathcal{C}^r$ with $r\geq 2$ and that \eqref{eq_largespectralgap} is satisfied, so that $\Lambda$ is also $\mathcal{C}^r$. Then, the scattering map $S$ is $\mathcal{C}^{r-1}$ (see~\cite{delshams2008geometric}). Moreover, if the vector field has a $\CCC^r$-dependence on some parameters, then  the scattering map depends on a $\CCC^{r-1}$ on them.

In general, the scattering map is only locally  defined, as the transverse homoclinic intersection of stable and unstable manifolds can be very complicated. In the present paper, we are able to describe quite precisely the domains of the scattering maps that we consider.


\subsection{The Arnold diffusion instability mechanism}\label{sec:Arnold} 


Once we have introduced the hierarchy of models in Section  \ref{section:hierarchy} and the dynamical systems tools in Section \ref{sec:mathtools}, we are ready to explain the instability mechanism that we consider to achieve drift in the eccentricity of the satellite. Such mechanism fits into what are usually called \emph{a priory chaotic} Hamiltonian systems in Arnold diffusion literature \cite{DelshamsLS00, Piftankin06,fejoz2016kirkwood}.

Consider first the coplanar Hamiltonian $\HH_{\CP}$ in~\eqref{def:Ham:Coplanar}, namely, $i_\rM=0$. Since $\HH_\CP$ is autonomous, its energy is a first integral, and therefore the coordinate $J$  as well for the Hamiltonian $\KK_{\CP}$ in~\eqref{def:KHam:CP} (in the extended phase space). In Section \ref{sec:coplanar} we will see that at each energy level (in a certain interval), the Hamiltonian $\HH_\CP$ has a periodic orbit  at the $2g+h$-resonance. 
In Poincar\'e variables, this periodic orbit is located at $\{\eta=\xi=0\}$. Moreover, in the considered energy interval, these periodic orbits are hyperbolic and, thus, they have stable and unstable invariant manifolds. In addition, we check numerically that they intersect transversally within each energy level. The obtained homoclinic orbit has a range of eccentricity between $e=0$ (the periodic orbit) and certain $e=e_\mathrm{\max}$ which depends on the energy level\footnote{By energy level we mean with respect to  the coplanar secular Hamiltonian $\mathcal{H}_{\CP}$.}. For energy levels realistic for Galileo $e_\mathrm{\max}\simeq 0.35$ (in particular, its orbit does not enter the Earth's atmosphere) whereas for higher energies $e_\mathrm{\max}$ reaches 0.78 (which implies that the osculating ellipse of the satellite hits the Earth atmosphere) or higher.

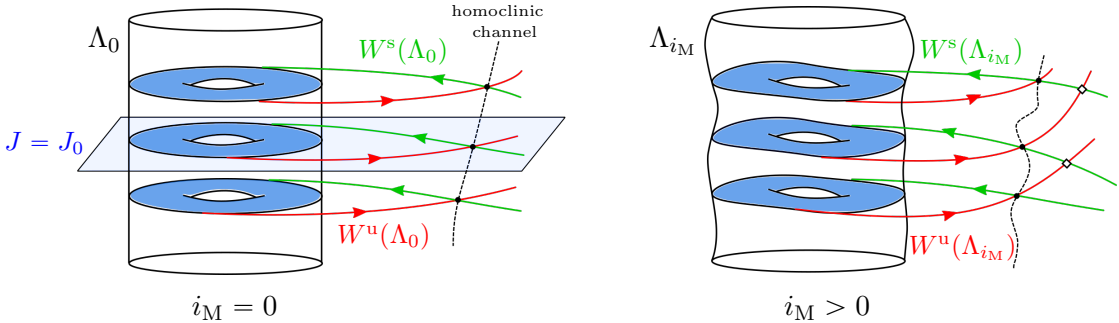
\begin{figure}[ht]
    \centering
    \vspace{5mm}
    \begin{overpic}[scale=0.7]{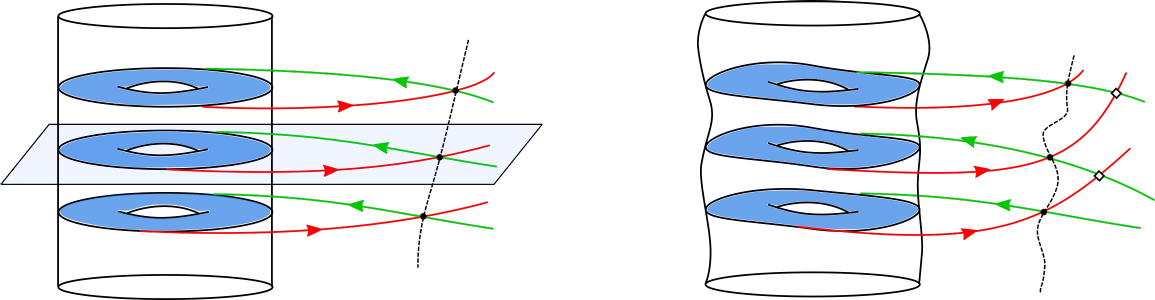}
    	\put(11,-4){$\iM=0$}
    	\put(1,22){\color{black}$\Lambda_0$}
    	\put(27,21){\footnotesize\color{myGreen} \small $W^{\sta}(\Lambda_0)$}
    	\put(-7,12){\color{blue}\small $J=J_0$}
    	\put(25,3){\color{red} \footnotesize \small $W^{\unst}(\Lambda_0)$}
    	\put(36,25){\tiny homoclinic}
    	\put(38,23){\tiny channel}
    	\put(68,-4){$\iM>0$}
    	\put(55,22){\color{black}$\Lambda_{\iM}$}
    	\put(81,21){\color{myGreen}\footnotesize \small $W^{\sta}(\Lambda_{\iM})$}
    	\put(80,2){\color{red}\footnotesize \small $W^{\unst}(\Lambda_{\iM})$}
    \end{overpic}
	\vspace{5mm}
    \caption{The normally hyperbolic invariant cylinder and its invariant manifolds. In the left figure $\iM=0$ and therefore $J$ is a first integral. The cylinder is foliated by invariant tori whose invariant manifolds intersect transversally within $J=\text{constant}$ giving rise to a homoclinic channel. In the right figure, $\iM>0$ (small enough). Then, $J$ is not a first integral anymore and one can construct heteroclinic connections between different tori in the cylinder (diamond intersections in the picture).
    See Corollary~\ref{corollary:NHIMiM0} and Theorem~\ref{theorem:invariantObjectsPerturbed}, respectively, for a detailed description of the notation. 
    }
    \label{fig:AddingPerturbation}
\end{figure}

In the extended phase space (i.e., adding the pair $(J,\Omega_\rM)$ and considering the Hamiltonian $\KK_\CP$ in \eqref{def:KHam:CP}) these periodic orbits become 2-dimensional tori. If one considers the union of such tori for an energy interval, one has a normally hyperbolic invariant cylinder (see Definition \ref{def:nhim}). Analogously, the union of the stable and unstable manifolds of the periodic orbits (now tori) constitute the invariant manifolds of the cylinder and their transverse intersections give rise to homoclinic channels (see Figure~\ref{fig:AddingPerturbation}-left). This allows to define scattering maps (in suitable domains, detailed in Section~\ref{sec:coplanar}).

The fact that $J$ is constant   when $i_\rM=0$, implies that solutions close to the invariant manifolds of the cylinder have a constrained drift in eccentricity: it ranges from $e=0$ (close to the cylinder) to close to the already mentioned $e_\mathrm{max}$ (close to  the ``first'' homoclinic point), which depends on $J$.

When $i_\rM>0$ but still small enough the dynamics changes drastically. Indeed, in Section \ref{sec:fullham}, we construct solutions of the full Hamiltonian $\KK$ in~\eqref{def:HamExtPhsSp} such that the action $J$ undergoes big changes. Equivalently, we construct solutions of the Hamiltonian $\HH$ in~\eqref{def:Ham:H:Poinc} with energy drift. These solutions travel along the resonance \eqref{def:resonance} and therefore, while increasing energy along the resonance, they perform larger homoclinic excursions which lead to  larger oscillations  in eccentricity.

The achievement of this drift is given by  an Arnold diffusion mechanism, which relies on the normally hyperbolic invariant manifold and the associated invariant manifolds and scattering maps (see Section \ref{sec:mathtools}).
Indeed, classical Fenichel Theory (see \cite{fenichel1971persistence}) implies that the normally hyperbolic invariant cylinder of the Hamiltonian $\HH_\CP$ in \eqref{def:Ham:Coplanar} persists for $i_\rM$ small enough and the same happens for the associated stable and unstable invariant manifolds (see Figure~\ref{fig:AddingPerturbation}-right). Moreover, all the objects are regular with respect to $i_\rM$. This implies that the transverse intersections present in the case $i_\rM=0$ are also persistent and the associated scattering maps are also smooth with respect to $i_\rM$.

We use the scattering maps to construct what is usually called a pseudo-orbit. That is, a sequence of concatenated ``pieces of orbits''. More precisely, this pseudo-orbit is formed by pieces of orbits in the cylinder (that is, pieces of orbits of what is usually called the \emph{inner dynamics}) which are connected by heteroclinic orbits (see Figure \ref{fig:Shadowing}-left). Recall that a point in the cylinder $Q_-$ is mapped to a point $Q_+$ by the scattering map if there is a heteroclinic orbit of the Hamiltonian \eqref{def:HamExtPhsSp} between them. Thus, to construct a pseudo-orbit one has to understand how to ``combine'' the inner dynamics (the flow restricted to the cylinder) and the scattering map (in occasions also called outer map) to achieve drift in  $J$.

 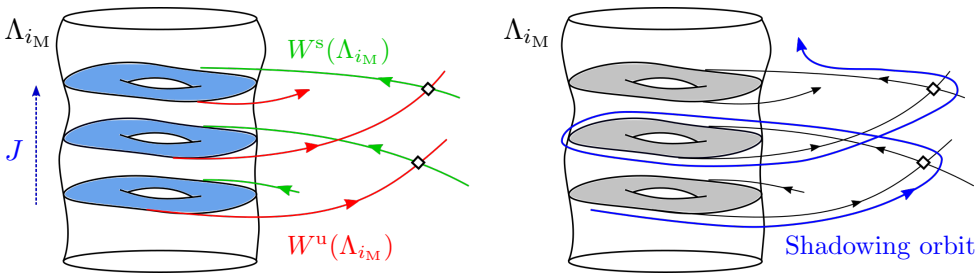
\begin{figure}[ht]
 	\centering
    \vspace{5mm}
	\begin{overpic}[scale=0.7]{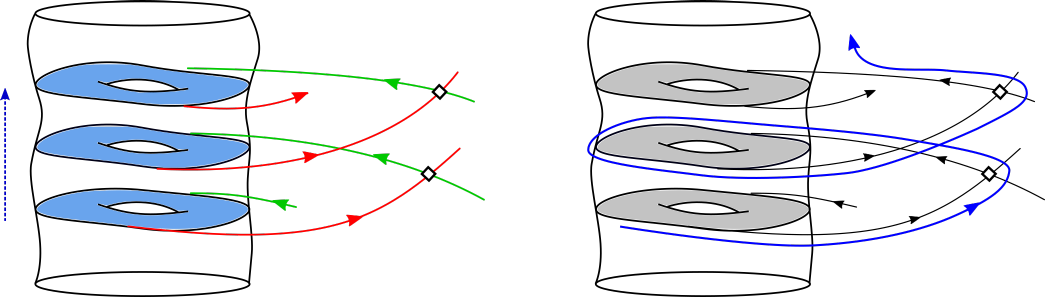}
 		\put(-3,12){\color{blue} $J$}
        \put(-3,25){\color{black}$\Lambda_{\iM}$}
    	\put(27,23){\color{myGreen}\footnotesize \small $W^{\sta}(\Lambda_{\iM})$}
    	\put(27,2){\color{red}\footnotesize \small $W^{\unst}(\Lambda_{\iM})$}
        \put(49,25){ $\Lambda_{\iM}$}
        \put(80,2){\color{blue}\small Shadowing orbit}
 	\end{overpic}
 	\caption{The pseudo-orbit (left) and the shadowing orbit (right): The blue orbit shadows (follows closely) the pseudo-orbit formed by heteroclinic orbits connecting different points in the cylinder and pieces of cylinder orbits.}
 	\label{fig:Shadowing}
 \end{figure}

Both the inner and scattering maps are computed numerically relying on perturbative methods (for $i_\rM>0$ small enough). Thanks to the particular form of the Hamiltonian~\eqref{def:HamExtPhsSp}, these maps have a rather simple form. Indeed, note that one can reduce the dimension of the model by considering a Poincar\'e map associated to the section $\{h=0\}$ and by eliminating the variable $\Gamma$ through energy conservation. Then, the cylinder becomes two dimensional and a good system of coordinates for it is $(J,\Omega_\rM)$. In these coordinates, the inner map is of the form
\begin{equation}
\label{def:innerMapPrevious}
	\FF^{\inner}_{\iM}: \begin{pmatrix}
    J 
  \\ \OmegaM 
	\end{pmatrix}
	\mapsto
	\begin{pmatrix}
	J + \iM \left(A_{1}^+(J)e^{i\OmegaM} +  A_{1}^-(J)e^{-i\OmegaM}\right)    + \OO(\iM^2)
  \\ 
  \OmegaM + n_{\OmegaM} \TTT_0(J)
   + \OO(\iM)
	\end{pmatrix},
\end{equation}
and one can define two scattering maps (denoted as primary and secondary) of the form 
\begin{equation}
\label{def:outerMapPrevious}
	\FF^{\outer,*}_{\iM}: \begin{pmatrix}
    J 
  \\ \OmegaM 
	\end{pmatrix}
	\mapsto
	\begin{pmatrix}
	J + \iM ( B_{1}^{*,+}(J)e^{i\OmegaM} + 
 B_{1}^{*,-}(J) e^{-i\OmegaM}) + \OO(\iM^2) 
  \\ 
  \OmegaM + n_{\OmegaM} \zeta^*(J)  + \OO(\iM)
	\end{pmatrix}, 
\end{equation}   
for $*\in\{\prim,\secn\}$, associated to two different homoclinic channels.
We compute numerically the functions appearing in the first orders of these maps. 
Those of the scattering maps, $\FF^{\outer,*}_{\iM}$, are given by Melnikov-like integrals (similar so those in \cite{DelshamsLS00, fejoz2016secular}).

Note that the scattering maps are not globally defined since the invariant manifolds of the cylinder may have tangencies. This is the reason why we must use two scattering maps: the union of domains of the two scattering maps contain the whole region of the cylinder we are interested in.

Then, we construct  a drifting pseudo-orbit. Note that we cannot use  now the  classical results~\cite{moeckel2002drift,Calvez07}, which ensure that such pseudo-orbits exist  provided the inner and scattering  maps dynamics share no common invariant curves, since they require globally defined maps (which is not the case for the scattering maps in the present paper). Instead, we construct it from a ``generalized transition chain'', that is, a sequence of invariant quasi-periodic tori of either the inner map or one of the scattering maps, which are connected by an iteration of the other map. This is done in Section~\ref{sec:fullham}.


Once the pseudo-orbit is constructed, the final step is to show that there is a true orbit of the Hamiltonian \eqref{def:HamExtPhsSp} which ``shadows'' (i.e. follows closely) the pseudo-orbit (see Figure \ref{fig:Shadowing}-right). For this last step we rely on shadowing results developed in \cite{GideaLS20}. These results are rather flexible and can be easily applied to our setting. Unfortunately, they do not provide time estimates. 

\begin{remark}\label{rmk:time}
Instead of \cite{GideaLS20}, we could have  relied on more quantitative shadowing arguments (see for instance \cite{treschev02, Piftankin06, gidea2006topological,ClarkeFG23}) which could provide time estimates. However, they  require a more precise knowledge of the inner dynamics and, therefore, to keep the proof simple, we do not use them. Still, they should be applicable also to our setting and should lead to instability times $T=T(\iM)$  satisfying
\[
|T|\leq \frac{C}{\iM},
\]
for some constant $C>0$ independent of $\iM>0$. 
\end{remark}




\section{The coplanar secular Hamiltonian}\label{sec:coplanar}

In this section, we analyze the dynamics of the Hamiltonian $\HH_\CP$ (see~\eqref{def:Ham:Coplanar}). That is, we consider the secular Hamiltonian with no inclination of the Moon with respect to the ecliptic plane (i.e., $\iM=0$). 
In particular, we analyze certain features of its dynamics and then we ``translate'' them to the Hamiltonian $\KK_\CP$ (see~\eqref{def:KHam:CP}) in the  extended phase space.


Since we consider $\al=0.077$ as a fixed parameter\footnote{Instead of taking it ``small enough''.} (see~\eqref{def:alpha}), the application of perturbative  techniques to analyze $\HH_{\CP}$ is not possible. 
Instead, in the following sections, we assume certain ans\"atze on the Hamiltonian and verify them numerically.

\subsection{Periodic orbits and their invariant manifolds}\label{sec:periodicorbits:coolanar}

Concerning the secular coplanar Hamiltonian in~\eqref{def:Ham:Coplanar} we assume the following ansatz. It concerns the existence of hyperbolic periodic orbits at circular motions and the transverse intersections of their stable and unstable invariant manifolds.

As we claimed in Section~\ref{sec:theoretical_averaged_coplanar}, the plane $\{\eta=\xi=0\}$ is invariant by the coplanar Hamiltonian. The ansatz is related to the (circular) periodic orbits lying on $\{\eta=\xi=0\}$. 


\begin{ansatz}\label{ansatz:periodic}
	The  Hamiltonian $\HH_{\CP}(\eta,\Gam,\xi,h)$ given in \eqref{def:Ham:Coplanar} satisfies the following statements.
	\begin{enumerate}
		%
  \item In every energy level $E\in [E_{\min},E_{\max}]=[-2.12\cdot 10^{-7},1.36\cdot 10^{-6}]$, the periodic orbit ${\PP}_E(t)=(0,\Gam_E(t),0,h_E(t))$ provided in Theorem~\ref{thm:coplanar_periodic_orbits} is hyperbolic. Denoting $\TTT(\PP_E)$ its period, we have that 
		\[
		n_{\OmegaM} \TTT(\PP_E) \in [3.9\pi,4.15\pi],
		\qquad
		\HH_{\CP}(\PP_E(t)) = E.
		\]
		Both the periodic orbit and the period depend smoothly on $E$.
		In addition, $\TTT(\PP_E)$ is a strictly increasing function with respect to $E$.

		%

   \item For each $E\in [E_{\min},E_{\max}]$, the invariant manifolds $W^{\unst}(\PP_E)$ and $W^{\sta}(\PP_E)$ intersect transversally.
  Let $\mathcal{Q}_E(t)$ be any of these homoclinic orbits.

  \item For $E\in [E_{\min},E_{\max}]$, the periodic orbit $\PP_E(t)$ and the homoclinic orbit $\chi_E(t)$ satisfy
  \[
  i(\mathcal{Q}_E(t)), i(\PP_E(t))\in [55.70\degre , 58.18\degre],\qquad \text{for all }\,t\in\RR,
  \]
  and 
  \[
 \max_{t\in\RR}e(\mathcal{Q}_E(t)) \geq 0.795,
  \]
  where $i(\cdot)$ and $e(\cdot)$ denote the osculating inclination and eccentricity respectively.
	\end{enumerate}
\end{ansatz}


Note that this ansatz assumes the existence of transverse homoclinic orbits to the periodic orbits but does not provide information on how these homoclinic orbits depend on $E$. In fact, we will have to consider the intersections between different branches of the invariant manifolds depending on the energy levels to avoid homoclinic tangencies.

We devote the next section to verify numerically this ansatz. 

%

\subsubsection{Numerical verification of Ansatz \ref{ansatz:periodic}}

As we claimed in Section~\ref{sec:theoretical_averaged_coplanar}, the existence of the circular periodic orbits located at $\{\eta=\xi=0\}$ is guaranteed by Theorem~\ref{thm:coplanar_periodic_orbits}. Considering non-dimensional units (see  Remark~\ref{rmk:unities}), some examples of these periodic orbits computed for the system given by $\HH_{\CP}$ are shown in Figure~\ref{fig:PO_example}. Note that Galileo corresponds to $\HH_{\CP}=1.7\cdot  10^{-8}$ (assuming $i=56.06\degre)$.

\begin{figure}[ht]
\begin{center}
\includegraphics[width=10cm]{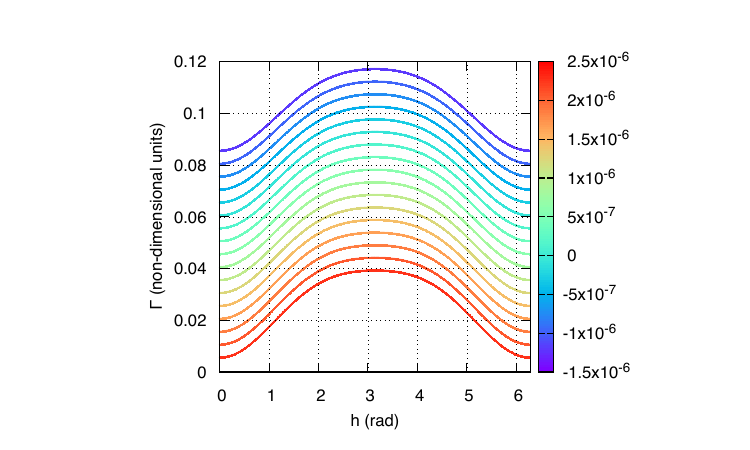}%
\end{center}
\caption{Examples of periodic orbits for $i_M=0$, non-dimensional units. The colorbar reports the value of ${\HH}_{\CP}$. Here it is shown the behavior of a range of energies larger than the one we are interested in.}
\label{fig:PO_example}
\end{figure}

Concerning the stability, the eigenvalues of the monodromy matrix of a given periodic orbit are of the type $\big \{1,1,e^{\TTT\lambda},e^{-\TTT\lambda} \big \}$, with $\TTT$ the period of the given orbit, in case of orbits with hyperbolic nature, or $\big \{1,1,i\nu,-i\nu\big \}$ in case of  orbits with elliptic nature. We will consider only the former case. In Figure~\ref{fig:hamlam}, we show the behavior of the eigenvalue greater than 1 for the hyperbolic periodic orbits, as a function of ${\HH}_{\CP}$.  

\begin{figure}[ht]
	\begin{center}
		\includegraphics[width=9cm]{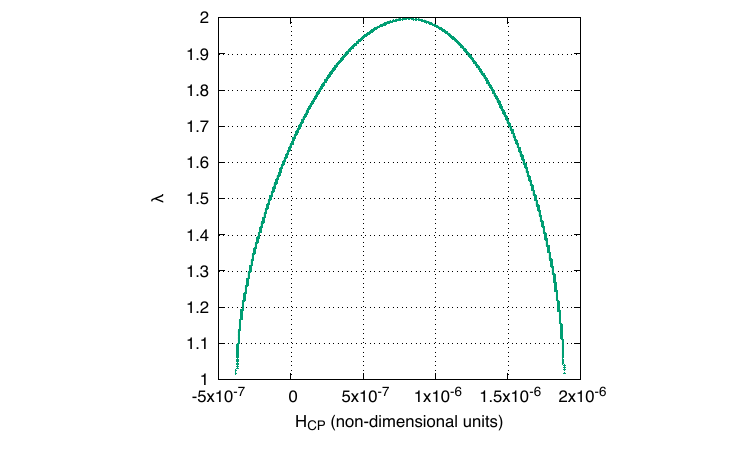} 
	\end{center}
	\caption{The value of the eigenvalue greater than 1, $\lambda$, as a function of the ${\HH}_{\CP}$ for the hyperbolic periodic orbits.} 
	\label{fig:hamlam}
\end{figure}

In Figure~\ref{fig:hamperiod}, we show the period of the hyperbolic periodic orbits and the value of $n_{\OmegaM} \TTT$ as a function of ${\HH}_{\CP}$.
Recall that $\OmegaM(t) = \Omega_{\rM,0} + n_{\OmegaM}t$ (see~\eqref{eq:OmegaM}), therefore the value $n_{\OmegaM} \TTT$ indicates the ratio between the periods of the variables $h$ and $\OmegaM$.
Notice that at ${\HH}_{\CP}\approx 4.4472 \cdot  10^{-7} \in [E_{\min},E_{\max}]$ (non-dimensional units), the period is such that $n_{\OmegaM} \TTT=4\pi$ which corresponds to the resonance $2h-\OmegaM$ (that is, $\TTT = 2\TTT_{\OmegaM}$).
This double resonance was already observed in~\cite{Daquin2022} and will be commented further at the end of the paper.

\begin{figure}[h!]
	\begin{center}
		\includegraphics[width=6cm]{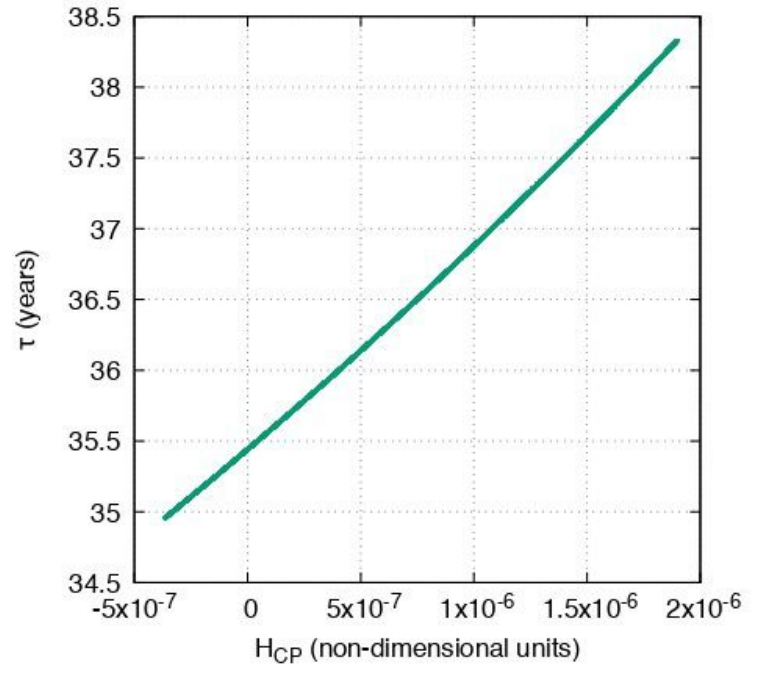} 
		\includegraphics[width=6cm]{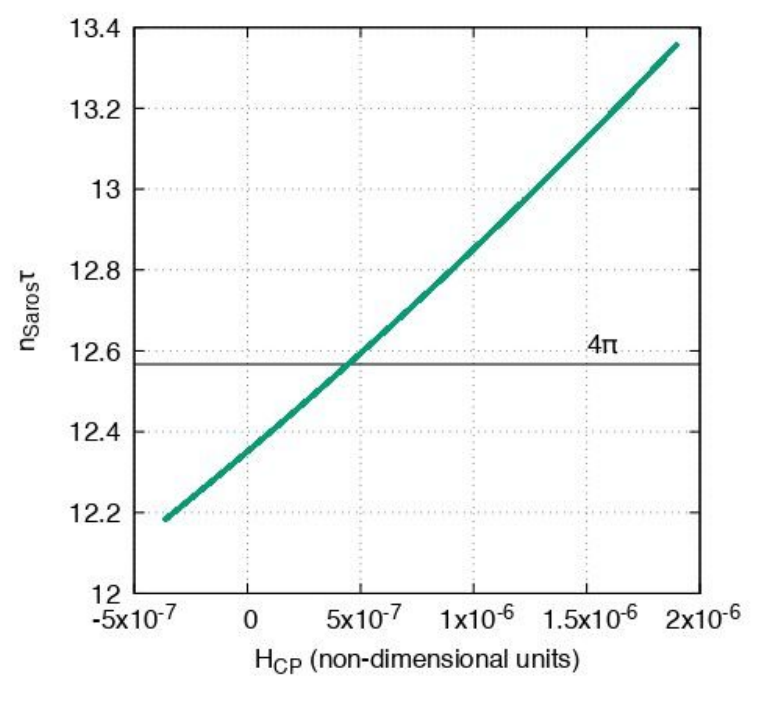}
	\end{center}
	\caption{The period (left) $ \TTT$ of the hyperbolic periodic orbits as a function of ${\HH}_{\CP}$. On the right, $n_{\Saros} \TTT$ (where $n_{\Saros}\equiv n_{\OmegaM}$)  as a function of ${\HH}_{\CP}$. }
	\label{fig:hamperiod}
\end{figure}

Next, we analyze the dynamics of the stable and the unstable manifolds of the circular periodic orbits and we look for transverse intersections between them at each energy level (Item 2 in the ansatz). 
%
%
Note that the Hamiltonian  $\HH_{\CP}$ is reversible with respect to the  involutions
\begin{equation*}
\label{def:symmetryPoincare}
	\Phi^{\hor}(\eta,\Gam,\xi,h) = (-\eta,\Gam,\xi,-h)
 \quad \text{and} \quad
 {\Phi}^{\ver}(\eta,\Gam,\xi,h) = (\eta,\Gam,-\xi,-h),
\end{equation*}
where $\hor$ and $\ver$ stand for horizontal and vertical respectively, referring to the projection of the symmetry axis onto the $(\xi,\eta)$ plane, which are $\{\eta=0, \, h=0 \}$, $\{\eta=0, \, h=\pi \}$ and  $\{\xi=0, \, h=0 \}$, $\{\xi=0, \, h=\pi \}$ respectively.
These symmetries  simplify the  numerical verification of the ansatz since one can easily see that the intersections between 
$W^{\unst}(\PP_E)$ and $W^{\sta}(\PP_E)$
have points at  these symmetry axes in all the considered energy levels. However, these intersections may not be transverse. Indeed, the angles between the invariant manifolds at the symmetry axes depend analytically on the energy and they may vanish for a discrete set of values of the energy.

In Figure~\ref{fig:man_example}, on the left, we show the behavior of the unstable invariant manifolds associated with the periodic orbit on the Poincar\'e section $\{h=0\}$, together with the corresponding eccentricity growth (colorbar). Each curve corresponds to a different energy levels. In the same figure on the right, we show the behavior of the invariant manifolds on the Poincar\'e section $\{h=0\}$ (purple) and  $\{h=\pi\}$ (green), for a same value of energy. 

\begin{figure}[h]
\begin{minipage}{0.5\linewidth}
\vspace{-0.3cm}
\includegraphics[width=9.5cm]{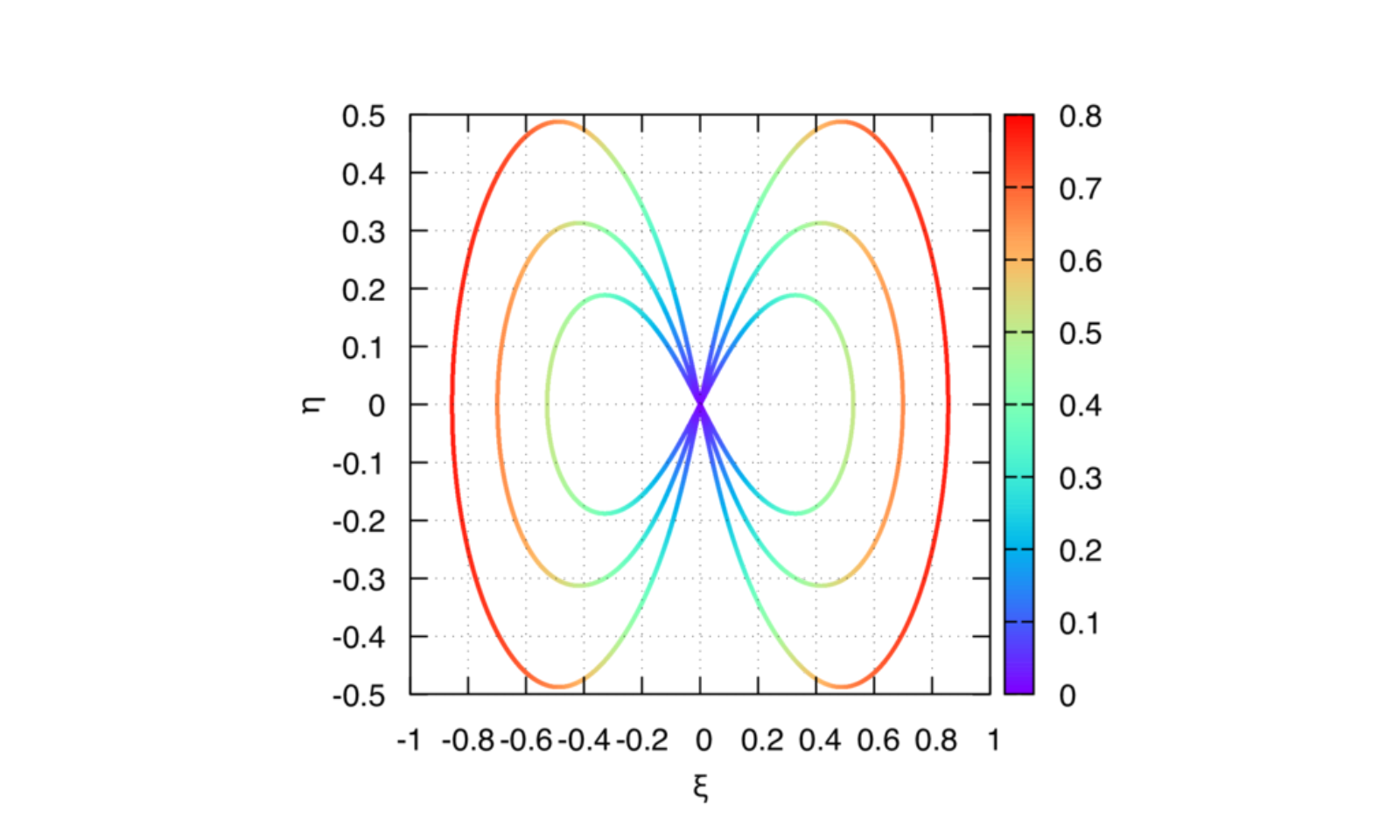}
\end{minipage}
\begin{minipage}{0.3\linewidth}
\includegraphics[width=8.2cm]{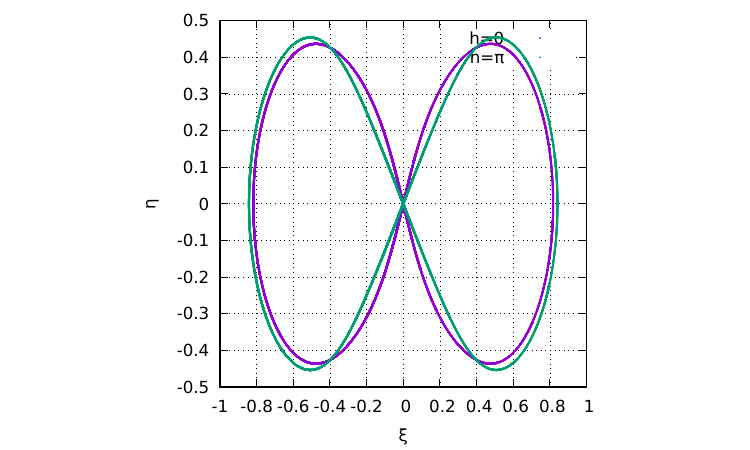}
\end{minipage}
\caption{Left: the unstable manifolds on the Poincar\'e section $\{h=0\}$ for different energy levels. The colorbar depicts the evolution of the eccentricity along  these unstable manifolds. Right: the behavior of the unstable manifolds computed assuming as Poincar\'e map $\{h=0\}$ (purple) and $\{h=\pi\}$ (green), for a given energy level. }
\label{fig:man_example}
\end{figure}
All the computations of the hyperbolic manifolds have been performed following the procedure explained in \cite{fejoz2016kirkwood}. 
The first intersection between the stable and the unstable manifolds takes place at $\eta=0$ ($\xi\neq 0$), after about 30 iterations on the Poincar\'e map of the fundamental domain assuming an initial displacement of $10^{-8}$ in non-dimensional units from the periodic orbit along the eigendirections. This choice corresponds to an error of $10^{-11}$ following \cite{fejoz2016kirkwood}.

To verify Item 2 in the ansatz we must show that, at every energy level, the circular periodic orbit has a transverse homoclinic. We first consider the homoclinic point at $\{\eta=0, \, h=0 \}$ with $\xi>0$. The corresponding value of $\xi$ and the maximum eccentricity growth achieved are shown in Figure~\ref{fig:xi0_ecmax}.
The invariant manifolds are transverse at these  points for all energies within the considered energy range except for  6 values. (These values will be showed later in the first column of Table \ref{tab:reldiff}, see also Figure~\ref{fig:splittingh0hpi}).

\begin{figure}[ht]
\begin{minipage}{0.45\linewidth}
\includegraphics[width=8.2cm]{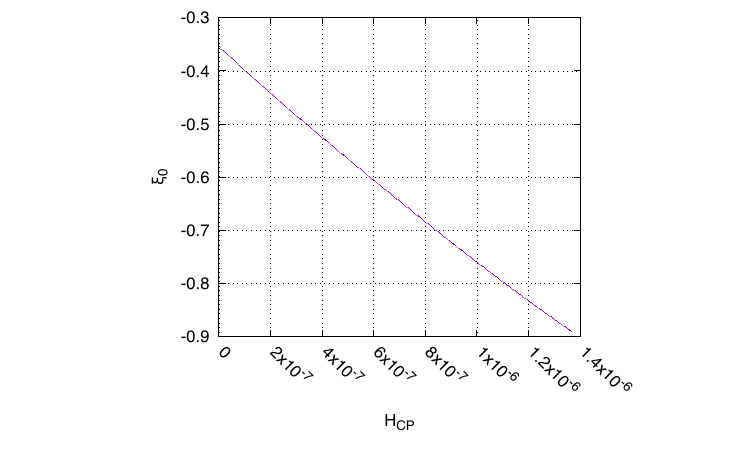}
\end{minipage}
\begin{minipage}{0.45\linewidth}
\includegraphics[width=8.2cm]{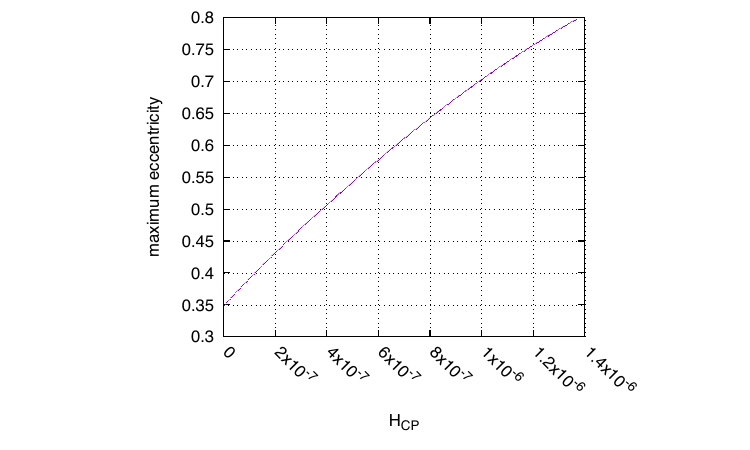}
\end{minipage}
 \caption{Left: first point of intersection $\xi_0$ between negative branch of the stable and the unstable manifolds of a given periodic orbit at the crossing with the $\{\eta=0,\, h=0\}$ axis, as a function of ${\HH}_{\CP}$ (non-dimensional units). Right: the corresponding maximum eccentricity achievable along the hyperbolic manifolds as a function of ${\HH}_{\CP}$ (non-dimensional units).}
 \label{fig:xi0_ecmax}
\end{figure}

For these 6 energy levels we must look for other homoclinic points and check the corresponding transversality. Note that one cannot consider the homoclinic points at  $\{\eta=0, \, h=0 \}$ with $\xi<0$ (see Figure \ref{fig:man_example}) since  they have the same tangencies as the first branch considered as they are $\Phi^\ver$-symmetric.  The first choice is to look for the homoclinic points at the symmetry axis $\{\eta=0, \, h= \pi \}$ (with either $\xi>0$ or $\xi<0$, since they are $\Phi^\ver$-symmetric). However, the energy values for which these homoclinic points develop tangencies are very close to those of the other branch considered. This is shown in Table \ref{tab:reldiff} and Figure~\ref{fig:splittingh0hpi}, where we show the splitting angle computed at the first point of intersection between the negative branch of the stable and the unstable manifolds at $\{\eta=0, \, h=0\}$ and $\{\eta=0, \, h=\pi\}$ as a function of the given energy level. One should expect that even if they look to be very close, the tangency energy values of the two branches are disjoint. However, to check this would require a numerical analysis with a much higher precision.

\begin{table}[t]
	\begin{center}
		\begin{tabular}{c||cc||c}
			\hline	
   $\HH_{\CP}$ & $\Delta \HH_{\CP}$ & $\Delta \HH_{\CP}/\HH_{\CP}$ & $\phi$\\
   \hline
    1.34294e-6 & 1.2e-11 & 8.9e-6 & 0.788\\
    1.23642e-6 & 1.5e-10 & 1.2e-4 & 0.997\\
    1.09175e-6 & 6.1e-12 & 5.6e-6 & 1.183\\
    8.9030e-7 & 6.6e-12 & 7.4e-6 & 1.276\\
    6.0662e-7 & 2.5e-11 & 4.1e-5 & 1.554\\
    2.1005e-7 & 7.5e-11 & 3.5e-4 & 1.935\\
    \hline
   \end{tabular}
   \vspace{1em}
   \caption{The value of  $\HH_{\CP}$ corresponding to a non-transverse intersection at $\{\eta=0,\, h=0\}$ and the absolute and relative difference with respect to the value of $\HH_{\CP}$ corresponding to a non-transverse intersection at $\{\eta=0,\, h=\pi\}$. On the last column, the splitting angle $\phi$ (rad) computed at $\{\xi=0,\, h=0\}$ at the same energy level. }
		\label{tab:reldiff}
   \end{center}
   \end{table}

\begin{figure}[ht]
\centering
		\includegraphics[width=12.cm]{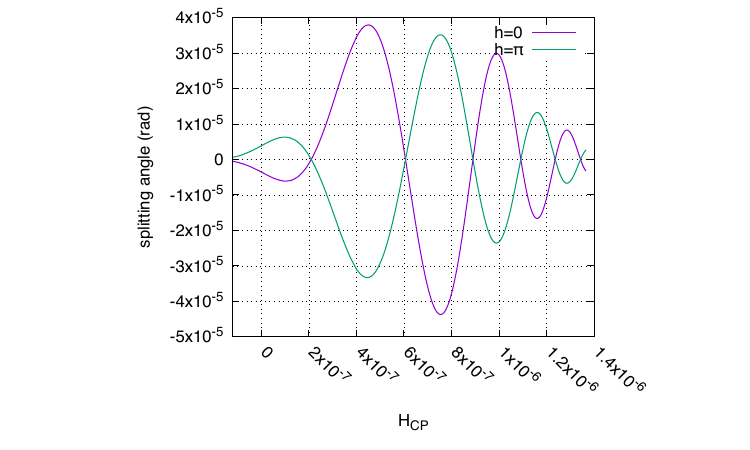}
 \caption{Splitting angle between the negative branch of the stable and the unstable manifolds computed at $\{\eta=0,\, h=0\}$ (purple) and $\{\eta=0,\, h=\pi\}$ axis (green), as a function of ${\HH}_{\CP}$ (non-dimensional units). The tangencies of both branches are very close to each other and  beyond the numerical accuracy that we consider. }
 \label{fig:splittingh0hpi}
\end{figure}

Instead, to keep the numerical analysis ``simple'', we consider a different branch: we look for homoclinic points at the symmetry axis $\{\xi=0, \, h= 0 \}$. We refer to these homoclinic points as secondary since their orbits approach twice the saddle, in contrast to the first branch, which we denote by primary.
The corresponding splitting angle $\phi$ is shown in Table~\ref{tab:reldiff} (last column) and an example of such intersection is given in Figure~\ref{fig:homoclinic_xi0}. 

\begin{figure}[h!]
\centering
		\includegraphics[width=12.cm]{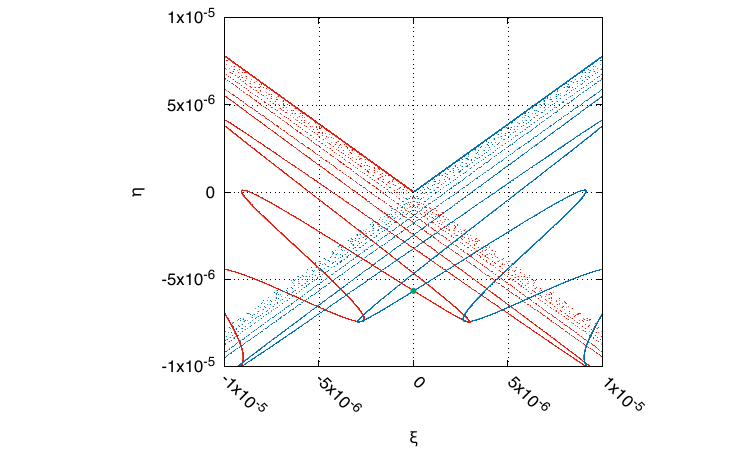}
 \caption{Homoclinic intersection (green point) computed at $\{\xi=0,\, h=0\}$ for ${\HH}_{\CP}=2.1005\cdot 10^{-7}$ (non-dimensional units). In red the negative branch of the unstable manifold, in blue the negative branch of the stable manifold.}
 \label{fig:homoclinic_xi0}
\end{figure}

Once we have computed the periodic orbits and the associated transverse homoclinics, Item 3 of the ansatz can be easily verified  by integrating numerically the flow of the coplanar system.

\subsection{The invariant cylinder and their invariant manifolds}

In this section we analyze how the periodic orbits provided by Ansatz~\ref{ansatz:periodic} (Item 1) give rise to a normally hyperbolic invariant cylinder and, by means of Item 2 of Ansatz~\ref{ansatz:periodic}, we analyze the associated stable and unstable invariant manifolds and their intersections.
Then, we study the inner dynamics (dynamics restricted to the cylinder) and the scattering maps  of the coplanar secular model (see Section \ref{sec:mathtools}).

Our aim is to study the complete problem as a perturbation of the coplanar one. 
To this end, we consider the Hamiltonian $\KK_{\CP}$ (see~\eqref{def:KHam:CP}) defined in the extended phase space $(\eta,\Gam,J,\xi,h,\OmegaM)$.
In other words, we keep the conjugated variables $(J,\OmegaM)$ even if $\OmegaM$ is a cyclic variable (see \eqref{eq:OmegaM}).
As a result, the periodic orbits considered in Ansatz~\ref{ansatz:periodic} become $2$-dimensional invariant tori. 
Indeed, we first notice that the energy level $\KK_{\CP}=0$ in the extended phase space corresponds to $\HH_{\CP}=E$ with $E=-n_{\OmegaM}J$. In addition, since $\dot{h}_E(t) \neq 0$, $t=h_{E}^{-1}(h)$.
Then, from functions $h^{-1}_E(t)$ and $\Gam_E(t)$, we can obtain a parametrization of these tori in terms of $(h,\OmegaM)$ of the form
\[
\Big(0,
\Gam_{J}^{\mathrm{ext}}(h),
J,0,
h,\OmegaM\Big),
\qquad
\text{for }
(h,\OmegaM) \in \TT^2,
\]
for 
\[
J \in [J_{\min},J_{\max}] = \left[-\frac{E_{\max}}{n_{\OmegaM}},-\frac{E_{\min}}{n_{\OmegaM}}\right],
\]
with $E_{\max}$, $E_{\min}$ as given in Ansatz~\ref{ansatz:periodic}.
Therefore, the union of these $2$-dimensional invariant tori form a normally hyperbolic invariant $3$-dimensional manifold $\Lambda_0$, diffeomorphic to a cylinder $\mathbb{T}^2\times [J_{\min},J_{\max}]$.
Applying the implicit function theorem with respect to the energy $J$, one can see that the cylinder  is analytic (by Ansatz~\ref{ansatz:periodic}, the periodic orbits are hyperbolic, thus non-degenerate).
We summarize this fact in the following lemma, which is a direct consequence of Ansatz \ref{ansatz:periodic} (see also Figure~\ref{fig:AddingPerturbation}-left).

\begin{corollary}\label{corollary:NHIMiM0}
Assume Ansatz~\ref{ansatz:periodic}.
Then, the Hamiltonian $\KK_{\CP}$ (see~\eqref{def:KHam:CP}) has an analytic
normally hyperbolic invariant $3$-dimensional cylinder $\Lambda_0$, which is foliated by $2$-dimensional invariant tori on the constant invariant hyperplanes $J=J_0 \in  [J_{\min},J_{\max}]$.

%

%
Moreover, the cylinder $\Lambda_0$ has  4-dimensional invariant manifolds, denoted by $W^{\unst}(\Lambda_0)$ and $W^{\sta}(\Lambda_0)$, which, for every $J\in  [J_{\min},J_{\max}]$, intersect transversally either at the symmetry axis $\{\eta=0,h=0\}$ with $\xi>0$ or at $\{\xi=0,h=0\}$ with $\eta>0$.

\end{corollary}

From now on, we will call primary homoclinic orbits (and denote them with a superindex $\prim$) those which intersect the symmetry axis $\{\eta=0,h=0\}$ with $\xi>0$ and we will call secondary homoclinic orbits (and denote them with a superindex $\secn$) those which intersect the symmetry axis $\{\xi=0,h=0\}$ with $\eta>0$. Indeed, the first ones only get close to the periodic orbit once while the second ones approach  it twice.

To analyze the dynamics on the cylinder and the scattering map associated with the transverse intersections of its invariant manifolds it is more convenient to consider a Poincar\'e map.
Therefore, we define a global Poincar\'e section and work with maps to reduce the dimension by one.

There are two natural choices of section, $\{h=0\}$ or $\{\OmegaM=0\}$, since both variables satisfy $\dot{h}\neq 0$ (see Theorem~\ref{thm:coplanar_periodic_orbits}) and $\dot{\Omega}_\rM= n_{\OmegaM} \neq 0$.
It would seem more natural to choose the section $\{\OmegaM=0\}$, however since we aim to prove drift on the energy $J$, it is more convenient to consider the section
\[
 \Sigma=\{h=0\}
\]
and the induced Poincar\'e map
\begin{equation}\label{def:PoincareMap}
 \PM_0:\Sigma\to\Sigma \qquad 
\end{equation}
Notice that this Poincar\'e section was already used in the numerical study in Section~\ref{sec:periodicorbits:coolanar}.
In addition, we denote the intersection of the cylinder $\Lambda_0$ with the section $\Sigma$ as
\[
\wt{\Lambda}_0 = \Lambda_0 \cap \Sigma.
\]
Notice that $\wt{\Lambda}_0$ is a 2-dimensional normally hyperbolic manifold for the Poincar\'e map $\PM_0$ with 3-dimensional  stable and unstable manifolds, which we denote $W^{\unst}(\wt{\Lambda}_0)$ and $W^{\sta}(\wt{\Lambda}_0)$ for $j=1,2$. They intersect at transverse homoclinic orbits which are just the those analyzed in Corollary \ref{corollary:NHIMiM0} intersected with the section $\Sigma$.
%

To fix notation, we denote by $\tD^{*}_0$, with $* \in \{\prim, \secn\}$, the subset of $[J_{\min},J_{\max}]$ 
for which there exist primary/secondary homoclinic orbits to the corresponding invariant torus in the cylinder.
%
By Corollary~\ref{corollary:NHIMiM0}, $\tD^{\prim}_0\cup \tD_0^{\secn} = [J_{\min},J_{\max}]$.

\begin{corollary}\label{corollary:NHIMiM0parametrized}
Assume Ansatz \ref{ansatz:periodic}.
The Poincar\'e map $\PM_{0}$ defined in \eqref{def:PoincareMap} and induced by Hamiltonian $\KK_{\CP}$ (see~\eqref{def:KHam:CP}) has a normally hyperbolic $2$-dimensional cylinder $\wt{\Lambda}_0$ foliated by invariant curves.
In addition, there exists an analytic function $\GG_0:[J_{\min},J_{\max}] \times \TT \to \RR^3 \times \TT^3$,
\begin{equation*}
\GG_0(J,\OmegaM) = \big(0,{\GG}_0^{\Gam}(J),J,0,0,\OmegaM\big),
\end{equation*}
that parametrizes $\wt{\Lambda}_0$,
\begin{equation*}
	\wt{\Lambda}_0 = \{ \GG_0(J,\OmegaM) \,:\, (J,\OmegaM) \in [J_{\min},J_{\max}] \times \TT  \}.
\end{equation*}
Moreover, for $* \in \{\prim, \secn\}$,
%
within the hypersurface $J=J_0$ with $J_0 \in \tD^{*}_0$, the invariant manifolds $W^{\unst}(\wt{\Lambda}_0)$ and $W^{\sta}(\wt{\Lambda}_0)$ intersect transversally at the symmetry axis of the involutions $\Phi^{\hor}$ (for $*=\prim$) and $\Phi^\ver$ (for $*=\secn$).

Let $\Xi_0^*$ denote these transverse intersections on the  symmetry axes.
Then, there exists an analytic function
$
\CCC^*_0 : \tD^*_0 \times \TT \to \RR^3 \times \TT^3,
$
such that
\begin{align*}
\CCC^{\prim}_0(J,\OmegaM) &= \Big(0,\CCC_0^{\Gam,\prim}(J),J,
\CCC_0^{\xi,\prim}(J),0,\OmegaM\Big),
\\
\CCC^{\secn}_0(J,\OmegaM) &= \Big(\CCC_0^{\eta,\secn}(J),\CCC_0^{\Gam,\secn}(J),J,
0,0,\OmegaM\Big),
\end{align*}
that parametrizes $\Xi_0^*$. That is,
\begin{equation*}
	\Xi^*_0 = \{ \CCC^*_0(J,\OmegaM) \,:\, (J,\OmegaM) \in [J_{\min},J_{\max}]\times \TT  \}.
\end{equation*}
\end{corollary}
The subscript $0$ in the previous structures and parameterizations indicates that we are dealing with the coplanar case, i.e., $\iM = 0$. 

Notice that Corollary~\ref{corollary:NHIMiM0} gives global coordinates for the cylinder $\wt{\Lambda}_0$. 
Moreover, these coordinates are symplectic with respect to the canonical symplectic form $d\OmegaM \wedge dJ$.
Indeed,  Corollary \ref{corollary:NHIMiM0parametrized} implies that at the cylinder $\wt{\Lambda}_0$ one has $\eta=\xi=h=0$ and $\Gam=\GG_0^{\Gam}(J)$. Then,  the pullback of the canonical form $d\xi\wedge d\eta + dh \wedge d \Gam + d \OmegaM \wedge d J$ to $\wt{\Lambda}_0$ is just $d \OmegaM \wedge d J$. 

Next, we consider the inner and the two scattering maps (for the primary and secondary homoclinic channels). The inner map is defined in the whole cylinder $\wt{\Lambda}_0$ whereas the scattering maps are defined in open domains of the cylinder where the associated homoclinic channels are transverse.
%
Since $J$ is conserved by the inner and scattering maps, these maps are integrable and the variables $(J,\OmegaM)$ are action-angle coordinates. 
In these variables, it will be easier to later understand the effect of the inclination $\iM$ on the inner and scattering maps.

\subsubsection{The coplanar inner map}

In this section we study the inner map, as introduced in~\eqref{def:innerMapPrevious}, restricted to the normally hyperbolic manifold $\wt{\Lambda}_0$.
In particular, the inner map $\FF_0^{\inner}:\wt{\Lambda}_0\to \wt{\Lambda}_0$ is defined as the Poincar\'e map $\PM_0$ in \eqref{def:PoincareMap} restricted to the symplectic invariant manifold $\wt{\Lambda}_0$.
We express $\FF_0^{\inner}$ using the global coordinates $(J,\OmegaM)$ of $\wt{\Lambda}_0$.
Since $J$ is an integral of motion and $\dot{\Omega}_{\rM} = n_{\OmegaM}$, the inner map has the form of
\begin{equation}\label{def:innerMapCoplanar}
	\FF_0^{\inner}: \begin{pmatrix}
		J \\ \OmegaM 
	\end{pmatrix}
	\mapsto
	\begin{pmatrix}
		J \\ \OmegaM + n_{\OmegaM}  
	\TTT_0(J)\end{pmatrix},
\end{equation}
where $\TTT_0(J)$ is the period of the periodic orbit obtained in Ansatz~\ref{ansatz:periodic} on the corresponding energy surface $E=-n_{\OmegaM}J$.

The following lemma is a direct consequence of Ansatz \ref{ansatz:periodic} (see Figure~\ref{fig:hamperiod}).

\begin{lemma} \label{corollary:innerMapCoplanar}
Assume Ansatz~\ref{ansatz:periodic}.
The analytic symplectic inner map $\FF_0^{\inner}$ defined in~\eqref{def:innerMapCoplanar} is twist, that is
\begin{equation*}
	\partial_J \TTT_0(J) < 0
	\quad \text{for} \quad J \in [J_{\min},J_{\max}].
\end{equation*}
In addition, there exists a unique $\Jres\in [J_{\min},J_{\max}]$ such that 
\[
n_{\Omega_\rM}\TTT_0(\Jres)=4\pi.
\]
Moreover,
\begin{equation*}
n_{\Omega_\rM}\TTT_0(J)\not\in 2\pi\ZZ\qquad \text{for all }\quad J \in [J_{\min},J_{\max}] \setminus \{\Jres\}.
\end{equation*}
\end{lemma}

This lemma is crucial to later, in Section~\ref{sec:pseudoorbit},  construct the heteroclinic connections between different tori of the perturbed cylinder (with $\iM>0$).

\subsubsection{The coplanar scattering maps}

In this section we study the scattering maps, as introduced in~\eqref{def:outerMapPrevious}, associated to the Poincar\'e map $\PM_0$ which sends $\wt{\Lambda}_0$ to itself by means of the homoclinic channels $\Xi_0^{\prim}$ and $\Xi_0^{\secn}$ (see Corollary~\ref{corollary:NHIMiM0parametrized}). 
For each homoclinic channel, we denote 
\begin{equation*}
	\FF^{\outer,\prim}_0: \wt{\Lambda}_0 \to \wt{\Lambda}_0,
	\qquad
	\FF^{\outer,\secn}_0: \wt{\Lambda}_0 \to \wt{\Lambda}_0.
\end{equation*}
Notice the abuse of notation since the scattering maps are only defined provided $J \in \tD_0^{\prim}$ or $J \in \tD_0^{\secn}$, respectively, and not in the whole
cylinder $\wt{\Lambda}_0$.

The scattering map is always exact symplectic, see \cite{delshams2008geometric}. 
So, since $J$ is preserved, the scattering maps must be of the form
\begin{equation}\label{def:outerMapCoplanar}
	\FF^{\outer,*}_0: \begin{pmatrix}
		J \\ \OmegaM
	\end{pmatrix}
	\mapsto
	\begin{pmatrix}
		J \\ \OmegaM + n_{\OmegaM} \zeta^*(J)
	\end{pmatrix},
 \quad \text{for} \quad
 *\in\{\prim,\secn\}.
\end{equation}
The functions $\zeta^*$ are usually called phase shift. Indeed the homoclinic orbits in $\Xi^*_0$ are homoclinic to a periodic orbit. However, they are asymptotic to different trajectories (that is, different phases) in the periodic orbit.

To compute the phase shifts $\zeta^*$, it is convenient to deal with a flow instead of the associated Poincar\'e map, because then one can rely on Melnikov Theory. However, the outer map induced by the flow associated to the Hamiltonian~\eqref{def:KHam:CP} does not preserve the section $\{h=0\}$ and, therefore, one cannot deduce the outer map associated with the Poincar\' e map $\PM_0$ from that of the flow. 
Then, we reparameterize the flow so that its scattering maps preserve this section. This reparameterization corresponds to identifying the variable $h$ with time $t$ and is given by
\begin{equation}\label{eq:ODEreparametrized}
	\begin{aligned}
		 \frac{d\xi}{ds} &= \frac{
			\partial_{\eta}\HH_{\CP}(\eta,\Gam,\xi,h)}{
			\partial_{\Gam} \HH_{\CP}(\eta,\Gam,\xi,h)}, 
		&
		\frac{d\eta}{ds} &= -\frac{
			\partial_{\xi}\HH_{\CP}(\eta,\Gam,\xi,h)}{
			\partial_{\Gam} \HH_{\CP}(\eta,\Gam,\xi,h)},
		\\
		\frac{dh}{ds} &= 1,
		&
		 \frac{d\Gam}{ds} & = -\frac{
			\partial_{h}\HH_{\CP}(\eta,\Gam,\xi,h)}{
			\partial_{\Gam} \HH_{\CP}(\eta,\Gam,\xi,h)}, 
		\\
		 \frac{d\OmegaM}{ds} &= \frac{n_{\OmegaM}}{\partial_{\Gam} \HH_{\CP}(\eta,\Gam,\xi,h)},
		&
		\frac{dJ}{ds} &= 0.
	\end{aligned}
\end{equation}
Note that the right hand side of the system of equations~\eqref{eq:ODEreparametrized} does not depend on $\OmegaM$. 
Let  $\fCPE\{s,(\eta,\Gam,J,\xi,h,\OmegaM)\} $ be the  flow associated to equation~\eqref{eq:ODEreparametrized} and $\fCP\{s,(\eta,\Gam,\xi,h)\}$ be the flow associated with its $(\eta,\Gam,\xi,h)$ components.
%
Componentwise, for the reduced and extended phase space, it can be written as
\begin{equation} \label{def:flowParametrizedCoplanar}
   \begin{aligned}
\fCP\{s,(\eta,\Gam,\xi,h)\} 
=
\big(&
\fC^{\eta}\{s,(\eta,\Gam,\xi,h)\},
\fC^{\Gam}\{s,(\eta,\Gam,\xi,h)\}, \\
&\fC^{\xi}\{s,(\eta,\Gam,\xi,h)\},
h+s
\big),
\\
\fCPE\{s,(\eta,\Gam,J,\xi,h,\OmegaM)\} 
=
\big(&
\fC^{\eta}\{s,(\eta,\Gam,\xi,h)\},
\fC^{\Gam}\{s,(\eta,\Gam,\xi,h)\},
J, \\
&\fC^{\xi}\{s,(\eta,\Gam,\xi,h)\},
h+s,
\OmegaM + \fC^{\OmegaM}\{s,(\eta,\Gam,\xi,h)\}
\big).
\end{aligned} 
\end{equation}

In addition, we denote the trajectories departing from points either in the normally hyperbolic cylinder or in one of the homoclinic channels, respectively, as
\begin{equation}\label{def:flowParametrizedCoplanarAuxiliar}
\begin{aligned}
\gamma_J(s)
&=
\varphi_{\CP}\{s,(0,\GG_0^{\Gam}(J),0,0)\},
&
\chi^{\prim}_J(s)
&=
\varphi_{\CP}\{s,(0,\CCC_0^{\Gam,\prim}(J),\CCC_0^{\xi,\prim}(J),0)\}, 
\\
& &
\chi^{\secn}_J(s)
&=
\varphi_{\CP}\{s,(\CCC_0^{\eta,\secn}(J),\CCC_0^{\Gam,\secn}(J),0,0)\}.
\end{aligned}
\end{equation}

\begin{lemma} \label{lemma:integralsScattering}
	Assume Ansatz \ref{ansatz:periodic}.
	The functions $\zeta^*(J)$ involved in the definition of the outer maps in \eqref{def:outerMapCoplanar} are given by
\[
\zeta^*(J) = \zeta^*_+(J) - \zeta^*_-(J).
\]
where
\begin{equation*}
\zeta^*_+(J) = - \zeta^*_-(J) =
\lim_{N \to +\infty} \sum_{l=0}^{N-1}  \left(\int_{2\pi l}^{2\pi (l+1)}
\frac{d s}{\partial_{\Gam}\HH_{\CP} \circ \chi^*_J(s)}
-  \TTT_0(J)
\right).
\end{equation*}
\end{lemma}

In Figure~\ref{fig:ham_phaseshift}-left, we plot the function $\zeta^{\prim}(J)$ for the case under study. This figure prompts us to assume the following ansatz, which will be used in Section \ref{sec:pseudoorbit}. Roughly speaking it will ensure that the inner and outer maps do not have simultaneous resonances at their first orders in $\iM$.

\begin{ansatz} \label{ansatz:phaseshiftouter}
For $J=\Jres$, the value introduced in Lemma \ref{corollary:innerMapCoplanar}, the function $\zeta^{\prim}$  satisfies that 
\[
\zeta^{\prim}(\Jres)\not\in \pi\ZZ\qquad \text{and}\qquad \partial_J {\zeta^{\prim}}(\Jres)\neq 0,
\]
\end{ansatz}

\begin{figure}[t]
\begin{center}
\includegraphics[width=8.2cm]{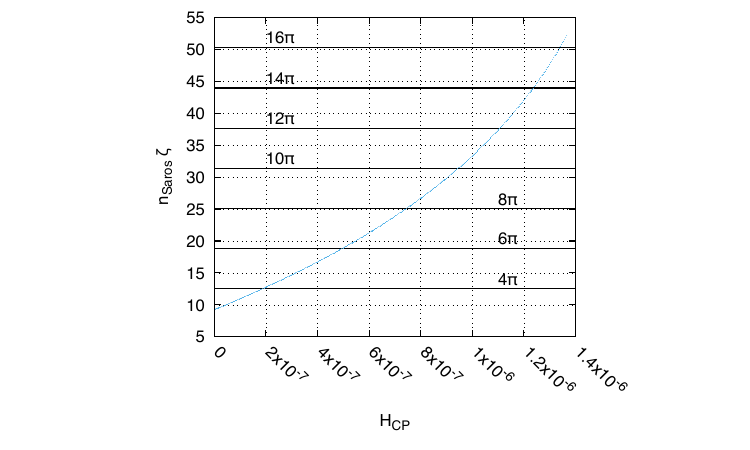} \hspace{-2cm}
\includegraphics[width=8.2cm]{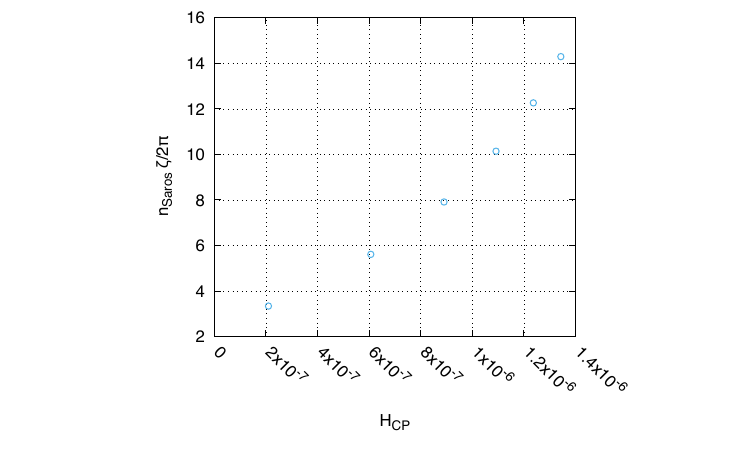}
\end{center}
\caption{As a function of ${\HH}_{\CP}\equiv -n_{\OmegaM} J$, we show on the left $n_{\Saros}\zeta^{\prim}(J)$. Here $n_{\Saros}\equiv n_{\OmegaM}$. Note that at  ${\HH}_{\CP}=4.4472\times 10^{-7}$  corresponding to the existence of a secondary resonance (see Fig.~\ref{fig:hamperiod} on the right), $n_{\Saros}\zeta^{\prim}(J)/\pi\approx 5.61$. On the right, $n_{\Saros}\zeta^{\secn}(J)/2\pi$ for the non-transverse values (see Table~\ref{tab:reldiff}). The integral of 
Lemma~\ref{lemma:integralsScattering} has been computed by means of the function {\tt qags} of the  {\tt quadpack} Fortran package. 
}
\label{fig:ham_phaseshift}
\end{figure}

Note that in  Figure~\ref{fig:ham_phaseshift}-left we only depict the horizontal lines at $\zeta^{\prim}=2\pi k$ instead of $\zeta=\pi k$ to have a clear picture. 
In Figure~\ref{fig:ham_phaseshift}-right, we plot the functions $\zeta^{\secn}(J)$ for the values of $J \in  [J_{\min},J_{\max}]$ where the primary homoclinic channel is not transverse (see Table~\ref{tab:reldiff}).

\begin{proof}[Proof of Lemma~\ref{lemma:integralsScattering}]
Since the flow component $(\fCPE)^{\OmegaM}$ is of the form $\OmegaM+\fC^{\OmegaM}$ with $\fC^{\OmegaM}$ independent of $\OmegaM$, its behavior is given by
\begin{equation}
\label{def:fundamentalFlowTime}
\varphi^{\OmegaM}\{s, (\eta,\Gam,\xi,h)\}
=
n_{\OmegaM}
\int_0^s \frac{ds}{\partial_{\Gam} \HH_{\CP} \circ \varphi_{\CP}\{s,(\eta,\Gam,\xi,h)\}}.
\end{equation}
First, we obtain an integral expression for the period $\TTT_0(J)$ given in \eqref{def:innerMapCoplanar}. 
Indeed, since the inner map is just the $2\pi$-time map of the flow $\fCPE$ for the components $(J,\OmegaM)$, by Corollary~\ref{corollary:NHIMiM0parametrized} one has that,
\[
 \TTT_0(J) = \fC^{\OmegaM}\{2\pi,(0,\GG_0^{\Gam}(J),0,0)\}
\]
and, by \eqref{def:flowParametrizedCoplanarAuxiliar} and \eqref{def:fundamentalFlowTime},
\begin{align}
\label{def:periodeEquacio}
\TTT_0(J)
&=
\int_0^{2\pi} 
\frac{ds}{\partial_{\Gam} \HH_{\CP} \circ \gamma_J(s)}.
\end{align}

Let us now consider the primary scattering map. One has  the homoclinic point
\[
\CCC_0^{\prim}(J,\OmegaM)= (0,\CCC_0^{\Gam,\prim}(J),J,\CCC_0^{\xi,\prim}(J),0,\OmegaM) \in 
W^{\unst,1}({\Lambda}_0)\cap W^{\sta,1}({\Lambda}_0)
\cap \{h=0\},
\]
(see Corollaries~\ref{corollary:NHIMiM0} and~\ref{corollary:NHIMiM0parametrized}).
Since the $(\eta,\Gam,\xi,h)$ components are independent of $\OmegaM$, this point is forward asymptotic (in the reparametrized time) to a point
\begin{align*}
Q_+ = \big( 0, \GG_0^{\Gam}(J), J,
0, 0, \OmegaM + n_{\OmegaM} \zeta^{\prim}_+(J) \big),
\end{align*}
and backward asymptotic to a point
\begin{align*}
Q_- = \big( 0, \GG_0^{\Gam}(J), J,
0, 0, \OmegaM + n_{\OmegaM} \zeta^{\prim}_-(J) \big).
\end{align*}
The scattering map corresponds to the application $Q_- \mapsto Q_+$ (see Definition~\ref{def:scattering}). 
Then, by the expression of the map $\FF^{\outer,\prim}_0$ in  \eqref{def:outerMapCoplanar}, one has that
\[
\zeta^{\prim}(J) = \zeta^{\prim}_+(J) - \zeta^{\prim}_-(J).
\]
By the definition of $\zeta_+^{\prim}(J)$, using definitions~\eqref{def:flowParametrizedCoplanarAuxiliar} of $\chi_J^{\prim}, \gamma_J$ and applying \eqref{def:fundamentalFlowTime}, one obtains 
\begin{align}
\zeta^{\prim}_+ (J)
&=
\frac1{n_{\OmegaM}}
\lim_{s\to +\infty}
\left( \varphi^{\OmegaM} \{s,(0,\CCC_0^{\Gam,\prim}(J),\CCC_0^{\xi,\prim},0)\} 
- \varphi^{\OmegaM} \{s, (0,\GG_0^{\Gam}(J),0,0)\}\right)
\label{proof:defzetaplus}
\\
&=\lim_{s \to +\infty} \int_0^s \left(
\frac{d\sigma}{\partial_{\Gam}\HH_{\CP} \circ \chi^{\prim}_J(\sigma)}
- \frac{d\sigma}{\partial_{\Gam}\HH_{\CP} \circ \gamma_J(\sigma)}
\right). \nonumber
\end{align}
Since system~\eqref{eq:ODEreparametrized} is $2\pi$-periodic in $s$, due to the identification of $s$ with $h$, it is more convenient to write these integrals as
\begin{align*}
\zeta^{\prim}_+(J) 
&=
\lim_{N \to +\infty}  \int_0^{2\pi N} \left(
\frac{d s}{\partial_{\Gam}\HH_{\CP} \circ \chi^{\prim}_J(s)}
- \frac{d s}{\partial_{\Gam}\HH_{\CP} \circ \gamma_J(s)}
\right),
\end{align*}
and, taking into account \eqref{def:periodeEquacio}, one has that
\begin{align*}
\zeta^{\prim}_+(J) 
&=
\lim_{N \to +\infty} \sum_{l=0}^{N-1} \int_{2\pi l}^{2\pi (l+1)} \left(
\frac{ds}{\partial_{\Gam}\HH_{\CP} \circ \chi^{\prim}_J(s)}
- \TTT_0(J)
\right).
\end{align*}

Let us recall that $\HH_{\CP}$ is reversible with respect to $\Phi^{\hor}(\eta,\Gam,\xi,h) = (-\eta,\Gam,\xi,-h)$ (see~\eqref{def:symmetryPoincare}).
Therefore, the flow satisfies $\fCP\{s,(0,\Gam,\xi,0)\} = \Phi^{\hor} \circ \fCP\{-s,(0,\Gam,\xi,0)\}$.
Then, one can see that $\fC^{\OmegaM}\{s,(0,\Gam,\xi,0)\} = - \fC^{\OmegaM}\{-s,(0,\Gam,\xi,0)\}$.
As a result,
\begin{align}
\zeta^{\prim}_-(J) 
&=
\frac1{n_{\OmegaM}} \lim_{s\to -\infty}
\left( \varphi^{\OmegaM} \{s,(0,\CCC_0^{\Gam,\prim}(J),\CCC_0^{\xi,\prim},0)\} 
- \varphi^{\OmegaM} \{s, (0,\GG_0^{\Gam}(J),0,0)\}\right)
\label{proof:defzetaminus}
\\
&= -\zeta^{\prim}_+(J). \nonumber
\end{align}
We proceed analogously for the secondary scattering map.
\end{proof}

\section{Dynamics of the system with \texorpdfstring{$\iM>0$}{iM>0}}
\label{sec:fullham}

In this section, we consider the  Hamiltonian $\KK$ in \eqref{def:HamExtPhsSp} and  assume $0<\iM\ll 1$. Since we want to compare its dynamics with that of $\KK_{\CP}$ (see also \eqref{def:KHam:CP}), we write $\KK$ as 
\begin{equation}\label{def:HamiM0PoincExtended}
	\begin{aligned}
		\KK(\eta,\Gam,J,\xi,h,\OmegaM;\iM) 
  &=\HH(\eta,\Gam,\xi,h,\OmegaM;\iM) + n_{\OmegaM}J\\
		&= \HH_{\CP}(\eta,\Gam,\xi,h) +  \al^3 \iM \RRR (\eta,\Gam,\xi,h, \OmegaM;\iM)
		+ n_{\OmegaM}J,
	\end{aligned}
\end{equation}
where $\HH_{\CP}$ has been introduced in \eqref{def:Ham:Coplanar}. 
See Appendix~\ref{appendix:HamiltonianPoincare} for the whole expression of $\RRR$.

\subsection{Perturbative analysis of the inner and scattering maps}

The first step is to compute perturbatively the inner and outer maps. 
To this end, we apply Poincar\'e-Melnikov techniques.
As done in the previous section (see \eqref{eq:ODEreparametrized}),  we reparameterize the flow so that the variable $h$ becomes ``time''. Namely,
\begin{equation}\label{eq:ODEreparametrizedMoon}
	\begin{aligned}
		\frac{d\xi}{ds} &= \frac{
			\partial_{\eta}\HH(\eta,\Gam,\xi,h,\OmegaM;\iM)}{
			\partial_{\Gam} \HH(\eta,\Gam,\xi,h,\OmegaM;\iM)}, 
		&
		\frac{d\eta}{ds} &= -\frac{
			\partial_{\xi}\HH(\eta,\Gam,\xi,h,\OmegaM;\iM)}{
			\partial_{\Gam} \HH(\eta,\Gam,\xi,h,\OmegaM;\iM)},
		\\
		\frac{dh}{ds} &= 1,
		&
		\frac{d\Gam}{ds} & = -\frac{
			\partial_{h}\HH(\eta,\Gam,\xi,h,\OmegaM;\iM)}{
			\partial_{\Gam} \HH(\eta,\Gam,\xi,h,\OmegaM;\iM)}, 
		\\
		\frac{d\OmegaM}{ds} &= \frac{n_{\OmegaM}}{\partial_{\Gam} \HH(\eta,\Gam,\xi,h,\OmegaM;\iM)}
		&
		\frac{dJ}{ds} &= -\al^3\iM\frac{\partial_{\OmegaM} \RRR(\eta,\Gam,\xi,h,\OmegaM;\iM)}{\partial_{\Gam} \HH(\eta,\Gam,\xi,h,\OmegaM;\iM)}.
	\end{aligned}
\end{equation}
This system is $\OO(\iM)$-close to  \eqref{eq:ODEreparametrized}.
Analogously, we consider the Poincar\'e map associated with this system,
\begin{align}\label{def:PoincareMapMoon}
    \PM_{\iM} : \Sigma \to \Sigma,\quad \Sigma=\{h=0\},
\end{align}
which is a $\OO(\iM)$-perturbation of the map $\PM_0$ given in \eqref{def:PoincareMap}.  

We study the system given by Hamiltonian $\KK$ (see \eqref{def:HamiM0PoincExtended}) as a perturbation of the one given by $\KK_{\CP}$ (see \eqref{def:KHam:CP}).
By classical Fenichel Theory (see \cite{fenichel1971persistence}) and Corollary~\ref{corollary:NHIMiM0}, one has that the flow associated to the Hamiltonian $\KK$ has a normally hyperbolic invariant $3$-dimensional cylinder which is $\iM$-close to the cylinder $\Lambda_0$.
Analogously, by Corollary~\ref{corollary:NHIMiM0parametrized}, the Poincar\'e map $\PM_{\iM}$ has a normally hyperbolic invariant $2$-dimensional cylinder at section $\Sigma$.

Recall that, also by Corollary~\ref{corollary:NHIMiM0parametrized}, the cylinder $\wt{\Lambda}_0$ possesses  two homoclinic channels which we denote by $\prim$ and $\secn$, which are defined for  $J\in \tD^{*}_0\subset [J_{\min},J_{\max}]$  for $*\in \{\prim,\secn\}$. Moreover, the sets $\tD^{*}_0$ satisfy that $\tD_0^{\prim} \cup \tD_0^{\secn} =  [J_{\min},J_{\max}]$.
%

For the full model, with $\iM>0$ small enough, we still have transverse homoclinic connections but in slightly smaller domains. In order to characterize them, we define 
\begin{equation} \label{def:domainsJPerturbed}
\tD^*_{\delta} = \left\{ J \in \tD^*_0 \, : \, 
\mathrm{dist}(J,\partial\tD^*_0) > \delta \right\},
\quad \text{ for } *\in \{\prim,\secn\}.
\end{equation}
Analogously, we also define 
\[
\tD_{\delta} = [J_{\min}+\delta,J_{\max}-\delta].
\]
Next theorem, which is a direct consequence of the classical Fenichel theory and Corollary~\ref{corollary:NHIMiM0parametrized}, summarizes the perturbative context.
\begin{theorem}
\label{theorem:invariantObjectsPerturbed}
Assume Ansatz \ref{ansatz:periodic}.
%
%
For any $\de>0$ and $r\geq 4$, there exists $\iM^0>0$ such that, for any $\iM\in (0,\iM^0)$, the map $\PM_{\iM}$ introduced in~\eqref{def:PoincareMapMoon} has the following properties.
\begin{enumerate}
\item It has a $\mathcal{C}^r$ normally hyperbolic invariant manifold $\wt{\Lambda}_{\iM,\de}$, which is $\iM$-close in the $\CCC^1$-topology to $\wt{\Lambda}_0$.
In addition, there exists a  function $\GG_{\iM}:\tD_{\delta}\times \TT \to \RR^3 \times \TT^3$ of the form
\begin{equation*}
	\GG_{\iM}(J,\OmegaM) = 
	 \big(\GG_{\iM}^{\eta}(J,\OmegaM),{\GG}_{\iM}^{\Gam}(J,\OmegaM),J,
	 \GG_{\iM}^{\xi}(J,\OmegaM),0,\OmegaM\big),
\end{equation*}
which   parameterizes  the cylinder $\wt{\Lambda}_{\iM,\de}$   as a graph, that is
\begin{equation*}
	\wt{\Lambda}_{\iM,\de} = \{ \GG_{\iM}(J,\OmegaM) \,:\, (J,\OmegaM) \in \tD_{\delta} \times \TT  \}.
\end{equation*}
\item There exists a nonvanishing $\mathcal{C}^{r-1}$ function, $a_{\iM}$, such that the pull back of the canonical form $d\xi\wedge d\eta+dh \wedge d\Gam+d\OmegaM\wedge dJ$ onto the cylinder $\wt{\Lambda}_{\iM,\de}$ is of the form 
\begin{equation}\label{def:formaSimplecticaInclined}
    a_{\iM}(J,\OmegaM) d\OmegaM \wedge d J.
\end{equation}

%
%

\item The homoclinic channels obtained in Corollary \ref{corollary:NHIMiM0parametrized}
 are persistent. That is, there exist $\mathcal{C}^r$ functions
$
\CCC^*_{\iM} : \tD^*_{\delta} \times \TT \to \RR^3 \times \TT^3,
$
\[
\CCC^*_{\iM}(J,\OmegaM) = \Big(
\CCC^{\eta,*}_{\iM}(J,\OmegaM),\CCC_{\iM}^{\Gam,*}(J,\OmegaM),J,
\CCC^{\xi,*}_{\iM}(J,\OmegaM),0,\OmegaM\Big),
\]
with $*\in\{\prim,\secn\}$, such that the manifolds
\begin{equation*}
	\Xi^*_{\iM, \de} = \{ \CCC^*_{\iM}(J,\OmegaM) \,:\, (J,\OmegaM) \in \tD^*_{\delta} \times \TT  \},
\end{equation*}
 belong to the transverse intersections between the  invariant manifolds $W^{\unst}(\wt{\Lambda}_{\iM,\de})$ and $W^{\sta}(\wt{\Lambda}_{\iM,\de})$. Moreover, the functions $\CCC^*_{\iM}$ are $\CCC^r$ regular with respect to $\iM$ and are $\mathcal{O}(\iM)$ to the function $\CCC^*_{0}$ obtained in Corollary \ref{corollary:NHIMiM0parametrized}. 
%
\end{enumerate}
\end{theorem}


We can define the inner and scattering maps in the invariant cylinder $\wt{\Lambda}_{\iM,\de}$ given in Theorem~\ref{theorem:invariantObjectsPerturbed} as we have done in Lemmas~\ref{corollary:innerMapCoplanar} and~\ref{lemma:integralsScattering}, respectively, for the coplanar case ($\iM=0$). 
We also compute for them first order expansions. To this end,  we consider the following definition.

\begin{definition}\label{def:numberharmonics}
Let $f$ be a $2\pi$-periodic function in $\OmegaM$ and denote by $f^{[k]}$ its $k$-th Fourier coefficient. 
Then, we define the set 
\[
\NNN_{\OmegaM}(f) = \{ k \in \ZZ \, : \, f^{[k]} \not\equiv 0 \}.
\]
\end{definition} 

Let us recall that $\HH = \HH_{\CP} + \al^3 \iM \RRR$ (see \eqref{def:HamiM0PoincExtended})
and that $\HH_{_\CP}$, the coplanar Hamiltonian, is independent of $\OmegaM$. However, the remainder $\RRR$ has the following harmonic structure 
\begin{align*} 
\RRR(\eta,\Gam,\xi,h,\OmegaM;\iM) =& \,
\RRR_1(\eta,\Gam,\xi,h,\OmegaM) +\iM\RRR_2(\eta,\Gam,\xi,h,\OmegaM)+ \OO(\iM^2),
\end{align*}
with
\begin{equation*} 
\NNN_{\OmegaM}(\RRR_1) = \{\pm 1\} \qquad \text{and}\qquad\NNN_{\OmegaM}(\RRR_2) = \{0,\pm 2\}.  
\end{equation*}
This is a consequence of the particular formulas of $\RRR$ in  Appendix~\ref{appendix:HamiltonianPoincare} (see equations~\eqref{def:RRR1appendix} and \eqref{def:harmonicsRRR2Appendix}). To fix notation, we write $\RRR_1$ as
\begin{equation*} 
\RRR_1(\eta,\Gam,\xi,h,\OmegaM)=e^{i\OmegaM} \RRR_{1}^+(\eta,\Gam,\xi,h) + e^{-i\OmegaM} \RRR_{1}^-(\eta,\Gam,\xi,h).
\end{equation*}

%


Note that $ \RRR_{1}^-= \overline{\RRR_{1}^+}$. We analyze the harmonic structure of all the objects involved in the definition of the inner and scattering maps. The first step is to study the asymptotic expansion with respect to $\iM$ of the flow associated to the vector field~\eqref{eq:ODEreparametrizedMoon}.

\begin{lemma}
\label{lemma:harmonicsFlowInclined}
Let $\fiMC\{s, (\eta,\Gamma,J,\xi,h,\OmegaM)\}$ be the flow associated with the vector field in~\eqref{eq:ODEreparametrizedMoon} induced by the Hamiltonian $\KK$ in~\eqref{def:HamiM0PoincExtended}.
Then, for $\iM>0$ small enough, it has an expansion
	\begin{align*}
		\fiMC\{s, (\eta,\Gamma,J,\xi,h,\OmegaM)\}
		=& \,
		\fCPE\{s, (\eta,\Gamma,J,\xi,h,\OmegaM)\}
		+ \iM \fiM_1\{s, (\eta,\Gamma,J,\xi,h,\OmegaM)\}
		\\
		&+ \iM^2 \fiM_2\{s, (\eta,\Gamma,J,\xi,h,\OmegaM)\} 
		+ \OO(\iM^3),
	\end{align*}
where $\fCPE$ has been introduced in \eqref{def:flowParametrizedCoplanar} and the remaining functions 
satisfy that
\begin{align*}
	\NNN_{\OmegaM} (\fiM_1\{s, (\eta,\Gamma,J,\xi,h,\OmegaM)\}) &= \{\pm 1\}, \\
	\NNN_{\OmegaM} (\fiM_2\{s, (\eta,\Gamma,J,\xi,h,\OmegaM)\}) &= \{0, \pm 2\}.
\end{align*}
\end{lemma}

The proof of this lemma is analogous to that of Lemma 3.6 in \cite{fejoz2016kirkwood}, and relies on classical perturbation theory jointly with the special form of~\eqref{eq:ODEreparametrizedMoon}.

\begin{lemma}
\label{lemma:harmonicParametrizations}
Assume Ansatz~\ref{ansatz:periodic}. 
Let $\GG_{\iM}$,  $a_{\iM}$ and $\CCC_{\iM}^{*}$, for $*\in\{\prim, \secn\}$, be as given in Theorem~\ref{theorem:invariantObjectsPerturbed}.
Then, for $\iM>0$ small enough, they have an asymptotic expansion of the form
\begin{align*}
    \GG_{\iM}(J,\OmegaM) &= \GG_0(J,\OmegaM) + \iM \GG_1(J,\OmegaM) +
    \iM^2 \GG_2(J,\OmegaM) + \OO(\iM^3),
    \\
       a_{\iM}(J,\OmegaM) &=
   1 + \iM a_1(J,\OmegaM) + \iM^2 a_2(J,\OmegaM) +  \OO(\iM^3),\\
    \CCC_{\iM}^*(J,\OmegaM) &= \CCC_0^*(J,\OmegaM) + \iM \CCC_1^*(J,\OmegaM) +
    \iM^2 \CCC_2^*(J,\OmegaM) + \OO(\iM^3),
\end{align*}
where $\GG_0$ and $\CCC_0^*$ have been introduced in Corollary~\ref{corollary:NHIMiM0parametrized} and the remaining functions satisfy that
\begin{align*}
    \NNN_{\OmegaM}(\GG_1) = \{\pm 1\}, 
    \qquad
    \NNN_{\OmegaM}(\GG_2) = \{0, \pm 2\}, 
    \\
    \NNN_{\OmegaM}(a_1) = \{\pm 1\}, 
    \qquad
    \NNN_{\OmegaM}(a_2) = \{0, \pm 2\}, 
    \\
    \NNN_{\OmegaM}(\CCC_1^*) = \{\pm 1\},
    \qquad
    \NNN_{\OmegaM}(\CCC_2^*) = \{0, \pm 2\},
\end{align*}
 and that, for $i=1,2$,
\begin{align*}
    \GG_i (J,\OmegaM) &= 
    \big(\GG_i^{\eta}(J,\OmegaM), \GG_i^{\Gamma}(J,\OmegaM), 0,
     \GG_i^{\xi}(J,\OmegaM),0,0\big),
     \\
     \CCC_i^{*}(J,\OmegaM) &=
     \big(\CCC_i^{\eta,*}(J,\OmegaM), \CCC_i^{\Gamma,*}(J,\OmegaM), 0,
     \CCC_i^{\xi,*}(J,\OmegaM), 0,0\big).
\end{align*}
\end{lemma}
Lemmas \ref{lemma:harmonicsFlowInclined} and \ref{lemma:harmonicParametrizations} allow to compute asymptotic expansions in $\iM>0$  for the inner and scattering maps of the inclined model. 
First, we introduce  notation (see \eqref{def:flowParametrizedCoplanarAuxiliar}) for the evolution of the coplanar flow for points on the normally hyperbolic invariant cylinder $\wt{\Lambda}_0$ and on the homoclinic channels $\Xi^* _0$ (see Corollary~\ref{corollary:NHIMiM0parametrized}):

\begin{equation}\label{def:auxiliarFunctionsPerturbedMaps}
\begin{aligned}
\gamma_J(s)
&=
\varphi_{\CP}\{s,(0,\GG_0^{\Gam}(J),0,0)\},
&
\chi^{\prim}_J(s)
&=
\varphi_{\CP}\{s,(0,\CCC_0^{\Gam,\prim}(J),\CCC_0^{\xi,\prim}(J),0)\}, 
\\
& &
\chi^{\secn}_J(s)
&=
\varphi_{\CP}\{s,(\CCC_0^{\eta,\secn}(J),\CCC_0^{\Gam,\secn}(J),0,0)\}, 
\\[0.5em]
\wt{\gamma}_J(s)
&=
\varphi^{\OmegaM}\{s,(0,\GG_0^{\Gam}(J),0,0)\},
&
\wt{\chi}^{\prim}_J(s)
&=
\varphi^{\OmegaM}\{s,(0,\CCC_0^{\Gam,\prim}(J),\CCC_0^{\xi,\prim}(J),0)\}, 
\\
& &
\wt{\chi}^{\secn}_J(s)
&=
\varphi^{\OmegaM}\{s,(\CCC_0^{\eta,\secn}(J),\CCC_0^{\Gam,\secn}(J),0,0)\}, 
\end{aligned}
\end{equation}
where $\varphi_{\CP}$ and $\varphi^{\OmegaM}$ have been introduced in \eqref{def:flowParametrizedCoplanar}.

\begin{theorem}
\label{theorem:PerturbedMaps}
Assume Ansatz \ref{ansatz:periodic}.
%
%
Fix $\delta>0$ and $\iM>0$ small enough. 
The normally hyperbolic manifold $\wt{\Lambda}_{\iM,\de}$ given in Theorem~\ref{theorem:invariantObjectsPerturbed} of the map $\PM_{\iM}$ (see~\eqref{def:PoincareMapMoon}) has associated inner and outer maps which are $\CCC^r$ (also with respect to $\iM$) and are of the following form.
\begin{itemize}
    \item The inner map is of the form
        \begin{equation*}
	\FF^{\inner}_{\iM} \begin{pmatrix}
    J 
  \\ \Omega_{M} 
	\end{pmatrix}
	=
	\begin{pmatrix}
	J + \iM A_1(J,\OmegaM) + \iM^2 A_2(J,\OmegaM) 
    + \OO(\iM^3)
  \\ 
  \Omega_{M} + n_{\OmegaM} \left\{\TTT_0(J)
  + \iM \TTT_1(J,\OmegaM) + \iM^2 \TTT_2(J,\OmegaM) 
  + \OO(\iM^3) \right\}
	\end{pmatrix},
\end{equation*}
where $(J,\OmegaM) \in \tD_\delta \times \mathbb{T}$, the function $\TTT_0$ has been introduced in  \eqref{def:innerMapCoplanar} and the functions $A_1, A_2, \TTT_1$ and $\TTT_2$ satisfy that
\begin{align*}
	\NNN_{\OmegaM}(A_1),\NNN_{\OmegaM}(\TTT_1) &= \{\pm 1\},  & \NNN_{\OmegaM}(A_2),\NNN_{\OmegaM}(\TTT_2)  &= \{0, \pm 2\}.
\end{align*}
Moreover, $A_1$ is of the form
\begin{align*}
	A_1(J,\OmegaM) &=  A_{1}^+(J)e^{i\OmegaM} +  A_{1}^-(J)e^{-i\OmegaM},
\end{align*}
where
\begin{align*}
    A_1^{\pm}(J) &= \mp i \al^3 \int_0^{2\pi} \frac{\RRR_1^{\pm}(\gamma_J(s))}{\partial_{\Gam} \HH_{\CP}(\gamma_J(s))} e^{\pm i\wt{\gamma}_J(s)} ds. 
\end{align*}
\item For  $*\in\{\prim,\secn\}$, the scattering map is of the form
\begin{equation*} 
	\FF^{\outer,*}_{\iM} \begin{pmatrix}
    J 
  \\ \Omega_{M} 
	\end{pmatrix}
	=
	\begin{pmatrix}
 	J + \iM B_1^*(J,\OmegaM) +\iM^2 B_2^*(J,\OmegaM) + \OO(\iM^3) 
  \\ 
  \Omega_{M} + n_{\OmegaM} \left\{ \zeta^*(J)  +\iM D_1^*(J,\OmegaM) + \iM^2 D_2^*(J,\OmegaM) +\OO(\iM^3) \right\}
	\end{pmatrix},
\end{equation*}
where $(J,\OmegaM) \in \tD^*_\delta \times \mathbb{T}$, the function $\zeta^*(J) = \zeta^*_+(J) - \zeta^*_-(J)$ has been introduced in Lemma~\ref{lemma:integralsScattering} and
 the functions $B^*_1, B^*_2, D^*_1$ and $D^*_2$ satisfy that
\begin{align*}
	\NNN_{\OmegaM}(B_1^*),\NNN_{\OmegaM}(D_1^*) &= \{\pm 1\},  & \NNN_{\OmegaM}(B_2^*),\NNN_{\OmegaM}(D_2^*)  &= \{0, \pm 2\}.
\end{align*}
Moreover, the functions $B_1^*$ are of the form
\begin{align*}
    B_1^*(J, \OmegaM) =  B_1^{*,+}(J)e^{i\OmegaM} + B_1^{*,-}(J) e^{-i\OmegaM},
\end{align*}
where 
\begin{align*}
    B_1^{*,\pm}(J) 
    = 
     \pm i\alpha^3\lim_{s \to -\infty} \Bigg[&
    \int_0^s \frac{\RRR_1^{\pm}(\chi_J^*(\sigma))}{\partial_{\Gam}\HH_{\CP}(\chi^*_J(\sigma))}
    e^{\pm i(\wt{\chi}^*_J(\sigma)+n_{\OmegaM}\zeta_+^*(J))} d\sigma
    \\
    &- \int_0^s \frac{\RRR_1^{\pm}(\gamma_J(\sigma))}{\partial_{\Gam}\HH_{\CP}(\gamma_J(\sigma))}
    e^{\pm i\wt{\gamma}_J(\sigma) } d\sigma
    \Bigg]
    \\
    \mp i\alpha^3\lim_{s \to +\infty} \Bigg[&
    \int_0^s \frac{\RRR_1^{\pm}(\chi_J^*(\sigma))}{\partial_{\Gam}\HH_{\CP}(\chi^*_J(\sigma))}
    e^{\pm i(\wt{\chi}^*_J(\sigma) + n_{\OmegaM}\zeta_+^*(J))} d\sigma
    \\
    &- \int_0^s \frac{\RRR_1^{\pm}(\gamma_J(\sigma))}{\partial_{\Gam}\HH_{\CP}(\gamma_J(\sigma))}
    e^{\pm i(\wt{\gamma}_J(\sigma) + 2 n_{\OmegaM}\zeta_+^*(J))} d\sigma
    \Bigg].
\end{align*}
\end{itemize}
\end{theorem}
\begin{proof}
First, we focus on the inner map. Its regularity is a consequence of the regularity of the original flow and the regularity of the invariant cylinder (which is a consequence of Fenichel Theory).
Notice that, since it is the dynamics of the Poincar\'e map $\PM_{\iM}$ (defined in~\eqref{def:PoincareMapMoon}) restricted to $\widetilde{\Lambda}_{\iM,\delta}$, it satisfies the following homological equation
\begin{align}\label{proof:homologicalEquationsInnerOuterMap}
    {\PM}_{\iM} \circ \GG_{\iM} = \GG_{\iM} \circ \FF^{\inner}_{\iM}, 
\end{align}
where $\GG_{\iM}$ has been introduced in Theorem~\ref{theorem:invariantObjectsPerturbed}. 
Notice that $\PM_{\iM}$ is defined by
\begin{align*}
\PM_{\iM}(\eta,\Gam,J,\xi,0,\OmegaM) 
=
\psi^{\iM}\{2\pi,(\eta,\Gam,J,\xi,0,\OmegaM)\}, 
\end{align*}
where $\psi^{\iM}$ is the flow associated to the vector field~\eqref{eq:ODEreparametrizedMoon}. 
Therefore, by Lemma~\ref{lemma:harmonicsFlowInclined}, one can consider the following expansion 
\begin{align*}
\PM_{\iM} = {\PM}_0 
+ \iM \PM_1 
+ \iM^2 \PM_2 
+ \OO(\iM^3), 
\end{align*}
where, 
\begin{align*}
  \PM_0(\eta,\Gam,J,\xi,0,\OmegaM) 
  &=  \fCPE \{2\pi,(\eta,\Gam,J,\xi,0,\OmegaM)\},
  \\
  \PM_k(\eta,\Gam,J,\xi,0,\OmegaM) 
  &= \fiM_k \{2\pi,(\eta,\Gam,J,\xi,0,\OmegaM)\}, 
  \quad 
k=1,2.
\end{align*}
Let us consider the following expansions of the inner map
\begin{align*}
    \FF^{\inner}_{\iM} &= 
    {\FF}_0^{\inner} +
    \iM \FF^{\inner}_1 +
    \iM^2 \FF^{\inner}_2 + \OO(\iM^3).
\end{align*}
By \eqref{proof:homologicalEquationsInnerOuterMap} and Lemma~\ref{lemma:harmonicParametrizations}, the map $\FF_0^{\inner}$ is as given in \eqref{def:innerMapCoplanar}.
Moreover, 
\begin{equation*}
    \PM_1 \circ \GG_0 + (D \PM_0 \circ \GG_0) \GG_1 
    =
    \GG_1 \circ \FF_0^{\inner} +
    (D \GG_0 \circ \FF_0^{\inner}) \FF_1^{\inner}.
\end{equation*}
and
\begin{equation}\label{proof:innerMapSecondTerm}
\begin{aligned}
    &\PM_2 \circ \GG_0 + 
    (D \PM_1 \circ \GG_0) \GG_1 +
    \frac12 (D^2 \PM_0 \circ \GG_0) \GG_1^{\otimes 2} +
    (D \PM_0 \circ \GG_0) \GG_2
    \\
    &=
    \GG_2 \circ \FF_0^{\inner} +
    (D \GG_1 \circ \FF_0^{\inner}) \FF_1^{\inner} +
    \frac12 (D^2 \GG_0 \circ \FF_0^{\inner}) + 
    (D \GG_0 \circ \FF_0^{\inner}) \FF_2^{\inner}.
\end{aligned}
\end{equation}
Then, by Corollary~\ref{corollary:NHIMiM0parametrized} and Lemma~\ref{lemma:harmonicParametrizations},
\begin{equation} \label{proof:innerMapFirstTerm}
 \pi_{\diamond}\{\FF_1^{\inner} \} = \pi_{\diamond} \{ 
 \PM_1 \circ \GG_0 + (D \PM_0 \circ \GG_0) \GG_1\} \},
 \quad
 \diamond = J, \OmegaM,
\end{equation}
where $\pi_J$ and $\pi_{\OmegaM}$ denote the projections onto the coordinates $J$ and $\OmegaM$, respectively.
Since $\NNN_{\OmegaM}(\GG_1), \NNN_{\OmegaM}(\PM_1) = \{\pm 1\}$ one has that 
$
\NNN_{\OmegaM}(\FF_1^{\inner})=\{\pm 1\}.
$
Moreover, using analogous arguments and equation~\eqref{proof:innerMapSecondTerm}, one can obtain that
$
\NNN_{\OmegaM}(\FF_2^{\inner})=\{0, \pm 2\}.
$

Now it only remains to compute the formula for $A_1(J,\OmegaM)$. Indeed, by~\eqref{proof:innerMapFirstTerm} and taking into account that $\partial_J \PM_0 = (0,0,1,0,0,0)$, we have that
\begin{align*}
    A_1(J,\OmegaM) 
    = 
    \pi_{J} \{ \FF_1^{\inner}(J,\OmegaM) \}
    =
    \pi_J \{ \PM_1 \circ \GG_0\}
    = 
    \pi_J \, {\psi_1 \{ 2\pi, \GG_0(J,\OmegaM)\}}.
\end{align*}
Since $\psi^{\iM}$ is the flow associated to the vector field~\eqref{eq:ODEreparametrizedMoon}, we can apply the fundamental theorem of calculus and  the expression of $\RRR$ in \eqref{eq:expansionH1Poincare} to obtain 
\begin{align*}
\pi_J \, \psi^{\iM}\{ 2\pi, \GG_0(J,\OmegaM)\}
&= 
- \al^3 \iM \int_0^{2\pi} 
\frac{\partial_{\OmegaM} \RRR(\psi^{\iM} \{  s
, \GG_0(J,\OmegaM)\};\iM)}{\partial_{\Gam} \HH(\psi^{\iM} \{ 
s
, \GG_0(J,\OmegaM)\};\iM)} ds
\\
&=
- \al^3 \iM \int_0^{2\pi} 
\frac{\partial_{\OmegaM} \RRR_1(\fCPE\{ s, \GG_0(J,\OmegaM)\})}{\partial_{\Gam} \HH_{\CP}(\fCP \{ s, \GG_0(J,\OmegaM)\})} ds + \OO(\iM^2).
\end{align*}
Then, by the expression of $\RRR_1$ in \eqref{def:RRR1appendix} and of functions~\eqref{def:auxiliarFunctionsPerturbedMaps}, one has that
\begin{align*}
    A_1(J,\OmegaM) 
    =& \,
    - i \al^3 \int_0^{2\pi} 
    \frac{\RRR_1^+(\gamma_J(s)) 
    e^{i(\OmegaM+\wt\gamma_J(s))}}{\partial_{\Gam} \HH_{\CP}(\gam_J(s))} ds
    \\
    &+ i \al^3 \int_0^{2\pi} 
    \frac{\RRR_1^-(\gamma_J(s)) 
    e^{-i(\OmegaM+\wt\gamma_J(s))}}{\partial_{\Gam} \HH_{\CP}(\gam_J(s))} ds.
\end{align*}


Next, we compute the scattering maps for a fixed $*=\{\prim,\secn\}$ (see Definition~\ref{def:scattering}).  Its regularity is proven in \cite{delshams2008scattering} (for regular vector fields with regular normally hyperbolic invariant manifolds).

%
Let us consider  points 
$\GG_{\iM}(J^+,\Omega_{\rM}^+), \GG_{\iM}(J^-,\Omega_{\rM}^-) \in\wt{\Lambda}_{\iM,\de}$ and $\CCC_{\iM}^*(J^0,\Omega_{\rM}^0) \in \Xi_{\iM,\de}^*$ (see Theorem~\ref{theorem:invariantObjectsPerturbed}) such that
the unstable fiber of $\GG_{\iM}(J^-,\Omega_{\rM}^-)$ intersects the stable fiber of $\GG_{\iM}(J^+,\Omega_{\rM}^+)$ at the point $\CCC_{\iM}^*(J^0,\Omega_{\rM}^0)$ in the  homoclinic channel $\Xi_{\iM,\delta}^*$.  
This implies that  
\begin{equation}\label{eq:scattering_prop}
\lim_{s\to \pm \infty} 
\big|
\psi^{\iM}  \{ s, \GG_{\iM}(J^{\pm},\Omega_{\rM}^\pm)\} -
\psi^{\iM}  \{ s, \CCC_{\iM}^*(J^0, \Omega_{\rM}^0)\} \big| = 0,
\end{equation}
where $\psi^{\iM}$ is the flow associated to the vector field~\eqref{eq:ODEreparametrizedMoon}.
In other words, the scattering map satisfies that $(J^+,\Omega_{\rM}^+) = \FF^{\outer,*}_{\iM}(J^-, \Omega_{\rM}^-)$.

The analysis of the harmonic structure can be done as for the inner map. We show now how to  compute the first order of the $J$-component. 
Recall that $\pi_J \GG_{\iM}(J,\OmegaM)= \pi_J \CCC_{\iM}^*(J,\OmegaM)=J$.
By the fundamental theorem of calculus and using~\eqref{eq:scattering_prop}, one has that
\begin{align*}
     J^0-J^{\pm}  = 
    \lim_{s \to {\pm}\infty} 
    \Bigg[& \int_s^0 \pi_J \partial_s
    \psi^{\iM}\{\sigma, \CCC_{\iM}^*(J^0,\Omega_{\rM}^0)\} d\sigma
    \\
    &- \int_s^0 \pi_J \partial_s
    \psi^{\iM} \{\sigma, \GG_{\iM}(J^{\pm},\Omega_{\rM}^{\pm}  )\}
    d\sigma \Bigg].
\end{align*}
Then, by~\eqref{eq:ODEreparametrizedMoon},
\begin{align*}
    J^0-J^{\pm} = \alpha^3 \iM \lim_{s \to \pm \infty} \Bigg(&
    \int_0^{s} 
\frac{\partial_{\OmegaM} \RRR(\psi^{\iM} \{ \sigma, \CCC_{\iM}^*(J^0,\Omega_{\rM}^0\};\iM)}{\partial_{\Gam} \HH(\psi^{\iM} \{ \sigma, \CCC_{\iM}^*(J^0,\Omega_{\rM}^0)\};\iM)} d\sigma
\\
&- \int_0^{s} 
\frac{\partial_{\OmegaM} \RRR(\psi^{\iM} \{\sigma, \GG_{\iM}(J^{\pm},\Omega_{\rM}^{\pm})\};\iM)}{\partial_{\Gam} \HH(\psi^{\iM} \{\sigma, \GG_{\iM}(J^{\pm},\Omega_{\rM}^{\pm})\};\iM)} d\sigma
\Bigg)
\end{align*}
and, by Lemmas~\ref{lemma:harmonicsFlowInclined}, \ref{lemma:harmonicParametrizations} and the expression of $\RRR$ in \eqref{eq:expansionH1Poincare}, we obtain 
\begin{align*}
    J^0 - J^{\pm} = \alpha^3 \iM \lim_{s \to \pm \infty} \Bigg(&
    \int_0^{s} 
\frac{\partial_{\OmegaM} \RRR_1(\fCPE \{ \sigma, \CCC_{0}^*(J^0,\Omega_{\rM}^0)\})}{\partial_{\Gam} \HH_{\CP}(\fCPE \{ \sigma, \CCC_{0}^*(J^0,\Omega_{\rM}^0)\})} d\sigma
\\
&- \int_0^{s} 
\frac{\partial_{\OmegaM} \RRR_1(\fCPE \{\sigma, \GG_{0}(J^{\pm},\Omega_{\rM}^{\pm})\})}{\partial_{\Gam} \HH_{\CP}(\fCPE \{\sigma, \GG_{0}(J^{\pm},\Omega_{\rM}^{\pm})\})} d\sigma
\Bigg) + \OO(\iM^2).
\end{align*}
Finally, by the expression of $\RRR_1$ in \eqref{def:RRR1appendix}, of $\fCPE$ in \eqref{def:flowParametrizedCoplanar} and of the functions~\eqref{def:auxiliarFunctionsPerturbedMaps}
\begin{align*}
    J^0&-J^{\pm} = i\alpha^3 \iM \lim_{s \to \pm\infty} 
    \\
    &\Bigg( \int_0^{s} 
    \frac{\RRR_1^+(\chi^*_{J^0}(\sigma)) e^{i\Omega_{\rM}^0+i\wt{\chi}_{J^0}^*(\sigma)} - 
    \RRR_1^-(\chi^*_{J^0}(\sigma)) e^{-i\Omega_{\rM}^0-i\wt{\chi}_{J^0}^*(\sigma)}}{
\partial_{\Gam} \HH_{\CP}(\chi^*_{J^0}(\sigma))} d\sigma
\\
&- \int_0^{s} 
\frac{\RRR_1^+(\gamma_{J^{\pm}}(\sigma)) e^{i\Omega_{\rM}^{\pm} + i\wt{\gamma}_{J^{\pm}}(\sigma)} -
\RRR_1^-(\gamma_{J^{\pm}}(\sigma)) e^{-i\Omega_{\rM}^{\pm} - i\wt{\gamma}_{J^{\pm}}(\sigma)}
}{\partial_{\Gam} \HH_{\CP}(\gamma_{J^{\pm}}(\sigma))} d\sigma
\Bigg) + \OO(\iM^2).
\end{align*} 
%
Notice that, by Lemma~\ref{lemma:integralsScattering}, \eqref{proof:defzetaplus} and \eqref{proof:defzetaminus}, one has that
\begin{align*}
\Omega_{\rM}^{\pm} &=  \Omega_{\rM}^0 + n_{\OmegaM} \zeta^*_{\pm}(J^0) + \OO(\iM),
\end{align*}
with $\zeta_+^* = -\zeta_-^*$.
Therefore, since $J^{\pm} = J^0 + \OO(\iM)$, one has that
\begin{align*}
\Omega_{\rM}^0 &=  \Omega_{\rM}^- + n_{\OmegaM} \zeta^*_{+}(J^-) + \OO(\iM),
\\
\Omega_{\rM}^+ &= \Omega_{\rM}^- + 
2 n_{\OmegaM} \zeta^{*}_+(J^-)+ \OO(\iM)
\end{align*} 
and, as a consequence,
\begin{align*}
    J^+ =& \, J^- + (J^0-J^-) - (J^0 - J^+)
    \\
    =& \, J^- + \iM \left( B_1^{*,+}(J^-) e^{i\Omega_{\rM}^-}
    + B_1^{*,-}(J^-) e^{-i\Omega_{\rM}^-} \right)
    + \OO(\iM^2),
\end{align*}
with $B_1^{*,\pm}$ are the functions defined in the statement of the result.  
\end{proof}



\subsection{Existence of diffusing orbits}

Once we have computed the first orders in $\iM$ of both the inner and the scattering maps (see Theorem~\ref{theorem:PerturbedMaps}), the next step is to construct a ``drifting pseudo-orbit''.  By a pseudo-orbit, we mean a sequence of points in the cylinder obtained by applying successively iterations of the inner map $\FF_{\iM}^\inner$ and the scattering maps $\FF_{\iM}^{\outer,\prim}$ and  $\FF_{\iM}^{\outer,\secn}$. By drifting, we mean that we look for a pseudo-orbit such that its initial condition is close to the bottom of the cylinder,  i.e $J\sim J_{\min}$, and that eventually it hits a neighborhood of its top, i.e., $J\sim J_{\max}$.

To construct this pseudo-orbit we rely on the following ansatz, which we verify numerically. 
See Figures~\ref{fig:A1sincos_B1sincos}, \ref{fig:A1sincos_B1sincos_prn_ntt} and~\ref{fig:ham_norm_prn}.

\begin{figure}[t]
\begin{center}
		\includegraphics[width=10cm]{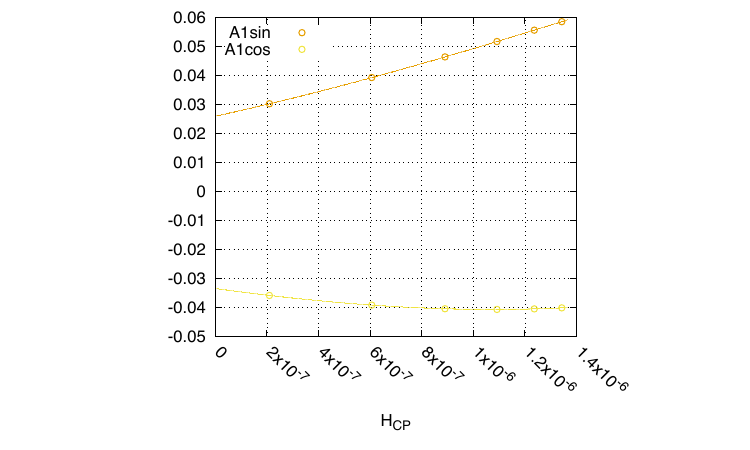}
  \end{center}
 \caption{The components of $A_1^{+}$ as a function of ${\HH}_{\CP}$ (non-dimensional units).}
 \label{fig:A1sincos_B1sincos}
\end{figure}

\begin{figure}[ht]
	\begin{minipage}{0.45\linewidth}
		\includegraphics[width=8.2cm]{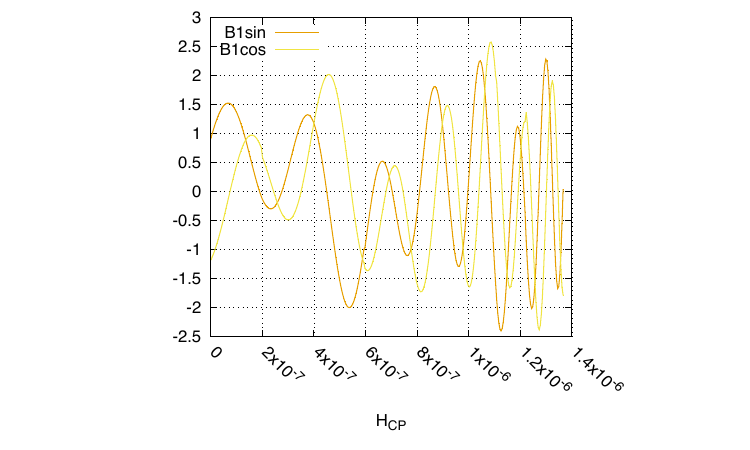}
	\end{minipage}
 	\begin{minipage}{0.45\linewidth}
		\includegraphics[width=8.2cm]{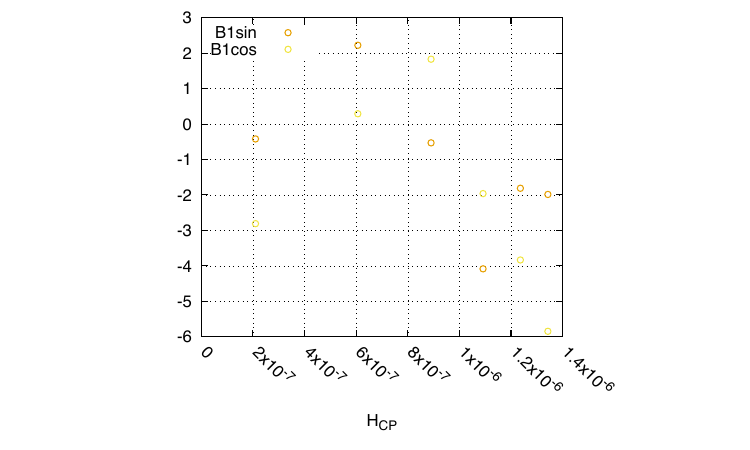}
	\end{minipage}
 \caption{
 For $*\in \{\prim,\secn\}$, we plot the real functions $B_{1,\cos}^*=B_1^+ + B_1^-$ and $B_{1,\sin}^*=i(B_1^+ - B_1^-)$  as a function of ${\HH}_{\CP}$ (non-dimensional units). Left, the $*=\prim$ case. Right, the $*=\secn$ case for the non-transverse values (see Table~\ref{tab:reldiff}).}
 \label{fig:A1sincos_B1sincos_prn_ntt}
\end{figure}

\begin{ansatz} \label{ansatz:straightening}
For any $\de>0$ and $*\in \{\prim, \secn\}$, the following functions of $J$
\begin{align*}
    f^*_{\pm}(J) &= \left(e^{\pm i n_{\OmegaM} \TTT_0(J)}-1\right) B_1^{*,\pm}(J) - \left(e^{\pm i n_{\OmegaM} \zeta^*(J)}-1\right) A_1^{\pm}(J), 
\end{align*}
do not vanish in the domain $\tD_{\delta}^*$.
\end{ansatz}

The following theorem ensures the existence of a pseudo-orbit. Its proof is deferred to Section \ref{sec:pseudoorbit}.

\begin{figure}[H]
	\begin{center}
		\includegraphics[width=9.cm]{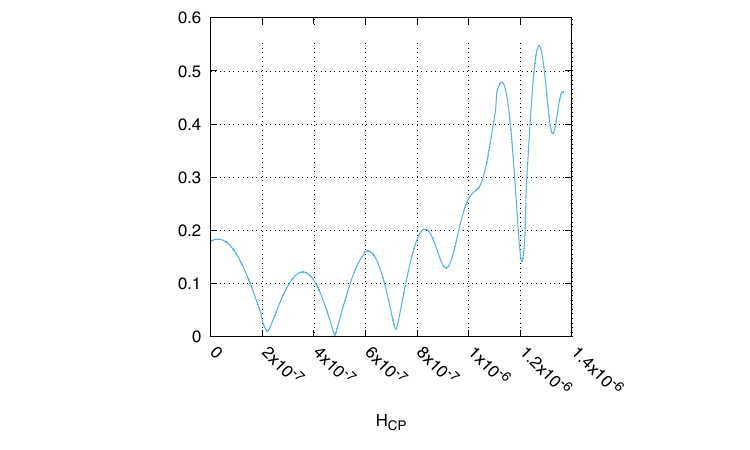}  		\hspace{-3.42cm}\includegraphics[width=9.cm]{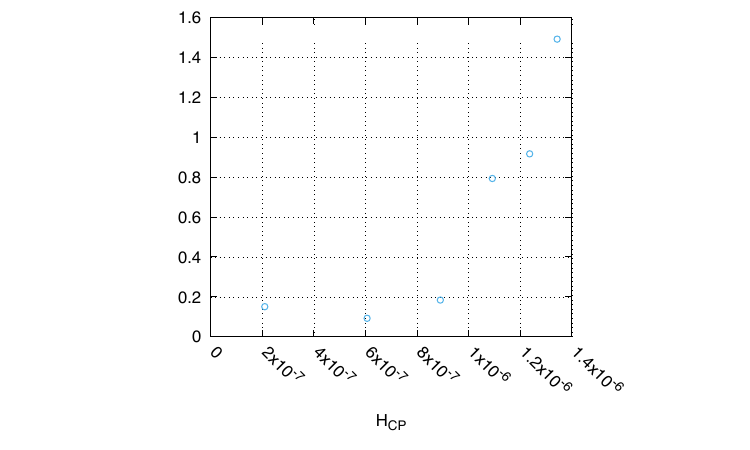} 
	\end{center}
	\caption{As a function of ${\HH}_{\CP}\equiv - J$, we show the norm of $f^{\prim}_{+}(J)$ in Ansatz~\ref{ansatz:straightening} on the left. The minimum value computed is 0.0014295 and corresponds to ${\HH}_{\CP}= 4.81143\times 10^{-7}$. On the right, the same norm computed for the non-transverse cases (see Table~\ref{tab:reldiff}).  The corresponding integrals have been computed by means of the function {\tt qags} of the  {\tt quadpack} Fortran package.}
	\label{fig:ham_norm_prn}
\end{figure}

\begin{theorem}\label{thm:pseudoorbit}
Assume Ans\"atze \ref{ansatz:periodic}, \ref{ansatz:phaseshiftouter}, \ref{ansatz:straightening}. Fix $\nu>0$ and $\de>0$. Then, if $\iM>0$ is small enough, there exist $z_0\in\widetilde{\Lambda}_{\iM,\de}$,  $N\in \NN$ and $l_m\in \{\prim,\secn\}$, $j_m, k_m\in\NN$ for $m=1\dots N$ such that the sequence
\[
z_{m}=\left(\FF_{\iM}^\inner\right)^{j_m}\circ\FF_{\iM}^{\outer,l_m}\circ \left(\FF_{\iM}^\inner\right)^{k_m}(z_{m-1})
\]
satisfies $z_m\in \wt\Lambda_{\iM,\de}$ for all $m=1\ldots N$ and the initial and final points $z_0$ and  $z_N$ satisfy
\[
|\pi_J z_0-J_{\min}|\leq \nu \quad \text{and}\quad |\pi_J z_N-J_{\max}|\leq \nu,
\]
where $\pi_J$ denotes the projection onto 
the $J$ component.

Moreover, in the regions 
\[
\JJ_{\min}=(J_{\min},J_{\min}+\nu)\times\TT,\qquad  \JJ_{\max}=(J_{\max}-\nu,J_{\max})\times\TT,
\]
the inner map $\FF_{\iM}^\inner$ has invariant tori which are a graph over $h$  and whose dynamics are conjugated to rigid quasi-periodic rotation.
\end{theorem}

Once we have a pseudo-orbit, the final step to prove Theorem \ref{thm:main} is to construct a true orbit of the Poincar\'e map $\PM_{\iM}$ which shadows the pseudo-orbit, that is, which visits small neighborhoods of the points $z_m$ of the pseudo-orbit given by Theorem \ref{thm:pseudoorbit}. This shadowing will rely on the following theorem by M. Gidea, R. de la Llave and T. M. Seara in \cite{GideaLS20}.

\begin{theorem}[Lemma 3.2 in \cite{GideaLS20}]\label{lemma:shadowingCPAM}
Assume that $f:\MM\to \MM$ is a $\mathcal{C}^r$ map with $r\geq 4$, $\Lambda\subset \MM$ is a normally hyperbolic invariant manifold, $\Gamma_j$, $j=1\ldots N$ are homoclinic channels and $\sigma^j:D_1^j\to D_2^j$ are the associated scattering maps, where $D_{1,2}^j\subset\Lambda$ are open sets. Assume that $\Lambda$ and $\Gamma_j$ are compact manifolds (possibly with boundary).

Then, for every $\nu>0$, there exist functions $n_i^*:\NN^i\to\NN$ 
and  functions $m_i^*:\NN^{2i+1}\times\NN^{i+1}\to\NN$ 
 such that for every pseudo-orbit $\{y_i\}_{y\geq 0}$ in $\Lambda$ of the form 
\[
y_{i+1}=f^{m_i}\circ\sigma^{\al_i}\circ f^{n_i}(y_i),
\]
with $n_i\geq n_i^*(\al_0,\ldots, \al_{i-1})$, $m_i\geq m_i^*(n_0,\ldots, n_i, m_0,\ldots. m_{i-1},\al_0,\ldots, \al_{i})$ and $\alpha_i \in \{1,\ldots, N\}$, there exists an orbit $\{z_i\}_{i\geq 0}$ of $f$ in $\MM$ such that, for all $i\geq 0$,
\[
z_{i+1}=f^{m_i+n_i}(z_i)\quad \text{and}\quad d(z_i,y_i)<\nu.
\]
\end{theorem}

Theorem \ref{thm:main} is a direct consequence of the following lemma.

\begin{lemma}\label{lemma:shadowingfinal}
Assume Ans\"atze \ref{ansatz:periodic}, \ref{ansatz:phaseshiftouter}, \ref{ansatz:straightening} and fix $\nu>0$ small. Then, for $\iM>0$ small enough, there exist a point $p_0\in\Sigma$ (see \eqref{def:PoincareMapMoon}) and $k_1\ldots k_N$ such that the iterates 
\[
p_j=\PM_{\iM}^{k_j}(p_0)
\]
satisfy that 
\[
\left|p_{j}-\GG_{\iM}(z_j)\right|\leq 2\nu,\qquad j=1\ldots N,
\]
where  $\{z_j\}_{i=0}^N$ is the pseudo-orbit obtained in Theorem~\ref{thm:pseudoorbit} and $\GG_{\iM}$ is the parameterization of the cyilinder given in Lemma \ref{theorem:invariantObjectsPerturbed}.

Moreover,
\[
|\pi_J p_0-J_{\min}|\leq 2\nu\qquad\text{and}\qquad |\pi_J p_{N}-J_{\max}|\leq 2\nu.
\]
\end{lemma}

\begin{proof}
To prove Lemma \ref{lemma:shadowingfinal} we have to apply Theorem \ref{lemma:shadowingCPAM}. To this end, we denote by $\Lambda$ a domain of the cylinder $\wt\Lambda_{\iM,\de}$ delimited by one of the ```top'' and ``bottom'' invariant tori given by Theorem~\ref{thm:pseudoorbit}. This makes $\Lambda$ compact and invariant. Note that  the projection onto $J$ of the  homoclinic channels $\Xi^*_{\iM,\de}$ obtained in Theorem \ref{theorem:invariantObjectsPerturbed}  covers the whole interval $[J_{\min}+\de, J_{\max}-\de]$ and that one can choose compact sets inside the homoclinic channels that still cover the same interval. 

Then, to apply Theorem \ref{lemma:shadowingCPAM}, it only remains to ensure that the amount of iterates of the inner maps 
is large enough to fit the hypotheses of the theorem. We follow ideas developed in \cite{GideaLS20}.

Since $\Lambda$ is compact and invariant by $\PM_{\iM}$ and $\FF_{\iM}^\inner$ is an area preserving map, we can use the Poincar\'e Recurrence Theorem to construct a pseudo-orbit which is arbitrarily close to that of Theorem \ref{thm:pseudoorbit} and satisfies the hypotheses of Theorem \ref{lemma:shadowingCPAM}. Indeed, the Poincar\'e Recurrence Theorem assures that there exists $\tilde{z}_0$ and $\tilde{k}_1$ big enough such that $\wt{Q}_0:=\big (\FF^{\inner}_{\iM} \big )^{\tilde{k}_1}(\tilde{z}_0)$ is as close as necessary to $Q_0:=\big (\FF^{\inner}_{\iM} \big )^{{k}_1}(z_0)$. Since $\FF^{\outer,i_m}_{\iM}$ is continuous, also $Q_1:=\FF^{\outer,i_m}_{\iM}(Q_0)$ and $\wt{Q}_1:=\FF^{\outer,i_m}_{\iM}(\wt{Q}_0)$ are as close as necessary. Then, applying again the Poincar\'e Recurrence Theorem, there exists $\tilde{j}_1$ such that $\tilde{z}_1=  \big (\FF^{\inner}_{\iM} \big )^{{j}_1}(\tilde{Q}_1)$ and $z_1= \big (\FF^{\inner}_{\iM}\big )^{j_1}(Q_1)$ are close enough. We can repeat this procedure $N$-times to obtain the result.

 
\end{proof}

Lemma \ref{lemma:shadowingfinal} completes the proof of Theorem \ref{thm:main}. Indeed the trajectory of the secular Hamiltonian $\mathtt{H}$ in \eqref{def:fullHamiltonian} with initial condition $p_0$  achieves the drift of energy stated in the theorem. 
It only remains to prove the statements on eccentricity and inclination stated in the   theorem. To this end, it is enough to recall the estimates on the inclination and eccentricity on the periodic and homoclinic orbits of the coplanar Hamiltonian given in Ansatz \ref{ansatz:periodic}. Then, since taking $\iM$ small enough, the shadowing orbits can be taken arbitrarily close to these periodic and homoclinic orbits, we obtain the statements in the theorem.


\subsection{Proof of Theorem \ref{thm:pseudoorbit}}\label{sec:pseudoorbit}
To construct the pseudo-orbit, the first step is to compare the inner and the scattering map obtained in  Theorem~\ref{theorem:PerturbedMaps}.
In order to do so, we apply  two steps of averaging either to the inner map or to one of the scattering maps (in some domains). 
These changes of coordinates straighten the $J$ component of one of the maps up to order $\OO(\iM^3)$ and therefore it becomes straightforward to compare the ``vertical jumps'' of the inner and outer maps.

Since we want to construct symplectic transformations, it is convenient to  straighten first the symplectic form given in~\eqref{def:formaSimplecticaInclined}.

\begin{lemma} 
Assume Ansatz~\ref{ansatz:periodic}. 
There exists an $\iM$-close to the identity change of coordinates $\Upsilon:\tD_{2\de}\times\TT\to\tD_{\de}\times\TT$, $(J,\OmegaM) =\Upsilon(\widecheck{J},\widecheck{\Omega}_{\rM})$,
 which transforms the symplectic form $a(J,\OmegaM) d \OmegaM \wedge d J$ given in~\eqref{def:formaSimplecticaInclined} into $
d \widecheck{\Omega}_{\rM} \wedge d\widecheck{J}$.

In these new coordinates,
\begin{itemize}
    \item The inner map is of the form
        \begin{equation*}
	\widecheck{\FF}^{\inner}_{\iM} \begin{pmatrix}
    \widecheck{J}
  \\ \widecheck{\Omega}_{\rM} 
	\end{pmatrix}
	=
	\begin{pmatrix}
	\widecheck{J} + \iM A_1(\widecheck{J},\widecheck{\Omega}_{\rM}) +
    \iM^2 \widecheck{A}_2(\widecheck{J},\widecheck{\Omega}_{\rM}) + 
    \OO(\iM^3)
  \\ 
  \widecheck{\Omega}_{\rM} + n_{\OmegaM} \left\{\TTT_0(\widecheck{J})
  + \iM \widecheck{\TTT}_1(\widecheck{J},\widecheck{\Omega}_{\rM})
  + \iM^2 \widecheck{\TTT}_2(\widecheck{J},\widecheck{\Omega}_{\rM}) 
  + \OO(\iM^3) \right\}
	\end{pmatrix},
\end{equation*}
where $(\widecheck{J},\widecheck{\Omega}_{\iM} )\in \tD_{2\delta} \times \mathbb{T}$, the function $\TTT_0$ is the one given for the coplanar inner map (see \eqref{def:innerMapCoplanar}), the function $A_1$ is given in Theorem~\ref{theorem:PerturbedMaps} and the functions $\widecheck{A}_2, \widecheck{\TTT}_1$ and $\widecheck{\TTT}_2$ satisfy that
\begin{align*}
	\NNN_{\OmegaM}(\widecheck{A}_2) = \{0, \pm 2\},
	\qquad
	\NNN_{\OmegaM}(\widecheck{\TTT}_1) = \{\pm 1\}, 
    \qquad
    \NNN_{\OmegaM}(\widecheck{\TTT
    }_2) = \{0, \pm 2\}.
\end{align*}
\item For  $*=\{\prim,\secn\}$, the scattering maps are of the form
\begin{equation*} 
	\widecheck{\FF}^{\outer,*}_{\iM} \begin{pmatrix}
    \widecheck{J} 
  \\ \widecheck{\Omega}_{\rM}
	\end{pmatrix}
	=
	\begin{pmatrix}
	\widecheck{J} + \iM B_1^*(\widecheck{J},\widecheck{\Omega}_{\rM})+ \iM^2 \widecheck{B}_2^*(\widecheck{J},\widecheck{\Omega}_{\rM}) + \OO(\iM^3) 
  \\ 
  \widecheck{\Omega}_{\rM} + n_{\OmegaM} \left\{ \zeta^*(\widecheck{J}) + \iM \widecheck{D}_1^{*}(\widecheck{J},\widecheck{\Omega}_{\rM})+ \iM^2 \widecheck{D}_2^*(\widecheck{J},\widecheck{\Omega}_{\rM}) + \OO(\iM^3) \right\}
	\end{pmatrix},
 \qquad 
\end{equation*}
where $(\widecheck{J},\widecheck{\Omega}_{\iM} )\in \tD^*_{2\delta} \times \mathbb{T}$, the functions $\zeta^*(J)$ and $B_1^*$ are given  in Lemma~\ref{lemma:integralsScattering} and~\eqref{theorem:PerturbedMaps} respectively  and the functions $\widecheck{B}_2^*, \widecheck{D}_1^*$ and $\widecheck{D}_2^*$ satisfy that
\begin{align*}
	\NNN_{\OmegaM}(\widecheck{B}_2^*) = \{0, \pm 2\},
	\qquad
	\NNN_{\OmegaM}(\widecheck{D}_1^*) = \{\pm 1\}, 
    \qquad
    \NNN_{\OmegaM}(\widecheck{D}_2^*) = \{0, \pm 2\}.
\end{align*}
\end{itemize}    
\end{lemma}
The proof of this lemma is analogous to that of \cite[Lemma 4.1]{fejoz2016kirkwood}.

Now we perform an averaging procedure. The inner map has a resonance (see Lemma~\ref{corollary:innerMapCoplanar}) in whose neigborhood one cannot perform averaging. Then, we perform two steps of averaging to the inner map in a domain away from the resonance. On the contrary, in a neighborhood of the resonance we perform two steps of averaging to straighten the primary outer map (this is possible thanks to Ansatz~\ref{ansatz:phaseshiftouter} which ensures the absence of low order resonances).
%
%
Taking this into account, we consider the new domains
\begin{equation*} 
\begin{split}
\widehat{\tD}_{2\delta} = \left\{ J \in \tD_{2\delta} \, : \, 
\mathrm{dist}(J,\Jres) > 2\delta \right\}, 
\qquad
\tR_{5\delta}=(\Jres-5\de, \Jres+5\de),
\end{split}
\end{equation*}
where $\Jres$ is the constant introduced in Lemma \ref{corollary:innerMapCoplanar}.
Note that the two domains are chosen so that they overlap (and they will still overlap after applying the averaging changes of coordinates). Note that one can define analogously the domains $\widehat\tD^*_{2\de}$ (see \eqref{def:domainsJPerturbed}), for $*\in \{\prim,\secn\}$.

\begin{lemma}\label{lemma:averaginginner}
Fix $\de>0$ and assume Ansatz~\ref{ansatz:periodic}.
There exists a symplectic change of variables $\iM$-close to the identity $\wt\Upsilon:\widehat\tD_{3\de}\times\TT\to\widehat\tD_{2\de}\times\TT$, $(\widecheck{J},\widecheck{\Omega}_{\rM}) =\wt\Upsilon(\wJ,\wOmegaM)$,
such that, in these new coordinates,
\begin{itemize}
    \item The inner map is transformed into
        \begin{equation*}
	\wt{\FF}^{\inner}_{\iM} \begin{pmatrix}
    \wJ 
  \\ \wOmegaM 
	\end{pmatrix}
	=
	\begin{pmatrix}
	\wJ + \OO(\iM^3)
  \\ 
  \wOmegaM + n_{\OmegaM} \left\{\TTT_0(\wJ)
  + \iM^2 \wt{\TTT}_2(\wJ) 
  + \OO(\iM^3) \right\}
	\end{pmatrix},
\end{equation*}
where $(\wt{J},\wt{\Omega}_{\iM} )\in \widehat{\tD}_{3\delta}  \times \mathbb{T}$ and  $\TTT_0$ is  the function introduced in \eqref{def:innerMapCoplanar}.
\item For $*\in\{\prim,\secn\}$, the scattering maps are transformed to
\begin{equation*} 
	\wt\FF^{\outer,*}_{\iM} \begin{pmatrix}
    \wJ 
  \\ \wOmegaM 
	\end{pmatrix}
	=
	\begin{pmatrix}
	\wJ + \iM \wt{B}_1^*(\wJ,\wOmegaM) + \OO(\iM^2) 
  \\ 
  \wOmegaM + n_{\OmegaM} \left\{ \zeta^*(\wJ)  + \OO(\iM) \right\}
	\end{pmatrix},
\end{equation*}
where $(\wt{J},\wt{\Omega}_{\iM} )\in \widehat{\tD}_{3\delta}^* \times \mathbb{T}$, $\zeta^*$ is given in Lemma~\ref{lemma:integralsScattering} and
$\wt{B}_1^{*}(\wJ,\wOmegaM) = \wt{B}_1^{*,+}(\wJ) e^{i\wOmegaM}+\wt{B}_1^{*,-}(\wJ) e^{-i\wOmegaM}$ with 
\begin{align*}
    \wt{B}_1^{*,\pm} (\wJ) 
    &=
    {B}_1^{*,\pm} (\wJ) 
    -
     A_1^{\pm}(\wt J)
    \frac{e^{\pm i n_{\OmegaM} \zeta^*(\wJ)}-1}{e^{\pm i n_{\OmegaM} \TTT_0(\wJ)}-1}.
\end{align*}
\end{itemize}
\end{lemma}

\begin{lemma}\label{lemma:averagingouter}
Fix $\de>0$.
Assume Ansätze~\ref{ansatz:periodic} and~\ref{ansatz:phaseshiftouter}.
There exists a symplectic change of variables $\iM$-close to the identity $\wh\Upsilon:\tR_{4\de}\times\TT\to\tR_{5\de}\times\TT$, $(\widecheck{J},\widecheck{\Omega}_{\rM}) =\wh\Upsilon(\wh J,\whOmegaM)$,
such that, in these new coordinates,
\begin{itemize}
\item The primary scattering map is transformed to
\begin{equation*} 
	\wh\FF^{\outer,\prim}_{\iM} \begin{pmatrix}
    \wh J 
  \\ \whOmegaM 
	\end{pmatrix}
	=
	\begin{pmatrix}
	\wh J + \OO(\iM^3) 
  \\ 
  \whOmegaM + n_{\OmegaM} \left\{ \zeta^{\prim}(\wh J)  +\iM^2 \wh D_2^{\prim}(\wh J)+\OO(\iM^3) \right\}
	\end{pmatrix},
\end{equation*}
where $(\wh{J},\wh{\Omega}_{\iM} )\in \widehat{\tR}_{4\delta} \times \mathbb{T}$ and  $\zeta^{\prim}$ is given in Lemma~\ref{lemma:integralsScattering}.
    \item The inner map is transformed to
        \begin{equation*}
	\wh{\FF}^{\inner}_{\iM} \begin{pmatrix}
    \wh J 
  \\ \whOmegaM 
	\end{pmatrix}
	=
	\begin{pmatrix}
	\wh J + \iM \widehat{A}_1(\wh J,\whOmegaM) +
      \OO(\iM^2)
  \\ 
  \whOmegaM + n_{\OmegaM} \left\{\TTT_0(\wh J)
  + \OO(\iM) \right\}
	\end{pmatrix},
\end{equation*}
where $(\wh{J},\wh{\Omega}_{\iM} )\in \widehat{\tR}_{4\delta} \times \mathbb{T}$, $\TTT_0$ is  given in \eqref{def:innerMapCoplanar} and $  \wh{A}_1(\wh J,\whOmegaM)= \wh{A}_1^{+}(\wh J) e^{i\whOmegaM} 
    +\wh{A}_1^{-}(\wh J) e^{-i\whOmegaM}$ with
\begin{align*}
    \wh{A}_1^{\pm} (\wh J) 
    &=
    {A}_1^{\pm} (\wh J) 
    -
     B_1^{\prim,\pm}(\wh J)
    \frac{e^{\pm i n_{\OmegaM}\TTT_0(\wh J) }-1}{e^{\pm i n_{\OmegaM} \zeta^{\prim}(\wh J)}-1}.
\end{align*}

\end{itemize}
\end{lemma}

The proofs of these two lemmas are analogous to that of \cite[Lemma 3.9]{fejoz2016kirkwood}.

Now we analyze the KAM curves that these maps possess in each of the regions. We rely on a version of the KAM Theorem from \cite{DelshamsLS00} (see also \cite{Herman83}).

\begin{theorem}\label{thm:KAM}
Let  $f:[0,1]\times\TT\to \RR\times\TT$ be an exact symplectic $\CCC^\ell$ map with $\ell>4$. Assume that $f=f_0+\eps f_1$ where $f_0(I,\psi)=(I,\psi+A(I))$, $A$ is $\CCC^\ell$, $|\pa_IA(I)|>M$ and $\|f_1\|_{\CCC^\ell}\leq 1$. Then, if $\eps^{1/2}M^{-1}=\rho$ is sufficiently small, for a set of Diophantine numbers $\omega$ with exponent $\theta=5/4$, we can find 1-dimensional invariant tori which are graph of $\CCC^{\ell-3}$ functions $u_\omega$, the motion on them is $\CCC^{\ell-3}$ conjugate to the rotation by $\omega$, $\|u_\omega\|_{\CCC^{\ell-3}}\leq\eps^{1/2}$ and these tori cover the whole annulus $[0,1]\times \mathbb{T}$ except for a set of measure of order $M^{-1}\eps^{1/2}$. \end{theorem}

Lemmas \ref{lemma:averaginginner} and \ref{lemma:averagingouter} and Theorem \ref{thm:KAM} imply the following lemma.

\begin{lemma}\label{lemma:transitionchain}
Fix $\de>0$ small. Then, for $\iM>0$ small enough the following is satisfied.
\begin{itemize}
 \item There exists a sequence of tori $\{\TT_{1,k}\}_{k=1}^{N_1}\subset\wt\Lambda_{\iM,\de}$ which are invariant by the map ${\FF}^{\inner}_{\iM}$ and whose dynamics are conjugated to a quasi-periodic rigid rotation such that 
 \begin{equation*}
\FF^{\outer,*}_{\iM}(\TT_{1,k})\pitchfork \TT_{1,k+1}, \qquad k=1\ldots N_1-1,
 \end{equation*}
for either $*=\prim$ or $*=\secn$ and 
\begin{equation}\label{def:firstlasttorus}
\TT_{1,1}\subset [J_{\min}, J_{\min}+2\de]\qquad \text{and} \qquad \TT_{1,N_1}\subset [\Jres-5\de, \Jres-2\de].
\end{equation}
 \item There exists a sequence of tori $\{\TT_{2,k}\}_{k=2}^{N_2}\subset\wt\Lambda_{\iM,\de}$ which are invariant by the map ${\FF}^{\outer,\prim}_{\iM}$ and whose dynamics are conjugated to a quasi-periodic rigid rotation such that 
 \[
\FF^{\inner}_{\iM}(\TT_{2,k})\pitchfork \TT_{2,k+1}, \qquad k=1\ldots N_2-1
 \]
 and 
\[
\TT_{2,1}\subset [\Jres-5\de, \Jres-2\de]\qquad \text{and} \qquad \TT_{2,N_2}\subset [\Jres+2\de, \Jres+5\de].
\]
Moreover,
\begin{equation}\label{def:toriInnerOuter}
\TT_{1,N_1}\pitchfork \TT_{2,1}.
\end{equation}
 \item There exists a sequence of tori $\{\TT_{3,k}\}_{k=1}^{N_3}\subset\wt\Lambda_{\iM,\de}$ which are invariant by the map ${\FF}^{\inner}_{\iM}$ and whose dynamics are conjugated to a quasi-periodic rigid rotation such that 
 \[
\FF^{\outer,*}_{\iM}(\TT_{3,k})\pitchfork \TT_{3,k+1}, \qquad k=1\ldots N_3-1,
 \]
for either $*=\prim$ or $*=\secn$ and 
\[
\TT_{3,1}\subset [\Jres+2\de, \Jres+5\de]\qquad \text{and} \qquad \TT_{3,N_3}\subset [J_{\max}-2\de, J_{\max}].
\]
Moreover,
\[
\TT_{2,N_2}\pitchfork \TT_{3,1}.
\]
\end{itemize}
\end{lemma}

\begin{proof}
Since the proofs of the three statements follow exactly the same lines, we only prove the first one. We prove the statement in the coordinates provided by Lemma~\ref{lemma:averaginginner}. Since they are $\OO(\iM)$-close to the original coordinates, one can easily deduce the statement for the original coordinates from that of the averaging coordinates.

By Lemma \ref{lemma:averaginginner}, the inner map $\wt{\FF}^{\inner}_{\iM}$ is $\OO(\iM^3)$-close to integrable. Moreover, by Lemma \ref{corollary:innerMapCoplanar} the map is twist and the twist has a lower bound independent of $\iM$. Then, Theorem \ref{thm:KAM} implies that there is a sequence of invariant tori $\{\wt\TT_{1,k}\}_{k=1}^{N_1}\subset\wt\Lambda_{\iM,\de}$ with quasi-periodic dynamics which satisfy \eqref{def:firstlasttorus} and 
\begin{equation}\label{disttori}
\mathrm{dist}\left(\wt\TT_{1,k},\wt\TT_{1,k+1}\right)\leq C \iM^{3/2}\qquad k=1,\ldots ,N_1-1.
\end{equation}
for some $C>0$.

Let $u_k$ be such that $\wt \TT_{1,k}$ can be expressed as graph as  $J=u_k(\OmegaM)$. From the expression of $\wt \FF^{\inner}_{\iM}$ in Lemma~\ref{lemma:averaginginner},  $u_k(\OmegaM)=J^k + \mathcal{O}(\iM^3)$  with $J^k$ a constant. Then 
$\wt \FF^{\outer,*}_{\iM} (\wt \TT_{1,k})$ is the graph of 
$$
w_{k+1}^{\outer}(\OmegaM)= J^k + \iM B_1(J^k,\OmegaM-n_{\OmegaM} \zeta^*(J^k)) + \mathcal{O}(\iM^2).
$$
Therefore the intersection 
$\wt\FF^{\outer,*}_{\iM}(\wt\TT_{1,k})\cap \wt\TT_{1,k+1}$
is defined by the points satisfying
$$
J^k + \iM \wt B_1^* \big (J^k,\OmegaM-n_{\OmegaM} \zeta^*(J^k) \big ) = J^{k+1} + \mathcal{O}(\iM^3).
$$
This condition, since by~\eqref{disttori}, $|J^k-J^{k+1}|\leq C\iM^{3/2}$, is equivalent to  
$$
\wt B_1^* \big (J^k,\OmegaM-n_{\OmegaM} \zeta^*(J^k) \big ) = \mathcal{O}(\iM^{1/2}). 
$$
Now, by Ansatz \ref{ansatz:straightening}, the functions ${\wt B}_1^{*,\pm}$ introduced in Lemma \ref{lemma:averaginginner} do not vanish for $J\in \widehat\tD^*_{\de}$ and $B_1^*$ has just harmonics $\pm 1$. Therefore one can easily show that 
$$ 
\wt B_1^*(J,\OmegaM)= |\wt B_1^{*,+}(J)| \cos (\OmegaM + \psi(J)), \qquad \psi= \mathrm{arg} (B_1^{*,+}(J))
$$
and to conclude that $\wt \FF^{\outer,*}_{\iM}(\wt \TT_{1,k})$ and $\wt \TT_{1,k+1}$, for $k=1,\ldots ,N_1-1$, intersects transversally. 

Finally, doing analogous arguments as the previous one, the transversality in \eqref{def:toriInnerOuter} is a direct consequence of Ansatz \ref{ansatz:straightening}.

%
%
%
\end{proof}

Then, the pseudo-orbit given in Theorem~\ref{thm:pseudoorbit} can be easily obtained from the  sequence of tori given by Lemma \ref{lemma:transitionchain}. Indeed, note that the fact the tori are quasi-periodic imply the following. Take points 
\[
P_1\in \FF^{\outer,*}_{\iM}(\TT_{1,k-1})\pitchfork \TT_{1,k}\quad \text{and}\quad P_2\in \FF^{\outer,*}_{\iM}(\TT_{1,k})\pitchfork \TT_{1,k+1}
\]
Then, for any  $\nu>0$ arbitrarily small, there exists $K$ such that 
\[
\left|(\FF^{\inner}_{\iM})^K(P_1)-P_2\right|\leq \nu.
\]
Therefore, to construct the pseudo-orbit it is enough to use that both the inner and outer maps are regular with respect to the parameter $\iM$.


\section{Conclusions}

In this work, we have shown how to model the eccentricity growth for the Galileo constellation by an Arnold diffusion mechanism. To this end, we have considered the full quadrupolar expansion of the lunar gravitational perturbation, coupled with the Earth's oblateness. By assuming that the Moon lies on the ecliptic plane, the dynamical system is autonomous and we can compute numerically the normally hyperbolic invariant manifold stemming from the $2g+h$ resonance and the associated stable and unstable manifolds. Then, we  are able to describe the full dynamics under the assumption that the inclination of the Moon is small enough. Indeed, in this regime, the  cylinder, its dynamics and its invariant manifolds are close to those of the coplanar  one. In other words, the inner map describing the cylinder dynamics and the outer map describing the homoclinic connections to the cylinder are first derived for the coplanar case, and then extended to the full system by means of a perturbative approach, assuming the lunar inclination as a small parameter. Thanks to the existence of the homoclinic connections, we are able to concatenate invariant objects along which the eccentricity increases, on different energy levels. There exist orbits that shadow the sequence of homoclinic orbits. Along these orbits, for $a=29600$ km, the eccentricity can transition from $0$ to $0.78$ and higher, eventually to achieve a re-entry.

The work is based on the idea that the chaotic behavior associated with the homoclinic connection can be exploited to jump from one energy level to the other. Although already proposed in very recent works \cite{Daquin2022,Legnaro2023}, the Arnold diffusion is handled here in a semi-analytical way, considering the full model. 
Possible resonance overlapings, although detected in the numerical computation, are not considered as the main trigger to get to the atmospheric reentry. As a matter of fact, in the normally hyperbolic invariant cylinder the resonance $2h-\OmegaM$ plays a role. This was already mentioned in \cite{Daquin2022}, where they see that, when $a= 29600$ km, both $2g+h$ and $2h-\OmegaM$ resonances interact. However, under the assumption that $i_\rM>0$ is small, the second resonance is weak and, in particular, it does not break up the invariant cylinder that exists along the $2g+h$ resonance. Therefore, one can apply an Arnold-like mechanism to drift along the $2g+h$ resonance even in the presence of a crossing weak resonance.

The procedure developed is general and can be applied to other resonances (e.g., for GLONASS) and other values of semi-major axis, for instance to show where to locate initially the satellites to facilitate eventually the end-of-life phase. The same argument can be applied also to design very stable graveyard solutions, by exploiting the dynamics in such a way that the excursion in eccentricity along the stable and unstable manifolds gets lower and lower (0 in the limit)\footnote{Note that at the resonances one can find stability zones associated to secondary tori and elliptic periodic orbits. Their analysis is beyond the scope of this paper.}. 

The analysis and the tools provided in the work lay the foundations to study the problem with the actual value of $i_\rM$. In this regard, we expect that the mechanism will persist and that the time of diffusion will reduce, but that the second resonance could be significant\footnote{See \cite{Kaloshin:2020} for the analysis of Arnold diffusion along double  (strong) resonances.}.

\appendix

\section{Expression of the Hamiltonian}

This appendix is devoted to compute explicit computations for the perturbative term of the Hamiltonians $\rH_1$ given in \eqref{def:hamiltonianDelaunayH1} and $\HH_1$ in \eqref{def:hamiltonianPoincareH1}.
In Appendix~\ref{appendix:HamiltonianDelaunay}, we obtain an expression for $\rH_1$ in the slow-fast coordinates introduced in Section~\ref{sec:slowfast} and in Appendix~\ref{appendix:HamiltonianPoincare} an expression for $\HH_1$ in the Poincar\'e coordinates introduced in Section~\ref{sec:poincare}.



\subsection{Hamiltonian in slow-fast coordinates $(y,x)$}
\label{appendix:HamiltonianDelaunay}

In this section, we compute explicit expressions for the Hamiltonian $\rH_1$ given in \eqref{def:hamiltonianDelaunayH1}. Let us recall its expression here:
\bes
\begin{split}
	\rH_1&(y , \Gam, x, h, \OmegaM;\iM)
	=	-\frac{\rho_1}{L^2} 
	\sum_{m=0}^2	
	\sum_{p=0}^2 
	D_{m,p}(y,\Gam)\sum_{s=0}^2				
	c_{m,s} F_{2,s,1}(i_\rM)\\
	&\times\Big[
	U_2^{m,-s}(\epsilon)\cos\Big(\psi_{m, p, s}(x,h,\OmegaM)\Big)
	+	U_2^{m,s}(\epsilon)\cos\Big(\psi_{m, p, -s}(x,h,\OmegaM)\Big)
	\Big],
\end{split}
\ees
where $D_{m,p} = \Dt_{m,p} \circ \UDelaunay$ 
and $\psi_{m,p,s}=\psit_{m,p,s} \circ \UDelaunay$ (see \eqref{def:Dtilde}, \eqref{def:PsiTilde} and~\eqref{eq:xy} for the change of coordinates).
In addition, $U_2^{m, \mp s}(\epsilon)$ is the Giacaglia function given in Table~\ref{tab:U}.

Applying the corresponding change of coordinates, one obtains the expressions for $D_{m,p}$ given in Table~\ref{tab:functionsDelaunay}.
\begin{table}[!t]
	\begin{center}
		\begin{tabular}{llll}
			\hline	
			$m$ 	&	$s$		&	$U_2^{m,s}(\epsilon)$ 									&	$\simeq$\\
			\hline 	
			0 		& 	0 		& 	$1-6 C^{2}+6 C^{4}$ 							&	0.762646\\
			0 		& 	-1 		&  	$-2 C S^{-1}\left(2 C^{4}-3 C^{2}+1\right)$		&	0.364961\\ 
			0 		& 	1 		&  $-2 C S\left(1-2 C^{2}\right)$					&	0.364961\\
			0 		& 	-2 		&  $C^{2} S^{-2}\left(C^{2}-1\right)^{2}$			&	0.039558\\
			0 		& 	2 		&  $C^2S^2$											&	0.039558\\
			1 		& 	0 		&  $-3 C S^{-1}\left(2 C^{4}-3 C^{2}+1\right)$	 	&	0.547442\\
			1 		& 	-1 		&  $S^{-2}\left(4 C^{6}-9 C^{4}+6 C^{2}-1\right)$	&	0.116974\\
			1 		& 	1 		&  $C^{2}\left(4 C^{2}-3\right)$					&	0.800502\\
			1 		& 	-2 		&  $-C S^{-3}\left(C^{2}-1\right)^{3}$				&	0.008206\\
			1 		& 	2 		&  $-C^{3} S$										&	-0.190687\\
			2 		& 	0 		&  $6 C^{2} S^{-2}\left(C^{2}-1\right)^{2}$	 		&	0.237353\\
			2 		& 	-1 		&  $-4 C S^{-3}\left(C^{2}-1\right)^{3}$			&	0.032826\\
			2 		& 	1 		&  $-4 C^{3} S^{-1}\left(C^{2}-1\right)$			&	0.762750\\
			2 		& 	-2 		&  $S^{-4}\left(C^{2}-1\right)^{4}$					&	0.001702\\
			2 		& 	2 		&  $C^4$											&	0.919179\\
			\hline
		\end{tabular}
		\caption{The Giacaglia function $U_2^{m, s}(\epsilon)$ for the moon perturbation (see~\cite{G1974}) where $C=\cos\frac{\epsilon}2$ and $S=\sin\frac{\epsilon}2$ and its value for $\epsilon = 23.44\degre$.}
		\label{tab:U}
	\end{center}
\end{table}
 	\begin{table}[h]
	\begin{tabular}{cclc}
		\hline	
		$m$ 	&	$p$		& \multicolumn{1}{c}{$D_{m,p}(y,\Gam)$}	& $\psi_{m,p,0}$	\\
		\hline 	
		0 		& 	0 		&	
		$-\frac{15}{64}(Ly)^{-2}(y - \Gam)(3y + \Gam)\left(L - 2 y\right)\left(L + 2 y\right)$
		& $x-h$
		\\
		0 		& 	1 		&
		$\phantom{-}\frac{1}{32}(Ly)^{-2}( y^2 - 6\Gam y - 3\Gam^2)\left(5L^2	-	12 y^2\right)$
		& $0$
		\\
		0 		& 	2 		& 	
		$-\frac{15}{64}(Ly)^{-2}(y - \Gam)(3y + \Gam)\left(L - 2 y\right)\left(L + 2 y\right)$
		& $-x+h$
		\\
		1 		& 	0 		& 
		$\phantom{-}\frac{15}{32}(Ly)^{-2}\sqrt{(y - \Gam)(3y + \Gam)}\left(3y + \Gam\right)\left(L - 2 y\right)\left(L + 2 y\right)$	
		& $x$	 
		\\
		1 		& 	1 		& 
		$-\frac{3}{16}(Ly)^{-2}\sqrt{(y - \Gam)(3y + \Gam)}(\Gam + y)\left(5L^2	-	12y^2\right)$
		& $h$		 
		\\
		1 		& 	2 		& 	
		$-\frac{15}{32}(Ly)^{-2}\sqrt{(y - \Gam)(3y + \Gam)}\left(y-\Gam\right)\left(L - 2 y\right)\left(L + 2 y\right)$  
		& $-x+2h$
		\\
		2 		& 	0 		& 
		$\phantom{-}\frac{15}{32}(Ly)^{-2}\left(3y +\Gam\right)^{2}\left(L - 2 y\right)\left(L + 2 y\right)$
		& $x+h$
		\\
		2 		& 	1 		& 	
		$\phantom{-}\frac{3}{16}(Ly)^{-2}(y - \Gam)(3y + \Gam)\left(5L^2	-	12y^2\right)$
		& $2h$
		\\
		2 		& 	2 		& 	
		$\phantom{-}\frac{15}{32}(Ly)^{-2}\left(y-\Gam\right)^{2}\left(L - 2 y\right)\left(L + 2 y\right)$
		& $-x+3h$
		\\
		\hline
	\end{tabular}
	\caption{Computation of the functions $(D_{m,p})_{m,p\in\{0,1,2\}}$ and  $(\psi_{m,p, 0})_{m,p\in\{0,1,2\}}$ for the prograde case.}
	\label{tab:functionsDelaunay}
\end{table}
Moreover, 
\begin{align*}
	\psi_{m,p, s}(x,h,\OmegaM) &=	(1 -p)x	
	-	(1 -p -m) h	+	s\left(\OmegaM	-	\frac{\pi}{2}\right)	
	-	\ry_{\modu{s}}\pi.
\end{align*}
with $\ry_{\modu{s}}$ as given in \eqref{def:auxiliaryFunctionsOriginal}.
Then, 
\begin{align*}
	\psi_{m,p, 0}(x,h) &= (1 -p)x -(1 -p -m) h,	\\
	\psi_{m,p, 1}(x,h,\OmegaM) &=	\psi_{m,p, 0}(x,h) + \OmegaM - \pi,	\\
	\psi_{m,p, -1}(x,h,\OmegaM) &=	\psi_{m,p, 0}(x,h) - \OmegaM,	\\
	\psi_{m,p, 2}(x,h,\OmegaM) &=	\psi_{m,p, 0}(x,h) + 2\OmegaM - \pi,	\\
	\psi_{m,p, -2}(x,h,\OmegaM) &=	\psi_{m,p, 0}(x,h) - 2\OmegaM + \pi.
\end{align*}
See Table~\ref{tab:functionsDelaunay} for the values of $\psi_{m, p, 0}(x,h)$.
Moreover, notice that for $s,m\in\{0,1,2\}$ the constants $c_{m,s}$ as defined in~\eqref{def:auxiliaryFunctionsOriginal} do not depend on $s$. Therefore, we denote
\begin{align*}
	\hat{c}_0 := c_{0,s} = \frac12, \qquad
	\hat{c}_1 := c_{1,s} = \frac13, \qquad
	\hat{c}_2 := c_{2,s} = -\frac1{12}.
\end{align*}

Applying all these expressions, one can express the Hamiltonian $\rH_1$ as a series in $\iM$. Indeed, considering the Kaula's inclination functions in~\eqref{eq:KaulaInclinationMoon}, one has that
\begin{align*} 
	F_{2,s,1}(\iM) 
	=	
	\left\{\begin{array}{ll}
		-\frac12+\OO(\iM^2)		&	\mbox{if $s=0$,}\\[0.5em]
		-\frac32 \iM + \OO(\iM^3) 				& 	\mbox{if $s=1$,}\\[0.5em]
		\OO(\iM^2)		& 	\mbox{if $s=2$.}
	\end{array}\right. 
\end{align*}
Therefore, the Hamiltonian $\rH_1$ can be expressed as
\begin{equation}\label{eq:expansionH1Delaunay}
\begin{aligned}
	\rH_{1}(y,\Gam,x,h,\OmegaM;\iM) =& \,
	\rH_{\CP,1}(y,\Gam,x,h) 
	+ \iM \rR(y,\Gam,x,h,\OmegaM;\iM),
	\\
	\rR(y,\Gam,x,h,\OmegaM;\iM) 
	=& \,
	\rR_1(y,\Gam,x,h,\OmegaM) \\
	&+
	\iM \rR_2(y,\Gam,x,h,\OmegaM) 
	+ 
	\OO(\iM^3),
\end{aligned}
\end{equation}
where
\begin{align*}
	\rH_{\CP,1}(y,\Gam, x, h)
	&=	\frac{\rho_1}{L^2}					
	\sum_{m=0}^2
	\sum_{p=0}^2 
	f_m^0	
	D_{m,p}(y,\Gam)
	\cos\Big(\psi_{m, p, 0}(x,h)	\Big),
	\\
	\rR_1(y,\Gam,x,h,\OmegaM) &=
	\cos(\OmegaM) \rR_{\cos,1}(y,\Gam,x,h) 
	+ \sin(\OmegaM) \rR_{\sin,1}(y,\Gam,x,h),
	\\
	\rR_{\cos,1}(y,\Gam,x,h) 
	&=	\frac{3\rho_1}{2 L^2}					
	\sum_{m=0}^2
	\sum_{p=0}^2 
	f_m^{\cos} D_{m,p}(y,\Gam)
	\cos\Big( 	\psi_{m,p,0}(x,h) \Big),
	\\
	\rR_{\sin,1}(y,\Gam,x,h) 
	&=	\frac{3\rho_1}{2 L^2}					
	\sum_{m=0}^2
	\sum_{p=0}^2 
	f_m^{\sin} D_{m,p}(y,\Gam)
	\sin\Big( 	\psi_{m,p,0}(x,h) \Big), 
\end{align*}
with
\begin{align*}
	{f}^0_m = \hat{c}_m U_2^{m,0}, 
	\quad
	{f}_m^{\cos} = \hat{c}_m \left( U_2^{m,1} - U_2^{m,-1} \right),
	\quad
	{f}_m^{\sin} = \hat{c}_m \left( U_2^{m,1} + U_2^{m,-1} \right).
\end{align*}
In addition, the harmonics of $\rR_2$ satisfy that (see Definition~\ref{def:numberharmonics})
\[
\NNN_{\OmegaM}(\rR_2) = \{0,\pm 2\}.
\]

 \subsection{Hamiltonian in Poincar\'e coordinates \texorpdfstring{$(\eta,\xi)$}{(eta,xi)}}
 \label{appendix:HamiltonianPoincare}
In this section, we compute explicit expressions for the Hamiltonian $\HH_1$ given in \eqref{def:hamiltonianPoincareH1}. Let us recall that
 \begin{equation*}
 	\begin{split}
 		\HH_1&(\eta , \Gam, \xi, h, \OmegaM;\iM)
 		=	(\rH_1 \circ \UPoincare) (\eta , \Gam, \xi, h, \OmegaM;\iM),
 	\end{split}
 \end{equation*}
with $\UPoincare$ as given in \eqref{eq:xieta}. 
Following the expression in \eqref{eq:expansionH1Delaunay}, one has that
\begin{equation}\label{eq:expansionH1Poincare}
\begin{aligned}
	\HH_{1}(\eta,\Gam,\xi,h,\OmegaM;\iM) =& \,
	\HH_{\CP,1}(\eta,\Gam,\xi,h) 
	+ \iM \RRR(\eta,\Gam,\xi,h,\OmegaM;\iM) ,
	\\
	\RRR(\eta,\Gam,\xi,h,\OmegaM;\iM) =& \,
	\RRR_1(\eta,\Gam,\xi,h,\OmegaM)
	\\
	&+ 
	\iM\RRR_2(\eta,\Gam,\xi,h,\OmegaM)
	+ 
	\OO(\iM^3),
\end{aligned}
\end{equation}
where $\HH_{\CP,1} = \rH_{\CP,1} \circ \UPoincare$, $\RRR_{1} = \rR_{1} \circ \UPoincare$ and
$\RRR_{2}=\rR_{2}\circ \UPoincare$.
Taking into account that the Poincar\'e change of coordinates satisfy that
\begin{align*}
	y = \frac{2L-\xi^2-\eta^2}4,
	\qquad
	(L-2y)\cos x = \frac{\xi^2-\eta^2}2,
	\qquad
	(L-2y)\sin x= \xi \eta,
\end{align*}
one has that
\begin{equation}\label{eq:expressionPoincareHCP1}
\begin{aligned}
	\HH_{\CP, 1}
	= \,	\frac{\rho_1}{L^2}					
	&\sum_{m=0}^2 \Big[
	f_m^0	\DDD_{m,0}(\eta,\Gam,\xi) \left(
	\frac{\xi^2-\eta^2}2 \cos\big((1-m)h\big)
	+
	\xi\eta \sin\big((1-m)h\big) \right)
	\\
	&+ f_m^0	
	\DDD_{m,1}(\eta,\Gam,\xi)
	\left(8L^2+12 L (\xi^2+\eta^2)-3(\xi^2+\eta^2)^2\right)
	\cos(mh)
	\\
	&+ f_m^0	\DDD_{m,2}(\eta,\Gam,\xi) \left(
	\frac{\xi^2-\eta^2}2 \cos\big((1+m)h\big)
	+
	\xi\eta \sin\big((1+m)h\big) \right)
	\Big],
 \end{aligned}
\end{equation}
where the functions $(\DDD_{m,p})_{m,p\in\{0,1,2\}}$ are given in Table~\ref{tab:functionsPoincare}.
From this explicit expression, one can easily see that the coplanar Hamiltonian $\HH_{\CP}=\HH_0+\al^3\HH_{\CP,1}$ (see~\eqref{def:hamiltonianPoincareH0} and~\eqref{eq:expressionPoincareHCP1}) is quadratic with respect to $(\eta,\xi)=(0,0)$.
This implies the following lemma.

\begin{lemma}
	\label{lemma:periodicOrbit}
	The Hamiltonian system given by $\HH_{\CP}(\eta,\Gam,\xi,h)$ in~\eqref{def:Ham:Coplanar} has orbits of the form $(\eta,\Gam,\xi,h)=(0,\Gam(t),0,h(t))$ satisfying that 
	\begin{align*}
		\dot{h} = \partial_{\Gam} \HH_{0}(0,\Gam,0) 
		+ \al^3 \partial_{\Gam} \HH_{\CP, 1}(0,\Gam,0,h),
		\qquad
		\dot{\Gam} = - \al^3 \pa_h\HH_{\CP,1}(0,\Gam,0,h),
	\end{align*}
	where
	\begin{align*}
		\partial_{\Gam} \HH_0(0,\Gam,0) 
		=
		-\frac{3\rho_0}{4 L^8}&(L+2\Gam),
		\\
		\partial_{\Gam} {\HH}_{\CP,1}(0,\Gam,0,h) 
		= -\frac{\rho_1}{8 L^4} &\Big(
		3 U_2^{0,0}(L + 2\Gam)
		+ {4 U_2^{1,0}}
		\frac{L^2-4L\Gam-4\Gam^2}{\sqrt{(L-2\Gam)(3L+2\Gam)}} \cos h
		\\
		&-{U_2^{2,0}}
		(L+2\Gam)\cos(2h) \Big),
		\\[0.4em]
		\partial_h {\HH}_{\CP,1}(0,\Gam,0,h)
		=
		\frac{\rho_1}{16 L^4} &\Big(
		2 U_2^{1,0}
		\sqrt{(L-2\Gam)(3L+2\Gam)}(2\Gam+L) 
		\sin h \\
		&+U_2^{2,0} (L-2\Gam)(3L+2\Gam)\sin(2h) \Big).
	\end{align*}
\end{lemma}
Analogously, one can proceed for $\RRR_1$, which can be written as
\begin{equation} \label{def:RRR1appendix}
		\RRR_1(\eta,\Gam,\xi,h,\OmegaM) =
		e^{i\OmegaM} \RRR_{1}^+(\eta,\Gam,\xi,h) 
		+ e^{-i\OmegaM} \RRR_{1}^-(\eta,\Gam,\xi,h),
\end{equation}
where
\begin{equation*}
\begin{aligned}
	\RRR_{1}^{\pm}
	= \,	\frac{3\rho_1}{2L^2}					
	&\sum_{m=0}^2 \Big[
	f^{\pm}_m	\DDD_{m,0}(\eta,\Gam,\xi) \left(
	\frac{\xi^2-\eta^2}2 \cos\big((1-m)h\big)
	+
	\xi\eta \sin\big((1-m)h\big) \right)
	\\
	&+ f^{\pm}_m	
	\DDD_{m,1}(\eta,\Gam,\xi)
	\left(8L^2+12 L (\xi^2+\eta^2)-3(\xi^2+\eta^2)^2\right)
	\cos(mh)
	\\
	&+ f^{\pm}_m	\DDD_{m,2}(\eta,\Gam,\xi) \left(
	\frac{\xi^2-\eta^2}2 \cos\big((1+m)h\big)
	+
	\xi\eta \sin\big((1+m)h\big) \right)
	\Big],
 \end{aligned}
\end{equation*}
with
\begin{align*}
{f}_m^{\pm} &= \frac{\hat{c}_m}2 \left( (1\mp i)U_2^{m,1} - (1\pm i)U_2^{m,-1} \right).
\end{align*}
In addition, the harmonics of $\RRR_2$ satisfy that
\begin{equation}\label{def:harmonicsRRR2Appendix}
  \NNN_{\OmegaM}(\RRR_2) = \{0, \pm 2\}.  
\end{equation}

\begin{table}[h]
\begin{equation*}
\def\arraystretch{1.3}
\begin{array}{ccl}
\hline
m & p & \multicolumn{1}{c}{\DDD_{m,p}(\eta,\Gam,\xi)} \\
\hline 
0 		& 	0 		&	
-\frac{15}{128} {L^{-2} (2L-M)^{-2}} 
{(2L-M-4\Gam)(6L-3M+4\Gam)(4L-M)}
\\
\hline
0 		& 	1 		&
\phantom{-}\frac1{128}
{L^{-2} (2L-M)^{-2}}((2L-M)^2-24(2L-M)\Gam -48\Gam^2)
\\
\hline
0 		& 	2 		& 	
-\frac{15}{128}
{L^{-2} (2L-M)^{-2}}
{(2L-M-4\Gam)(6L-3M+4\Gam)(4L-M) }
\\
\hline
1 & 0 &
\phantom{-}\frac{15}{64}
{L^{-2} (2L-M)^{-2}}
{\sqrt{(2L-M-4\Gam)}(6L-3M+4\Gam)^{3/2}(4L-M)}
\\
\hline
1 & 1 &
-\frac{3}{64}
{L^{-2} (2L-M)^{-2}}
{\sqrt{(2L-M-4\Gam)(6L-3M+4\Gam)}(2L-M+4\Gam)} 
\\
\hline
1 & 2 &
-\frac{15}{64}
{L^{-2} (2L-M)^{-2}}
{(2L-M-4\Gam)^{3/2}\sqrt{6L-3M+4\Gam}(4L-M)}
\\
\hline
2 & 0 &
\phantom{-}\frac{15}{64}
{L^{-2} (2L-M)^{-2}}
{(6L-3M+4\Gam)^{2}(4L-M)}
\\
\hline
2 & 1 &
\phantom{-}\frac{3}{64}
{L^{-2} (2L-M)^{-2}}
{(2L-M-4\Gam)(6L-3M+4\Gam)}
\\
\hline
2 & 2 &
\phantom{-}\frac{15}{64}
{L^{-2} (2L-M)^{-2}}
{(2L-M-4\Gam)^{2}(4L-M) }
\\
\hline
\end{array}
\end{equation*}
	\caption{Computation of the functions $(\DDD_{m,p})_{m,p\in\{0,1,2\}}$ with $M:=\xi^2+\eta^2$.}
\label{tab:functionsPoincare}
\end{table}

{\small
\bibliographystyle{abbrv}
\bibliography{references} 
}

\end{document}